\documentclass[12pt]{amsart}
\usepackage{amsfonts}
\usepackage{amsmath}
\usepackage{amssymb}
\usepackage{mathrsfs}
\usepackage{geometry}
\usepackage{graphicx}
\usepackage{subfigure}
\usepackage{hyperref}
\usepackage{microtype}
\usepackage{enumerate}
\usepackage{geometry}
\usepackage{tikz}
\usepackage{tikz-cd}
\usepackage{color}
\numberwithin{equation}{section}
\allowdisplaybreaks[1]
\usepackage{titletoc}
\usepackage{amsthm}
\usepackage{cite}
\newcommand{\la}{\lambda}
\newcommand{\al}{\alpha}
\newcommand{\be}{\beta}
\newcommand{\ga}{\gamma}

\newcommand{\vp}{\varphi}

\newcommand{\R}{\mathbb{R}}
\newcommand{\N}{\mathbb{N}}

\newcommand{\om}{\omega}

\newcommand{\n}[1]{\Vert #1\Vert }
\newcommand{\bn}[1]{\big \Vert #1 \big \Vert }
\newcommand{\bbn}[1]{\Big\Vert #1 \Big \Vert }
\newcommand{\lr}[1]{\left\{ #1 \right\} }
\newcommand{\lrc}[1]{\left[ #1 \right] }
\newcommand{\lrs}[1]{\left( #1 \right) }

\newcommand{\lra}[1]{\langle #1 \rangle }

\newcommand{\babs}[1]{\big | #1 \big|}

\newcommand{\wt}[1]{\widetilde{#1}}

\newcommand{\pa}{\partial}

\newcommand{\cf}{{\mathcal F}}

\begin{document}

\newtheorem{theorem}{Theorem}[section]
\newtheorem{lemma}[theorem]{Lemma}

\theoremstyle{definition}
\newtheorem{definition}[theorem]{Definition}
\newtheorem{example}[theorem]{Example}
\newtheorem{remark}[theorem]{Remark}

\numberwithin{equation}{section}

\newtheorem{proposition}[theorem]{Proposition}
\newtheorem{corollary}[theorem]{Corollary}
\newtheorem{goal}[theorem]{Goal}
\newtheorem{algorithm}{Algorithm}

\renewcommand{\figurename}{Fig.}

\title[Well/Ill-posedness of Boltzmann with Soft Potential]{Well/Ill-posedness of the Boltzmann Equation with Soft Potential}

\author[X. Chen]{Xuwen Chen}
\address{Department of Mathematics, University of Rochester, Rochester, NY 14627, USA}
\email{xuwenmath@gmail.com}

\author[S. Shen]{Shunlin Shen}
\address{School of Mathematical Sciences, Peking University, Beijing, 100871, China}
\email{slshen@pku.edu.cn}

\author[Z. Zhang]{Zhifei Zhang}
\address{School of Mathematical Sciences, Peking University, Beijing, 100871, China}

\email{zfzhang@math.pku.edu.cn}

\begin{abstract}
We consider the Boltzmann equation with the soft potential and angular cutoff. Inspired by the methods from dispersive PDEs, we establish its sharp local well-posedness and ill-posedness in $H^{s}$ Sobolev space. We find the well/ill-posedness separation at regularity $s=\frac{d-1}{2}$, strictly $\frac{1}{2}$-derivative higher than the scaling-invariant index $s=\frac{d-2}{2}$, the usually expected separation point.
 \end{abstract}
\maketitle
\tableofcontents


\section{Introduction}

We consider the Boltzmann equation
\begin{equation}\label{equ:Boltzmann}
\left\{
\begin{aligned}
\left( \partial_t + v \cdot \nabla_x \right) f (t,x,v) =& Q(f,f),\\
f(0,x,v)=& f_{0}(x,v),
\end{aligned}
\right.
\end{equation}
where $f(t,x,v)$ is the distribution function for the particles
at time $t\geq 0$, position $x\in \R^{d}$ and velocity $v\in \R^{d}$. The collision operator $Q$ is conventionally split into a gain term
and a loss term
\begin{equation*}
Q(f,g)=Q^{+}(f,g)-Q^{-}(f,g)
\end{equation*}
where the gain term is
\begin{align}
Q^{+}(f,g)=&\int_{\mathbb{R}^{d}}\int_{\mathbb{S}^{d-1}} f(v^{\ast })g(u^{\ast}) B(u-v,\omega)dud\omega,
\end{align}
and the loss term is
\begin{align}
Q^{-}(f,g)=& \int_{\mathbb{R}^{d}}\int_{\mathbb{S}^{d-1}} f(v)g(u) B(u-v,\omega) du d\omega,
\end{align}
with the relation
between the pre-collision and after-collision velocities that
\begin{align*}
u^{*}=u+\omega\cdot (v-u) \omega,\quad v^{*}=v-\omega\cdot(v-u)\omega.
\end{align*}
The Boltzmann collision kernel function $B(u-v,\omega)$ is a non-negative function depending only on the relative velocity $|u-v|$ and the deviation angle $\theta$ through $\cos \theta:=\frac{u-v}{|u-v|}\cdot \omega$. Throughout the paper, we consider
\begin{align}\label{equ:kernel function}
B(u-v,\omega)=|u-v|^{\ga}\textbf{b}(\cos \theta)
\end{align}
 under the Grad's angular cutoff assumption
\begin{align*}
0\leq  \textbf{b}(\cos \theta)\leq C|\cos\theta|.
\end{align*}
 The different ranges $\ga<0$, $\ga=0$, $\ga>0$ correspond to soft potentials, Maxwellian molecules, and hard potentials, respectively. See also \cite{cercignani1994mathematical,cercignani1988boltzmann,villani2002review} for a more detailed physics background. This collision kernel \eqref{equ:kernel function} comes from an important model case of inverse-power law potentials and there have been a large amount of literature devoted to various problems for this model, such as its hydrodynamics limits which provide a description between the kinetic theory and hydrodynamic equations.
 For a detailed presentation and the derivation of macroscopic equations from the fundamental laws of physics, see for example \cite{saint2009hydrodynamcis}.

The Cauchy problem for the Boltzmann equation is one of the fundamental mathematical problems in kinetic theory, as it is of vital importance for the physical interpretation and practical
application.
For instance, in the absence of uniqueness or continuous dependence on the initial condition, numerical calculations and algorithms, even if they can be done, could present puzzling results.
Despite
the innovative work \cite{diperna1989cauchy,gressman2011global} and many nice developments,
the well/ill-posedness of the Boltzmann equation remains largely open. So far, there have been many developed methods and techniques for well-posedness, see for example \cite{alexandre2011global,alexandre2013local,ampatzoglou2022global,arsenio2011global,duan2017global,duan2016global,
guo2003classical,he2017well,he2023cauchy,sohinger2014boltzmann,toscani1988global}.

In the recent series of paper \cite{chen2019local,chen2019moments,chen2021small}, by taking dispersive techniques on the study of the quantum many-body hierarchy dynamics, especially space-time collapsing/multi-linear estimates techniques (see for instance \cite{chen2015unconditional,
chen2014derivation,chen2013rigorous,chen2016focusing,chen2016collapsing,chen2016klainerman,chen2016correlation,
chen2019derivation,chen2022quantitative,chen2023derivation,herr2016gross,herr2019unconditional,
kirkpatrick2011derivation,klainerman2008uniqueness,sohinger2015rigorous}), T. Chen, Denlinger, and Pavlovi$\acute{c}$ provided a new approach to prove the well-posedness of the Boltzmann equation and suggested the possibility of a systematic study of Boltzmann equation using dispersive tools.
With the dispersive techniques, the regularity index for well-posedness, which is usually at least the continuity threshold $s>\frac{d}{2}$, has been relaxed to $s>\frac{d-1}{2}$ for both Maxwellian molecules and hard potentials with cutoff in \cite{chen2019local}. It is of mathematical and physical interest to prove well-posedness at the optimal regularity.
From the scaling point of view,
 the Boltzmann equation \eqref{equ:Boltzmann} is invariant under the scaling
\begin{align}
f_{\la}(t,x,v)=\la^{\al+(d-1+\ga)\be}f(\la^{\al-\be}t,\la^{\al}x,\la^{\be}v),
\end{align}
for any $\al$, $\be\in \R$ and $\la>0$. Then in the $L^{2}$ setting, it holds that
\begin{align*}
\n{|\nabla_{x}|^{s}|v|^{r}f_{\la}}_{L_{xv}^{2}}=\la^{^{\al+(d-1+\ga)\be}}\la^{\al s-\be r}\la^{-\frac{d}{2}\al-\frac{d}{2}\be}
\n{|\nabla_{x}|^{s}|v|^{r}f}_{L_{xv}^{2}}.
\end{align*}
This gives the scaling-critical index
\begin{align}\label{equ:scaling,L2}
s=\frac{d-2}{2},\quad  r=s+\ga.
\end{align}
From the past experience of scaling analysis,
it is believed that the well/ill-posedness threshold\footnote{Instead of scaling invariance of equation, the critical regularity for the Boltzmann equation is sometimes believed at $s=\frac{d}{2}$ in the sense that the critical embedding $H^{\frac{d}{2}}\hookrightarrow L^{\infty}$ fails, see for example \cite{alexandre2013local,duan2016global,duan2021global,duan2018solution}.} in $H^{s}$ Sobolev space is $s_{c}=\frac{d-2}{2}$ with $r\geq 0$.
 Surprisingly, for the 3D constant kernel case, X. Chen and Holmer in \cite{chen2022well} prove the well/ill-posedness threshold in $H^{s}$ Sobolev space is exactly at regularity $s=\frac{d-1}{2}$, and thus point out the actual optimal regularity for the global well-posedness problem.

On the one hand, while there are many well-known progress such as \cite{christ2003asymptotics,christ2003ill-posedness,kenig1996bilinear,kenig1996quadratic,kenig2001ill-posedness,molinet2001ill-posedness,
molinet2002well-posedness,tzvetkov2006ill-posedness} regarding the study of dispersive equations, the illposedness of the Boltzmann equation remains largely open away from \cite{chen2022well}. One certainly would like to have the sharp problem resolved for the Boltzmann equations. On the other hand,
To initiate a systematic study of a large project including sharp well-posedness, blow-up analysis, regularity criteria, etc, it is of priority to find out the well/ill-posedness separation point.
In the paper, moving forward from the special case \cite{chen2022well}, we investigate the general kernel with soft potentials, for which both the sharp well-posedness and ill-posedeness are open. We settle this problem and provide the well/ill-posedness threshold. With the finding of this optimal regularity index, we deal with the sharp small data global well-posedness in another paper \cite{chen2023sharp}.\footnote{The hard potential case is also interesting and the ill-posedness result remains open. Our approximation solution gives desired bad behaviors for the hard potential. But it needs a totally different work space to generate the exact solution. Hence, we put it for further work.}

We start with the connection between the analysis of \eqref{equ:Boltzmann} and the theory of nonlinear dispersive PDEs.
Let $\wt{f}(t,x,\xi)$ be the inverse Fourier transform in the velocity variable, that is,
\begin{align}
\wt{f}(t,x,\xi)=\mathcal{F}_{v\mapsto \xi}^{-1}(f).
\end{align}
Then the linear part of \eqref{equ:Boltzmann} is changed into the symmetric hyperbolic Schr\"{o}dinger equation
\begin{align}
i\pa_{t}\wt{f}+\nabla_{\xi}\cdot \nabla_{x}\wt{f}=0,
\end{align}
which, in the nonlinear context, enables the application of Strichartz estimates that
\begin{align}
\n{e^{it\nabla_{\xi}\cdot \nabla_{x}}\wt{f}_{0}}_{L_{t}^{q}L_{x\xi}^{p}}\lesssim \n{\wt{f}_{0}}_{L_{x\xi}^{2}},\quad \frac{2}{q}+\frac{2d}{p}=d,\quad q\geq 2,\ d\geq2.
\end{align}

We introduce the Sobolev norms
\begin{align}
\n{\wt{f}}_{H_{x}^{s}H_{\xi}^{r}}=\n{\lra{\nabla_{x}}^{s}\lra{\nabla_{\xi}}^{r}\wt{f}}_{L_{x\xi}^{2}}=\n{\lra{\nabla_{x}}^{s}\lra{v}^{r}f}_{L_{xv}^{2}}
=\n{f}_{L_{v}^{2,r}H_{x}^{s}},
\end{align}
and the Fourier restriction norms (see \cite{beals1983self,bourgain1993fourier1,bourgain1993fourier2,klainerman1993space,rauch1982nonlinear})
\begin{align}\label{equ:fourier restriction norm}
\n{\wt{f}}_{X^{s,r,b}}=\n{\hat{f}(\tau,\eta,v)\lra{\tau+\eta\cdot v}^{b}\lra{\eta}^{s}\lra{v}^{r}}_{L_{\tau,\eta,v}^{2}},
\end{align}
where $\hat{f}(\tau,\eta,v)$ denotes the Fourier transform of $\wt{f}(t,x,\xi)$ in $(t,x,\xi)\mapsto (\tau,\eta,v)$, and is thus the Fourier transform of
$f(t,x,v)$ itself in only $(t,x)\mapsto (\tau,\eta)$, that is,
$$\hat{f}(\tau,\eta,v)=\mathcal{F}(\wt{f})=\mathcal{F}_{(t,x)\mapsto(\tau,\eta)}\lrs{f}.$$
It is customary to define their finite time restrictions via
\begin{align}\label{equ:fourier restriction norm,finite time}
\n{\wt{f}}_{X_{T}^{s,r,b}}=\inf\lr{\n{F}_{X^{s,r,b}}:F|_{[-T,T]}=\wt{f}}.
\end{align}

We recall the definition of well-posedness, see for example \cite{mihaela2023local,tao2006nonlinear}.

\begin{definition}
We say that (\ref{equ:Boltzmann}) is well-posed in  $L_{v}^{2,r}H_{x}^{s}$ if for each $R>0$, there exists a time $T=T(R)>0$, and a set $X$, such that all of the following are satisfied.
\end{definition}
\begin{enumerate}
\item[(a)] (Existence and Uniqueness) For each $f_{0}\in L_{v}^{2,r}H_{x}^{s}$ with $\left\Vert f_{0}\right\Vert _{L_{v}^{2,r}H_{x}^{s}}\leqslant R$,
 there exists a unique solution $f(t,x,v)$ to the integral equation of \eqref{equ:Boltzmann} in
 \begin{align*}
C([-T,T];L_{v}^{2,r}H_{x}^{s})\bigcap X.
\end{align*}
Moreover, $f(t,x,v)\geqslant 0$ if $f_{0}\geqslant 0$.\footnote{If $f_{0}\in L_{x,v}^{1}$, the solution $f(t)$ should also have the $L_{x,v}^{1}$ integrability in terms of the mass conservation law. However, this is not a simple problem. We deal with it in \cite{chen2023sharp} by using regularity criteria which are beyond the scope of this paper.}

\item[(b)] (Uniform Continuity of the Solution Map)\footnote{
One could replace (c) with the Lipschitz continuity which is usually the case as well.}
The map $f_{0}\mapsto f$ is uniform continuous with the $C([-T,T];L_{v}^{2,r}H_{x}^{s})$ norm. Specifically,
 suppose $f$ and $g$ are two solutions to (\ref{equ:Boltzmann}) on $[-T,T],$ $\forall
\varepsilon >0$, $\exists\ \delta (\varepsilon )$ independent of $f$ or $g$
such that
\begin{equation}
\left\Vert f(t)-g(t)\right\Vert _{C([-T,T];{L_{v}^{2,r}H_{x}^{s}}
)}<\varepsilon \text{ provided that }\left\Vert f(0)-g(0)\right\Vert
_{L_{v}^{2,r}H_{x}^{s}}<\delta (\varepsilon ) . \label{eq:uniform continuity}
\end{equation}
\end{enumerate}

We take $X$ to be the Fourier restriction norm space $X_{T}^{s,s+\ga,b}$ defined by \eqref{equ:fourier restriction norm,finite time} with $b\in (\frac{1}{2},1)$.

\begin{theorem}[Main Theorem]\label{thm:main theorem}
Let $d=2,3$.

\begin{enumerate}
\item[$(1)$] For $s>\frac{d-1}{2}$, $\frac{1-d}{2}\leq \ga\leq 0$, \eqref{equ:Boltzmann} is locally well-posed in $L_{v}^{2,s+\ga}H_{x}^{s}$.

\item[$(2)$] For $0 \leq  s_{0}<\frac{d-1}{2}$, $\frac{1-d}{2}\leq \ga\leq 0$, $r_{0}=\max\lr{0,s_{0}+\ga}$,
 \eqref{equ:Boltzmann} is ill-posed in $L_{v}^{2,r_{0}}H_{x}^{s_{0}}$
 in the sense that
 the data-to-solution map is not uniformly continuous.
In particular, for each $M\gg1 $, there exists a time sequence
$\left\{t_{0}^{M}\right\} _{M}$ such that
	$$t_{0}^{M}<0, \quad t_{0}^{M}\nearrow 0,$$
and two solutions $f^{M}(t)$,  $g^{M}(t)$ in $[t_{0}^{M},0]$ with
\begin{align*}
\n{f^{M}(t_{0}^{M})}_{L_{v}^{2,r_{0}}H_{x}^{s_{0}}}\sim \n{g^{M}(t_{0}^{M})}_{L_{v}^{2,r_{0}}H_{x}^{s_{0}}} \sim 1,
\end{align*}
such that they are initially close at $t=t_{0}^{M}$
	\begin{equation*}
		\Vert f^{M}(t_{0}^{M})-g^{M}(t_{0}^{M})\Vert
		_{L_{v}^{2,r_{0}}H_{x}^{s_{0}}}\leqslant \frac{1}{\ln \ln M}\ll 1,
	\end{equation*}
	but become fully separated at $t=0$
	\begin{equation*}
		\Vert f^{M}(0)-g^{M}(0)\Vert _{L_{v}^{2,r_{0}}H_{x}^{s_{0}}}\sim 1.
	\end{equation*}

\end{enumerate}

\end{theorem}

Theorem \ref{thm:main theorem} is the main novelty, which finds the well/ill-posedness threshold, by establishing the sharp local well-posedness, and proving the ill-posedness for the soft potential case. There have been many nice work on the well-posedness part by the energy method which requires higher regularity, see for example \cite{guo2003classical,guo2003the,strain2008exponential,alexandre2013local}.
For both Maxwellian molecules and hard potentials, the regularity index $s>\frac{d-1}{2}$ for well-posedness was achieved in \cite{chen2019local} without ill-posedness.
Our well-posedness result solves the remaining soft potential case.

We remark that, as scaling \eqref{equ:scaling,L2} in $L^{2}$ setting gives the restriction that $s+\ga\geq 0$, the range $\frac{1-d}{2}\leq \ga\leq 0$ should be sharp  if one seeks the optimal regularity $s>\frac{d-1}{2}$. In addition, the endpoint case $\ga=-1$ with $d=3$ plays an important role in the derivation of the Boltzmann equation from quantum many-body dynamics in \cite{chen2023derivationboltzmann}, where the collision kernel is composed of part hard sphere and part inverse power potential:
\begin{align}\label{equ:hard,soft,kernel}
B(u-v,\omega)=\lrs{1_{\lr{|u-v|\leq 1}}|u-v|+1_{\lr{|u-v|\geq 1}}|u-v|^{-1}} \textbf{b}(\frac{u-v}{|u-v|}\cdot \omega),
\end{align}
which also provides yet another physical background to our problem here.
Our proof for ill-posedness also works for kernel \eqref{equ:hard,soft,kernel}.

\begin{corollary}\label{lemma:ill-posedness,kernels}
For $d=3$, $0 \leq s_{0}<1$, \eqref{equ:Boltzmann} is ill-posed in $L_{v}^{2}H_{x}^{s_{0}}$ with the kernel \eqref{equ:hard,soft,kernel} in the sense that
 the data-to-solution map is not uniformly continuous.
\end{corollary}

\subsection{Outline of the Paper}
In Section \ref{section:Well-posedness}, we prove the well-posedeness of \eqref{equ:Boltzmann}. The bilinear estimates for gain/loss terms are the key step to conclude the well-posedness and the proof highly relies on the techniques from dispersive PDEs.

In Section \ref{section:Bilinear Estimate for Loss Term}, we appeal to dispersive estimates to prove the loss term bilinear estimate. This can be directly handled because of the factorization of the kernel.
In Section \ref{section:Bilinear Estimate for Gain Term}, we deal with the gain term, which requires a more subtle analysis due to the complicated partial
convolution structure.
One important observation is that the energy conservation provides a lower bound estimate for after-collision velocities, which enables the application of the Littlewood-Paley theory and frequency analysis techniques in multi-linear estimates.
Then with a convolution type estimate in \cite{alonso2010estimates}, we are able to establish the gain term bilinear estimate with the help of Strichartz estimates in the Fourier restriction norm space.
 Finally, in Section \ref{section:Proof of Well-posedness}, we complete the proof of well-posedness after our built-up $X^{s,r,b}$ spaces and its related frequency analysis in this context.

In Section \ref{section:Ill-posedness}, we prove the ill-posedness of \eqref{equ:Boltzmann}. The idea is to construct an approximation solution which has the norm deflation property and then perturb it into an exact solution. We improvise and sharpen the prototype approximation solution found in \cite{chen2022well}. To overcome the singularities carried by the soft potentials, which were known to be the main difficulties, we introduce a new scaling on the approximation solution, create an elaborate $Z$-norm, which is used to prove a closed estimate for the gain and loss terms, that is,
\begin{equation}\label{equ:closed estimate,z norm}
\|Q^\pm (f_1,f_2) \|_{Z} \lesssim \|f_1\|_{Z} \|f_2\|_{Z},
\end{equation}
and conclude the existence of small corrections.
 With this new treatment, the extra restriction that $s_{0}>\frac{1}{2}$ in \cite{chen2022well} can now be removed.

 In Section \ref{equ:Bounds on the approximation solution}, we first construct the approximation solution $f_{\textrm{a}}$ and prove its norm deflation.
Then in Section \ref{section:Discussion on the $L^{1}$-based space and hard potentials}, we give a discussion on the $L^{1}$-based spaces and the hard potentials case, for which our approximation solution also gives desired bad behaviors.\footnote{It then provides a formal answer to a question raised by K. Nakanishi.} Therefore, a similar mechanism of norm deflation in different settings is possible and deserves further investigations.

In Sections \ref{section:$Z$-norm Bounds on the Approximation Solution}--\ref{section:Bounds on the Correction Term}, we introduce the $Z$-norm space and perform a perturbation argument to turn the approximation solution into the exact solution. In Section \ref{section:$Z$-norm Bounds on the Approximation Solution}, we first provide the $Z$-norm bounds on the approximation solution. Then in Section \ref{section:Bounds on the Error Terms}, we deal with the error terms and prove the $Z$-norm error estimates. Proving the error estimates, as it includes a large quantity of error terms involving singularities at which we need geometric techniques on the nonlinear interactions between frequencies, is the most intricate part which we treat in Sections \ref{section:Analysis of term1}--\ref{section:Analysis of Q+},.
After dealing with the error terms,
 we prove that there is an exact solution which is mostly $f_{\textrm{a}}$ in Section \ref{section:Bounds on the Correction Term},
and thus conclude the ill-posedness result in Section \ref{section:Proof of Illposedness}.

After the proof of the main theorem, we put and review some tools in Appendix \ref{section:Sobolev-type and Time-independent Bilinear Estimates} and the Strichartz estimates in Appendix \ref{section:Strichartz Estimates}.~\\

\noindent \textbf{Acknowledgements} The authors would like to thank professors Yan Guo and Tong Yang for helpful discussions about this work. X. Chen was supported in part by NSF grant DMS-2005469 and a Simons fellowship numbered 916862, S. Shen was supported in part by the Postdoctoral Science Foundation of China under Grant 2022M720263, and Z. Zhang was supported in part by NSF of China under Grant 12171010 and 12288101.
\section{Well-posedness}\label{section:Well-posedness}

To conclude the well-posedeness of \eqref{equ:Boltzmann}, it suffices to prove the following bilinear estimates
\begin{align}\label{equ:lwp,bilinear estimate}
\n{\lra{\nabla_{x}}^{s}\lra{v}^{s+\ga}Q^{\pm}(f,g)}_{L_{t}^{2}L_{x,v}^{2}}\lesssim \n{\wt{f}}_{X^{s,s+\ga,b}}\n{\wt{g}}_{X^{s,s+\ga,b}}.
\end{align}
Note that no $v$-variable Fourier transform of the collision kernel in \eqref{equ:lwp,bilinear estimate} is needed if we fully work in the $X^{s,s+\ga,b}$ space. Here, we will work on the $(x,\xi)$ side and prove \eqref{equ:lwp,bilinear estimate} by use of the Fourier transform of the kernel.

Taking the inverse $v$-variable Fourier transform on both side of \eqref{equ:Boltzmann}, we get
\begin{align}
i\pa_{t}\wt{f}+\nabla_{\xi}\cdot \nabla_{x}\wt{f}=i\mathcal{F}_{v\mapsto \xi}^{-1}\lrc{Q(f,f)}.
\end{align}
By the well-known Bobylev identity in a more general case, see for example \cite{alexandre2000entropy,desvillettes2003use}, it holds that (up to an unimportant constant)
\begin{align}
\mathcal{F}_{v\mapsto \xi}^{-1}\lrc{Q^{-}(f,g)}(\xi)=& \n{\textbf{b}}_{L^{1}(\mathbb{S}^{d-1})}\int \frac{\wt{f}(\xi-\eta) \wt{g}(\eta)}{|\eta|^{d+\ga}}d\eta,\\
\mathcal{F}_{v\mapsto \xi}^{-1}\lrc{Q^{+}(f,g)}(\xi)=&\int_{\R^{d}\times \mathbb{S}^{d-1}}\frac{\wt{f}(\xi^{+}+\eta)\wt{g}(\xi^{-}-\eta)}{|\eta|^{d+\ga}}
\textbf{b}(\frac{\xi}{|\xi|}\cdot \omega)d\eta d\omega,
\end{align}
where $\xi^{+}=\frac{1}{2}\lrs{\xi+|\xi|\omega}$ and $\xi^{-}=\frac{1}{2}\lrs{\xi-|\xi|\omega}$.
For convenience, we take the notation that $\wt{Q}^{\pm}(\wt{f},\wt{g})=\mathcal{F}_{v\mapsto \xi}^{-1}\lrc{Q^{\pm}(f,g)}$.

In Sections \ref{section:Bilinear Estimate for Loss Term}-\ref{section:Bilinear Estimate for Gain Term}, we establish the bilinear estimates for the loss and gain terms respectively. Then in Section \ref{section:Proof of Well-posedness}, we complete the proof of the well-posedness of \eqref{equ:Boltzmann}.

\subsection{Bilinear Estimate for Loss Term}\label{section:Bilinear Estimate for Loss Term}

\begin{lemma}
For $s>\frac{d-1}{2}$, it holds that
\begin{align}\label{equ:bilinear estimate,Q-,L2}
\n{\lra{\nabla_{x}}^{s}\lra{\nabla_{\xi}}^{s+\ga}\wt{Q}^{-}(\wt{f},\wt{g})}_{L_{t}^{2+}L_{x\xi}^{2}}\lesssim \n{\wt{f}}_{X^{s,s+\ga,b}}\n{\wt{g}}_{X^{s,s+\ga,b}}.
\end{align}
\end{lemma}
\begin{proof}
By the fractional Leibniz rule in Lemma \ref{lemma:generalized leibniz rule}, we have
\begin{align*}
\left\Vert \wt{Q}^{-}(\wt{f},\wt{g})\right\Vert_{L_{t}^{2+}H_{x}^{s}H_{\xi}^{s+\ga}} = &\left\Vert \left\langle \nabla_{x}\right\rangle^{s} \int \left\langle \nabla _{\xi}\right\rangle ^{s+\ga}\widetilde{f}(t,x,\xi-\eta)
\frac{\widetilde{g}(t,x,\eta)}{|\eta|^{d+\ga}}d\eta \right\Vert _{L_{t}^{2+}L_{x\xi}^{2}}\\
\lesssim& \left\Vert  \int \bn{\lra{\nabla_{x}}^{s} \left\langle \nabla _{\xi}\right\rangle ^{s+\ga}\widetilde{f}(t,x,\xi-\eta)}_{L_{x}^{2}}
\bbn{\frac{\widetilde{g}(t,x,\eta)}{|\eta|^{d+\ga}}}_{L_{x}^{\infty}}d\eta \right\Vert _{L_{t}^{2+}L_{\xi }^{2}}\\
&+\left\Vert  \int \bn{ \left\langle \nabla _{\xi}\right\rangle ^{s+\ga}\widetilde{f}(t,x,\xi-\eta)}_{L_{x}^{2d+}}
\bbn{\frac{\lra{\nabla_{x}}^{s} \widetilde{g}(t,x,\eta)}{|\eta|^{d+\ga}}}_{L_{x}^{\frac{2d}{d-1}-}}d\eta \right\Vert _{L_{t}^{2+}L_{\xi }^{2}}.
\end{align*}
Applying Sobolev inequalities that $W^{s,\frac{2d}{d-1}-}\hookrightarrow L^{\infty}$, $W^{s,2}\hookrightarrow L^{2d+}$ and Young's inequality,
\begin{align}
\left\Vert \wt{Q}^{-}(\widetilde{f},\widetilde{g})\right\Vert_{L_{t}^{2+}H_{x}^{s}H_{\xi}^{s+\ga}}\lesssim&
  \bn{\lra{\nabla_{x}}^{s} \left\langle \nabla_{\xi}\right\rangle^{s+\ga}\widetilde{f}(t,x,\xi)}_{L_{t}^{\infty}L_{\xi}^{2}L_{x}^{2}}
\bbn{\frac{\lra{\nabla_{x}}^{s} \widetilde{g}(t,x,\eta)}{|\eta|^{d+\ga}}}_{L_{t}^{2+}L_{\eta}^{1}L_{x}^{\frac{2d}{d-1}-}}\notag\\
&+  \bn{ \left\langle \nabla _{\xi}\right\rangle ^{s+\ga}\widetilde{f}(t,x,\xi)}_{L_{t}^{\infty}L_{\xi}^{2}L_{x}^{2d+}}
\bbn{\frac{\lra{\nabla_{x}}^{s} \widetilde{g}(t,x,\eta)}{|\eta|^{d+\ga}}}_{L_{t}^{2+}L_{\eta}^{1}L_{x}^{\frac{2d}{d-1}-}}\notag\\
\lesssim& \bn{\lra{\nabla_{x}}^{s} \left\langle \nabla_{\xi}\right\rangle^{s+\ga}\widetilde{f}(t,x,\xi)}_{L_{t}^{\infty}L_{\xi}^{2}L_{x}^{2}}
\bbn{\frac{\lra{\nabla_{x}}^{s} \widetilde{g}(t,x,\eta)}{|\eta|^{d+\ga}}}_{L_{t}^{2+}L_{\eta}^{1}L_{x}^{\frac{2d}{d-1}-}}.\label{equ:bilinear estimate,Q-,g}
\end{align}
We are left to deal with the last term on the right hand side of \eqref{equ:bilinear estimate,Q-,g}. Set
\begin{align*}
G(\eta)=\bbn{\lra{\nabla_{x}}^{s} \widetilde{g}(t,x,\eta)}_{L_{x}^{\frac{2d}{d-1}-}}.
\end{align*}
Then by Hardy-Littlewood-Sobolev inequality \eqref{equ:endpoint estimate,hls,L1} in Lemma \ref{lemma:endpoint estimate,hls} with $\la=d+\ga$, we obtain
\begin{align*}
\int \frac{G(\eta)}{|\eta|^{d+\ga}}d\eta\lesssim \n{G}_{L^{\frac{2d}{d-1}-}}^{\al}\n{G}_{L^{\frac{d}{-\ga}+}}^{1-\al}.
\end{align*}
Therefore, we have
\begin{align*}
& \bbn{\frac{\lra{\nabla_{x}}^{s} \widetilde{g}(t,x,\eta)}{|\eta|^{d+\ga}}}_{L_{t}^{2+}L_{\eta}^{1}L_{x}^{\frac{2d}{d-1}-}}\\
\lesssim& \bn{\lra{\nabla_{x}}^{s} \widetilde{g}(t,x,\eta)}_{L_{t}^{2+}L_{\eta}^{\frac{2d}{d-1}-}L_{x}^{\frac{2d}{d-1}-}}^{\al}
\bn{\lra{\nabla_{x}}^{s} \widetilde{g}(t,x,\eta)}_{L_{t}^{2+}L_{\eta}^{\frac{d}{-\ga}+}L_{x}^{\frac{2d}{d-1}-}}^{1-\al}\\
\leq& \bn{\lra{\nabla_{x}}^{s} \widetilde{g}(t,x,\eta)}_{L_{t}^{2+}L_{x}^{\frac{2d}{d-1}-}L_{\eta}^{\frac{2d}{d-1}-}}^{\al}
\bn{\lra{\nabla_{x}}^{s} \widetilde{g}(t,x,\eta)}_{L_{t}^{2+}L_{x}^{\frac{2d}{d-1}-}L_{\eta}^{\frac{d}{-\ga}+}}^{1-\al}
\end{align*}
where in the last inequality we have used the Minkowski inequality. Applying Sobolev inequality that $W^{s+\ga,\frac{2d}{d-1}-}\hookrightarrow L^{\frac{d}{-\ga}+}$ and Strichartz estimate \eqref{equ:strichartz estimate,xsb}, we arrive at
\begin{align*}
&\bbn{\frac{\lra{\nabla_{x}}^{s} \widetilde{g}(t,x,\eta)}{|\eta|^{d+\ga}}}_{L_{t}^{2+}L_{\eta}^{1}L_{x}^{\frac{2d}{d-1}-}}\\
\leq& \bn{\lra{\nabla_{x}}^{s} \widetilde{g}(t,x,\eta)}_{L_{t}^{2+}L_{x}^{\frac{2d}{d-1}-}L_{\eta}^{\frac{2d}{d-1}-}}^{\al}
\bn{\lra{\nabla_{\eta}}^{s+\ga}\lra{\nabla_{x}}^{s} \widetilde{g}(t,x,\eta)}_{L_{t}^{2+}L_{x}^{\frac{2d}{d-1}-}L_{\eta}^{\frac{2d}{d-1}-}}^{1-\al}\\
\leq&\bn{\lra{\nabla_{\eta}}^{s+\ga}\lra{\nabla_{x}}^{s} \widetilde{g}(t,x,\eta)}_{L_{t}^{2+}L_{x}^{\frac{2d}{d-1}-}L_{\eta}^{\frac{2d}{d-1}-}}\\
\leq& \n{\wt{g}}_{X^{s,s+\ga,b}}.
\end{align*}
Hence, we complete the proof of \eqref{equ:bilinear estimate,Q-,L2}.
\end{proof}

\subsection{Bilinear Estimate for Gain Term}\label{section:Bilinear Estimate for Gain Term}
Before proving the bilinear estimate for the gain term, we first give a useful lemma as follows.
\begin{lemma}\label{lemma:Q+,bilinear estimate}
Let  $\frac{1}{p}+\frac{1}{q}=\frac{1}{2}$.
\begin{align}\label{equ:Q+,bilinear estimate}
\bbn{\int_{\mathbb{S}^{d-1}}\int_{\R^{d}}\frac{\wt{f}(\xi^{+}+\eta)\wt{g}(\xi^{-}-\eta)}{|\eta|^{d+\ga}}
\textbf{b}(\frac{\xi}{|\xi|}\cdot \omega)d\eta d\omega}_{L_{\xi}^{2}}\lesssim
\n{\wt{f}}_{L^{\frac{2pd}{2d-p\ga}}} \n{\wt{g}}_{L^{\frac{2qd}{2d-q\ga}}}.
\end{align}
In particular, we have
\begin{align}
\n{\wt{Q}^{+}(\wt{f},\wt{g})}_{L_{\xi}^{2}}
\lesssim& \n{\wt{f}}_{L_{\xi}^{2}}
\n{\wt{g}}_{L_{\xi}^{\frac{d}{-\ga}}},\label{equ:Q+,bilinear estimate,PM,f,g}\\
\n{\wt{Q}^{+}(\wt{f},\wt{g})}_{L_{\xi}^{2}}
\lesssim&  \n{\wt{f}}_{L_{\xi}^{\frac{d}{-\ga}}}
\n{\wt{g}}_{L_{\xi}^{2}}.\label{equ:Q+,bilinear estimate,PM,g,f}
\end{align}
\end{lemma}
\begin{proof}
For the case of Maxwellian molecules, it holds that
\begin{align}\label{equ:Q+,bilinear estimate,constant kernel}
\bbn{\int_{\mathbb{S}^{d-1}} \wt{f}(\xi^{+})\wt{g}(\xi^{-})
\textbf{b}(\frac{\xi}{|\xi|}\cdot \omega) d\omega}_{L_{\xi}^{2}}\lesssim \n{\wt{f}}_{L_{\xi}^{p}}\n{\wt{g}}_{L_{\xi}^{q}},\quad \frac{1}{p}+\frac{1}{q}=\frac{1}{2},
\end{align}
which is proved in \cite[Theorem 1]{alonso2010estimates}.
By Cauchy-Schwarz inequality and then \eqref{equ:Q+,bilinear estimate,constant kernel}, we have
\begin{align}\label{equ:Q+,bilinear estimate,proof}
&\bbn{\int_{\mathbb{S}^{d-1}}\int_{\mathbb{R}^{d}}\frac{\wt{f}(\xi^{+}+\eta)\wt{g}(\xi^{-}-\eta)}{|\eta|^{d+\ga}}
\textbf{b}(\frac{\xi}{|\xi|}\cdot \omega)d\eta d\omega}_{L_{\xi}^{2}}\\
\leq& \bbn{\int_{\mathbb{S}^{d-1}}\lrc{\int_{\mathbb{R}^{d}}\frac{|\wt{f}(\xi^{+}+\eta)|^{2}}{|\eta|^{d+\ga}}
d\eta}^{\frac{1}{2}}
\lrc{ \int_{\mathbb{R}^{d}}\frac{|\wt{g}(\xi^{-}-\eta)|^{2}}{|\eta|^{d+\ga}}d\eta}^{\frac{1}{2}}\textbf{b}(\frac{\xi}{|\xi|}\cdot \omega) d\omega}_{L_{\xi}^{2}}\notag\\
\lesssim& \bbn{\lrc{\int_{\mathbb{R}^{d}}\frac{|\wt{f}(\xi+\eta)|^{2}}{|\eta|^{d+\ga}}
d\eta}^{\frac{1}{2}}}_{L_{\xi}^{p}}
\bbn{\lrc{\int_{\mathbb{R}^{d}}\frac{|\wt{g}(\xi-\eta)|^{2}}{|\eta|^{d+\ga}}d\eta}^{\frac{1}{2}}}_{L_{\xi}^{q}} \notag\\
=& \bbn{\int_{\mathbb{R}^{d}}\frac{|\wt{f}(\xi+\eta)|^{2}}{|\eta|^{d+\ga}}
d\eta}^{\frac{1}{2}}_{L_{\xi}^{\frac{p}{2}}}
\bbn{\int_{\mathbb{R}^{d}}\frac{|\wt{g}(\xi-\eta)|^{2}}{|\eta|^{d+\ga}}d\eta}^{\frac{1}{2}}_{L_{\xi}^{\frac{q}{2}}}\notag\\
\lesssim& \n{\wt{f}}_{L^{\frac{2pd}{2d-p\ga}}} \n{\wt{g}}_{L^{\frac{2qd}{2d-q\ga}}},\notag
\end{align}
where in the last inequality we have used Hardy-Littlewood-Sobolev inequality \eqref{equ:hardy-littlewood-sobolev inequality}.
This completes the proof of \eqref{equ:Q+,bilinear estimate}.
Then by taking
$$(p,q)=(\frac{2d}{d+\ga},-\frac{2d}{\ga}),\quad (p,q)=(-\frac{2d}{\ga},\frac{2d}{d+\ga}),$$
we immediately obtain \eqref{equ:Q+,bilinear estimate,PM,f,g} and \eqref{equ:Q+,bilinear estimate,PM,g,f}.

\end{proof}

To prove the bilinear estimate for the gain term, we need a detailed frequency analysis from Littlewood-Paley theory.\footnote{See \cite{chen2019derivation,chen2022unconditional,CSZ22} for some examples sharing similar critical flavor but carrying completely different structures.}
Let $\chi(x)$ be a smooth function and satisfy $\chi(x)=1$ for all $|x|\leq 1$ and $\chi(x)=0$ for $|x|\geq 2$. Let $N$ be a dyadic number and set $\vp_{N}(x)=\chi(\frac{x}{N})-\chi(\frac{x}{2N})$. Define the Littlewood-Paley projector
\begin{align}
\widehat{P_{N}u}(\eta)=\vp_{N}(\eta)\widehat{u}(\eta).
\end{align}
We denote by $P_{N}^{x}$/$P_{M}^{\xi}$ the projector of the $x$-variable and $\xi$-variable respectively.
Now, we delve into the analysis of the bilinear estimate.
\begin{lemma}
 For $s>\frac{d-1}{2}$, we have
\begin{align}\label{equ:Q+,bilinear estimate,L2}
\n{\lra{\nabla_{x}}^{s}\lra{\nabla_{\xi}}^{s+\ga}\wt{Q}^{+}(\wt{f},\wt{g})}_{L_{t}^{2}L_{x\xi}^{2}}\lesssim \n{\wt{f}}_{X^{s,s+\ga,b}}\n{\wt{g}}_{X^{s,s+\ga,b}}.
\end{align}
\end{lemma}
\begin{proof}
By duality,
\eqref{equ:Q+,bilinear estimate,L2} is equivalent to
\begin{align}\label{equ:Q+,bilinear estimate,L2,equivalent}
\int \wt{Q}^{+}(\wt{f},\wt{g}) h dx d\xi dt\lesssim
\n{\wt{f}}_{X^{s,s+\ga,b}}\n{\wt{g}}_{X^{s,s+\ga,b}} \n{h}_{L_{t}^{2}H_{\xi}^{-s-\ga}H_{x}^{-s}} .
\end{align}
We denote by $I$ the integral in \eqref{equ:Q+,bilinear estimate,L2,equivalent}.
Inserting a Littlewood-Paley decomposition gives that
\begin{align*}
I=\sum_{\substack{M,M_{1},M_{2}\\N,N_{1},N_{2}}} I_{M,M_{1},M_{2},N,N_{1},N_{2}}
\end{align*}
where
\begin{align*}
I_{M,M_{1},M_{2},N,N_{1},N_{2}}= \int \wt{Q}^{+}(P_{N_{1}}^{x}P_{M_{1}}^{\xi}\wt{f},P_{N_{2}}^{x}P_{M_{2}}^{\xi}\wt{g}) P_{N}^{x}P_{M}^{\xi}h dx d\xi dt.
\end{align*}
Note that $\wt{Q}^{+}$ commutes with $P_{N}^{x}$, so this gives the constraint that $N\lesssim \max\lrs{N_{1},N_{2}}$ due to that
 \begin{align} \label{equ:property,constraint,projector}
 P_{N}^{x}(P_{N_{1}}^{x}\wt{f}P_{N_{2}}^{x}\wt{g})=0, \quad \text{if $N\geq 10\max\lrs{N_{1},N_{2}}$.}
 \end{align}
In addition, we observe that such a property \eqref{equ:property,constraint,projector} is also hinted in the $\xi$-variable, that is,
\begin{align}\label{equ:property,constraint,projector,v,variable}
P_{M}^{\xi}\wt{Q}^{+}(P_{M_{1}}^{\xi}\wt{f},P_{M_{2}}^{\xi}\wt{g})=0, \quad \text{if $M\geq 10\max\lrs{M_{1},M_{2}}$.}
\end{align}
Indeed, notice that
\begin{align}\label{equ:fourier transform,decomposition,Q+}
\cf_{\xi}\lrs{P_{M}^{\xi}\wt{Q}^{+}(P_{M_{1}}^{\xi}\wt{f},P_{M_{2}}^{\xi}\wt{g})}=\vp_{M}(v)\int_{\mathbb{R}^{d}}\int_{\mathbb{S}^{d-1}} (\vp_{M_{1}}f)(v^{\ast })(\vp_{M_{2}}g)(u^{\ast}) B(u-v,\omega)dud\omega.
\end{align}
Then from the energy conservation which implies the inequality $|v|^{2}\leq |v^{*}|^{2}+|u^{*}|^{2}$, we have the lower bound that
\begin{align}\label{equ:u*,v*,lower bound}
\text{$|u^{*}|\geq \frac{M}{4}$, or  $|v^{*}|\geq \frac{M}{4}$}
\end{align}
for all $(u,\omega)\in \R^{d}\times \mathbb{S}^{d-1}$ and $|v|\geq \frac{M}{2}$.
Therefore, for $M\geq 10\max\lrs{M_{1},M_{2}}$, the lower bound \eqref{equ:u*,v*,lower bound} forces the $v^{*}$-variable or $u^{*}$-variable off their own support set, which makes the integral on the right hand side of \eqref{equ:fourier transform,decomposition,Q+} vanish. Hence, this gives the constraint that $M\lesssim \max\lrs{M_{1},M_{2}}$.

 Now, we divide the sum into four cases as follows

Case A. $M_{1}\geq M_{2}$, $N_{1}\geq N_{2}$.

Case B. $M_{1}\leq M_{2}$, $N_{1}\geq N_{2}$.

Case C. $M_{1}\geq M_{2}$, $N_{1}\leq N_{2}$.

Case D. $M_{1}\leq M_{2}$, $N_{1}\leq N_{2}$.

We only need to treat Cases A and B, as Cases C and D follow similarly.

\textbf{Case A. $M_{1}\geq M_{2}$, $N_{1}\geq N_{2}$.}

Let $I_{A}$ denote the integral restricted to the Case A. By Cauchy-Schwarz,
\begin{align*}
I_{A}\lesssim
\sum_{\substack{ _{\substack{ M,M_{1}\geq M_{2}  \\ N,N_{1}\geq N_{2}}}
\\ M_{1}\geqslant M,N_{1}\geqslant N}}
\n{\wt{Q}^{+}(P_{N_{1}}^{x}P_{M_{1}}^{\xi}\wt{f},P_{N_{2}}^{x}P_{M_{2}}^{\xi}\wt{g})}_{L_{t}^{2}L_{x\xi}^{2}}
\n{P_{N}^{x}P_{M}^{\xi}h}_{L_{t}^{2}L_{x\xi}^{2}}.
\end{align*}
By using the estimate \eqref{equ:Q+,bilinear estimate,PM,f,g} in Lemma \ref{equ:Q+,bilinear estimate} and then H\"{o}lder inequality, we have
\begin{align*}
I_{A} \lesssim &\sum_{\substack{ _{\substack{ M,M_{1}\geq M_{2}  \\ N,N_{1}\geq N_{2}}}
\\ M_{1}\geqslant M,N_{1}\geqslant N}}
 \bbn{\n{P_{N_{1}}^{x}P_{M_{1}}^{\xi }\widetilde{f}}_{L_{\xi}^{2}}
\Vert P_{N_{2}}^{x}P_{M_{2}}^{\xi }\widetilde{g}\Vert _{L_{\xi }^{\frac{d}{-\ga}}}}_{L_{t}^{2}L_{x}^{2}}\n{P_{N}^{x}P_{M}^{\xi}h}_{L_{t}^{2}L_{x\xi}^{2}} \\
\leq&\sum_{\substack{ _{\substack{ M,M_{1}\geq M_{2}  \\ N,N_{1}\geq N_{2}}}
\\ M_{1}\geqslant M,N_{1}\geqslant N}}
 \n{ P_{N_{1}}^{x}P_{M_{1}}^{\xi }\widetilde{f}}_{L_{t}^{\infty}L_{x}^{2}L_{\xi }^{2}}
 \n{ P_{N_{2}}^{x}P_{M_{2}}^{\xi }\widetilde{g}}_{L_{t}^{2}L_{x}^{\infty}L_{\xi }^{\frac{d}{-\ga}}}
 \n{P_{N}^{x}P_{M}^{\xi}h}_{L_{t}^{2}L_{x\xi}^{2}} .
\end{align*}
By using Minkowski inequality, Sobolev inequality that $W^{\frac{d-1}{2}+\ga,\frac{2d}{d-1}}\hookrightarrow L^{\frac{d}{-\ga}}$, and Bernstein inequality that $\n{P_{N_{2}}^{x}\wt{g}}_{L_{x}^{\infty}}\lesssim \n{\lra{\nabla_{x}}^{\frac{d-1}{2}}P_{N_{2}}^{x}\wt{g}}_{L_{x}^{\frac{2d}{d-1}}}$, we obtain
\begin{align*}
I_{A} \lesssim &\sum_{\substack{ _{\substack{ M,M_{1}\geq M_{2}  \\ N,N_{1}\geq N_{2}}}
\\ M_{1}\geqslant M,N_{1}\geqslant N}}
\frac{N^{s}M^{s+\ga}}{N_{1}^{s}M_{1}^{s+\ga}}
\n{P_{N_{1}}^{x}P_{M_{1}}^{\xi }\lra{\nabla_{x}}^{s}\lra{\nabla _{\xi}}^{s+\ga}\widetilde{f}}_{L_{t}^{\infty}L_{x}^{2}L_{\xi}^{2}} \\
&\times \n{P_{N_{2}}^{x}P_{M_{2}}^{\xi}\lra{\nabla_{x}}^{\frac{d-1}{2}}\lra{\nabla_{\xi}}^{\frac{d-1}{2}+\ga} \widetilde{g}}_{L_{t}^{2}L_{x}^{\frac{2d}{d-1}}L_{\xi}^{\frac{2d}{d-1}}}
\n{P_{N}^{x}P_{M}^{\xi}\lra{\nabla_{x}}^{-s}\lra{\nabla_{\xi}}^{-s-\ga}h}_{L_{t}^{2}L_{x\xi}^{2}}\\
\lesssim &\sum_{\substack{ _{\substack{ M,M_{1}\geq M_{2}  \\ N,N_{1}\geq N_{2}}}
\\ M_{1}\geqslant M,N_{1}\geqslant N}}
\frac{N^{s}M^{s+\ga}}{
N_{1}^{s}M_{1}^{s+\ga}}
\frac{1}{N_{2}^{s-\frac{d-1}{2}}}\frac{1}{M_{2}^{s-\frac{d-1}{2}}} \n{P_{N}^{x}P_{M}^{\xi}h}_{L_{t}^{2}H_{\xi}^{-s-\ga}H_{x}^{-s}}\\
&\times \n{P_{N_{1}}^{x}P_{M_{1}}^{\xi }\left\langle \nabla
_{x}\right\rangle^{s}\left\langle \nabla_{\xi}\right\rangle^{s+\ga}\widetilde{f}}_{L_{t}^{\infty}L_{x}^{2}L_{\xi }^{2}}
\n{P_{N_{2}}^{x}P_{M_{2}}^{\xi}\left\langle \nabla _{x}\right\rangle^{s}\left\langle \nabla _{\xi
}\right\rangle ^{s+\ga}\widetilde{g}}_{L_{t}^{2}L_{x}^{\frac{2d}{d-1}}L_{\xi}^{\frac{2d}{d-1}}},
\end{align*}
where in the last inequality we have used Bernstein inequality again.
By Strichartz estimate \eqref{equ:strichartz estimate,xsb},
\begin{align*}
I_{A}\lesssim&\sum_{\substack{ _{\substack{ M,M_{1}\geq M_{2}  \\ N,N_{1}\geq N_{2}}}
\\ M_{1}\geqslant M,N_{1}\geqslant N}}
\frac{N^{s}M^{s+\ga}}{
N_{1}^{s}M_{1}^{s+\ga}}
\frac{1}{N_{2}^{s-\frac{d-1}{2}}}\frac{1}{M_{2}^{s-\frac{d-1}{2}}} \n{P_{N}^{x}P_{M}^{\xi}h}_{L_{t}^{2}H_{\xi}^{-s-\ga}H_{x}^{-s}}\\
&\times \Vert P_{N_{1}}^{x}P_{M_{1}}^{\xi }\widetilde{f}\Vert _{X^{s,s+\ga,b}}\Vert P_{N_{2}}^{x}P_{M_{2}}^{\xi }\widetilde{g}\Vert _{X^{s,s+\ga,b}}.
   \end{align*}
Note that $s>\frac{d-1}{2}$, so we use that $\Vert P_{N_{2}}^{x}P_{M_{2}}^{\xi }\widetilde{g}\Vert _{X^{s,s+\ga,b}}\lesssim \Vert\widetilde{g}\Vert _{X^{s,s+\ga,b}}$ and then carry out the $N_{2}$, $M_{2}$ sums to obtain
\begin{align*}
I_{A}\lesssim& \n{\widetilde{g}}_{X^{s,s+\ga,b}}\sum_{\substack{ M_{1}\geq M  \\ N_{1}\geq N}}
\frac{N^{s}M^{s+\ga}}{N_{1}^{s}M_{1}^{s+\ga}}\n{P_{N_{1}}^{x}P_{M_{1}}^{\xi}\wt{f}}
_{X^{s,s+\ga,b}}\n{P_{N}^{x}P_{M}^{\xi}h}_{L_{t}^{2}H_{\xi}^{-s-\ga}H_{x}^{-s}}.
\end{align*}
By Cauchy-Schwarz in $M$, $M_{1}$, $N$ and $N_{1}$, we have
\begin{align}\label{equ:bilinear estimate,Q+,cauchy-schwarz}
I_{A}\lesssim& \n{\widetilde{g}}_{X^{s,s+\ga,b}}\lrs{\sum_{\substack{ M_{1}\geq M  \\ N_{1}\geq N}}
\frac{N^{s}M^{s+\ga}}{N_{1}^{s}M_{1}^{s+\ga}}
\n{P_{N_{1}}^{x}P_{M_{1}}^{\xi}\wt{f}}_{X^{s,s+\ga,b}}^{2}}^{\frac{1}{2}}\\
&\lrs{\sum_{\substack{ M_{1}\geq M  \\ N_{1}\geq N}}
\frac{N^{s}M^{s+\ga}}{N_{1}^{s}M_{1}^{s+\ga}}
\n{P_{N}^{x}P_{M}^{\xi}h}_{L_{t}^{2}H_{\xi}^{-s-\ga}H_{x}^{-s}}^{2}
}^{\frac{1}{2}} \notag \\
\lesssim &\n{\wt{f}}_{X^{s,s+\ga,b}}\n{\wt{g}}_{X^{s,s+\ga,b}} \n{h}_{L_{t}^{2}H_{\xi}^{-s-\ga}H_{x}^{-s}},\notag
\end{align}
which completes the proof of \eqref{equ:Q+,bilinear estimate,L2,equivalent} for Case A.

\textbf{Case B. $M_{1}\leq M_{2}$, $N_{1}\geq N_{2}$.}

Following the same way as Case A, we have

\begin{align*}
I_{B}\lesssim
\sum_{\substack{ _{\substack{ M,M_{2}\geq M_{1}  \\ N,N_{1}\geq N_{2}}}
\\ M_{2}\geqslant M,N_{1}\geqslant N}}
\n{\wt{Q}^{+}(P_{N_{1}}^{x}P_{M_{1}}^{\xi}\wt{f},P_{N_{2}}^{x}P_{M_{2}}^{\xi}\wt{g})}_{L_{t}^{2}L_{x\xi}^{2}}
\n{P_{N}^{x}P_{M}^{\xi}h}_{L_{t}^{2}L_{x\xi}^{2}}.
\end{align*}
By using the estimate \eqref{equ:Q+,bilinear estimate,PM,g,f} in Lemma \ref{equ:Q+,bilinear estimate} and then H\"{o}lder inequality, we have
\begin{align*}
I_{B} \lesssim &\sum_{\substack{ _{\substack{ M,M_{2}\geq M_{1}  \\ N,N_{1}\geq N_{2}}}
\\ M_{2}\geqslant M,N_{1}\geqslant N}}
 \bbn{\n{P_{N_{1}}^{x}P_{M_{1}}^{\xi }\widetilde{f}}_{L_{\xi}^{\frac{d}{-\ga}}}
\Vert P_{N_{2}}^{x}P_{M_{2}}^{\xi }\widetilde{g}\Vert _{L_{\xi }^{2}}}_{L_{t}^{2}L_{x}^{2}}\n{P_{N}^{x}P_{M}^{\xi}h}_{L_{t}^{2}L_{x\xi}^{2}} \\
\leq&\sum_{\substack{ _{\substack{ M,M_{2}\geq M_{1}  \\ N,N_{1}\geq N_{2}}}
\\ M_{2}\geqslant M,N_{1}\geqslant N}}
 \n{ P_{N_{1}}^{x}P_{M_{1}}^{\xi }\widetilde{f}}_{L_{t}^{2}L_{x}^{\frac{2d}{d-1}}L_{\xi }^{\frac{d}{-\ga}}}
 \n{ P_{N_{2}}^{x}P_{M_{2}}^{\xi }\widetilde{g}}_{L_{t}^{\infty}L_{x}^{2d}L_{\xi}^{2}}
 \n{P_{N}^{x}P_{M}^{\xi}h}_{L_{t}^{2}L_{x\xi}^{2}} .
\end{align*}
By Minkowski inequality, Sobolev inequality that $W^{\frac{d-1}{2}+\ga,\frac{2d}{d-1}}\hookrightarrow L^{\frac{d}{-\ga}}$, $W^{\frac{d-1}{2},2}\hookrightarrow L^{2d}$ and Bernstein inequality, we obtain
\begin{align*}
I_{B} \lesssim &\sum_{\substack{ _{\substack{ M,M_{2}\geq M_{1}  \\ N,N_{1}\geq N_{2}}}
\\ M_{2}\geqslant M,N_{1}\geqslant N}}
\n{P_{N_{1}}^{x}P_{M_{1}}^{\xi }\lra{\nabla_{\xi}}^{\frac{d-1}{2}+\ga}\widetilde{f}}_{L_{t}^{2}L_{x}^{\frac{2d}{d-1}}L_{\xi}^{\frac{2d}{d-1}}} \\
&\times \n{P_{N_{2}}^{x}P_{M_{2}}^{\xi}\lra{\nabla_{x}}^{\frac{d-1}{2}} \widetilde{g}}_{L_{t}^{\infty}L_{x}^{2}L_{\xi}^{2}}
\n{P_{N}^{x}P_{M}^{\xi}h}_{L_{t}^{2}L_{x\xi}^{2}}\\
\lesssim &\sum_{\substack{ _{\substack{ M,M_{2}\geq M_{1}  \\ N,N_{1}\geq N_{2}}}
\\ M_{2}\geqslant M,N_{1}\geqslant N}}
\frac{N^{s}M^{s+\ga}}{
N_{1}^{s}M_{2}^{s+\ga}}
\frac{1}{N_{2}^{s-\frac{d-1}{2}}}\frac{1}{M_{1}^{s-\frac{d-1}{2}}} \n{P_{N}^{x}P_{M}^{\xi}h}_{L_{t}^{2}H_{\xi}^{-s-\ga}H_{x}^{-s}}\\
&\times \n{P_{N_{1}}^{x}P_{M_{1}}^{\xi }\lra{ \nabla_{x}}^{s}\lra{ \nabla_{\xi}}^{s+\ga}\widetilde{f}}_{L_{t}^{2}L_{x}^{\frac{2d}{d-1}}L_{\xi}^{\frac{2d}{d-1}}}
\n{P_{N_{2}}^{x}P_{M_{2}}^{\xi}\lra{\nabla _{x}}^{s}\left\langle \nabla_{\xi
}\right\rangle ^{s+\ga}\widetilde{g}}_{L_{t}^{\infty}L_{x}^{2}L_{\xi }^{2}}.
\end{align*}
By Strichartz estimate \eqref{equ:strichartz estimate,xsb},
\begin{align*}
I_{B}\lesssim&\sum_{\substack{ _{\substack{ M,M_{2}\geq M_{1}  \\ N,N_{1}\geq N_{2}}}
\\ M_{2}\geqslant M,N_{1}\geqslant N}}
\frac{N^{s}M^{s+\ga}}{
N_{1}^{s}M_{2}^{s+\ga}}
\frac{1}{N_{2}^{s-\frac{d-1}{2}}}\frac{1}{M_{1}^{s-\frac{d-1}{2}}} \n{P_{N}^{x}P_{M}^{\xi}h}_{L_{t}^{2}H_{\xi}^{-s-\ga}H_{x}^{-s}}\\
&\times \Vert P_{N_{1}}^{x}P_{M_{1}}^{\xi }\widetilde{f}\Vert _{X^{s,s+\ga,b}}\Vert P_{N_{2}}^{x}P_{M_{2}}^{\xi }\widetilde{g}\Vert _{X^{s,s+\ga,b}}.
   \end{align*}
Note that $s>\frac{d-1}{2}$, so we use that
$$\n{P_{N_{1}}^{x}P_{M_{1}}^{\xi }\widetilde{f}}_{X^{s,s+\ga,b}}\lesssim \n{P_{N_{1}}^{x}\widetilde{f}} _{X^{s,s+\ga,b}},\quad
\n{P_{N_{2}}^{x}P_{M_{2}}^{\xi }\widetilde{g}}_{X^{s,s+\ga,b}}\lesssim \n{P_{M_{2}}^{\xi }\widetilde{g}} _{X^{s,s+\ga,b}},$$ and then carry out the $N_{2}$, $M_{1}$ sums to obtain
\begin{align*}
I_{B}\lesssim& \sum_{\substack{ M_{2}\geq M  \\ N_{1}\geq N}}
\frac{N^{s}M^{s+\ga}}{N_{1}^{s}M_{2}^{s+\ga}}
\n{P_{N_{1}}^{x}\wt{f}}_{X^{s,s+\ga,b}}
\n{P_{M_{2}}^{\xi }\widetilde{g}} _{X^{s,s+\ga,b}}
\n{P_{N}^{x}P_{M}^{\xi}h}_{L_{t}^{2}H_{\xi}^{-s-\ga}H_{x}^{-s}}.
\end{align*}
In a similar way to \eqref{equ:bilinear estimate,Q+,cauchy-schwarz}, we use Cauchy-Schwarz inequality to get
\begin{align*}
I_{B}\lesssim& \sum_{N_{1}\geq N}\frac{N^{s}}{N_{1}^{s}}\n{P_{N_{1}}^{x}\wt{f}}_{X^{s,s+\ga,b}}\lrs{\sum_{M_{2}\geq M}\frac{M^{s+\ga}}{M_{2}^{s+\ga}}\n{P_{M_{2}}^{\xi }\widetilde{g}} _{X^{s,s+\ga,b}}^{2}}^{\frac{1}{2}}\\
&\times \lrs{\sum_{M_{2}\geq M}\frac{M^{s+\ga}}{M_{2}^{s+\ga}}\n{P_{N}^{x}P_{M}^{\xi}h}_{L_{t}^{2}H_{\xi}^{-s-\ga}H_{x}^{-s}}^{2}}^{\frac{1}{2}}\\
\lesssim&\n{\widetilde{g}}_{X^{s,s+\ga,b}}\sum_{N_{1}\geq N}\frac{N^{s}}{N_{1}^{s}}\n{P_{N_{1}}^{x}\wt{f}}_{X^{s,s+\ga,b}}\n{P_{N}^{x}h}_{L_{t}^{2}H_{\xi}^{-s-\ga}H_{x}^{-s}}\\
\lesssim& \n{\wt{f}}_{X^{s,s+\ga,b}}\n{\wt{g}}_{X^{s,s+\ga,b}} \n{h}_{L_{t}^{2}H_{\xi}^{-s-\ga}H_{x}^{-s}}.
\end{align*}
Hence, we complete the proof of of \eqref{equ:Q+,bilinear estimate,L2,equivalent} for Case $B$.

\end{proof}

\subsection{Well-posedness in Fourier Restriction Norm Space}\label{section:Proof of Well-posedness}

We first recall some standard results on the Fourier restriction norms and Strichartz estimates.

\begin{lemma}
Let $b\in (\frac{1}{2},1)$, $s\in \R$, $r\in \R$, and $\theta(t)$ be a smooth cutoff function.
Define
\begin{align}
U(t):=e^{it \nabla_{x}\cdot \nabla_{\xi}},\quad D(F):=\int_{0}^{t}U(t-\tau)F(\tau)d\tau.
\end{align}
Then we have
\begin{align}
\n{\wt{f}}_{C_{t}^{0}H_{x}^{s}H_{\xi}^{r}}\lesssim& \n{\wt{f}}_{X^{s,r,b}} \label{equ:xsb,infty} ,\\
\n{\theta(t)S(t)\wt{f}_{0}}_{X^{s,r,b}}\lesssim& \n{\wt{f}_{0}}_{H_{x}^{s}H_{\xi}^{r}}\label{equ:xsb,free solution},\\
\n{\theta(t)D(F)}_{X^{s,r,b}}\lesssim& \n{F}_{X^{s,r,b-1}}\label{equ:xsb,duhamel estimate},\\
\n{\wt{f}}_{X^{s,r,b-1}}\lesssim&\n{\wt{f}}_{L_{t}^{p}H_{x}^{s}H_{\xi}^{r}},\quad p\in (1,2],\ b\leq \frac{3}{2}-\frac{1}{p},\label{equ:xsb,Lp estimate}\\
\n{\wt{f}}_{L_{t}^{q}L_{x\xi}^{p}}\lesssim& \n{\wt{f}}_{X^{0,0,b}},\quad \frac{2}{q}+\frac{2d}{p}=d,\quad q\geq 2,\ d\geq2.\label{equ:strichartz estimate,xsb}
\end{align}
\end{lemma}
\begin{proof}
These type estimates are well-known in the dispersive literatures.
The Strichartz estimate \eqref{equ:strichartz estimate,xsb} follows from the linear Strichartz estimate \eqref{equ:strichartz estimate,linear} and the transference principle. See for example \cite[Chapter 2.6]{tao2006nonlinear}.
\end{proof}

We prove the existence, uniqueness, and the Lipschitz continuity of the solution map. The nonnegativity of $f$ follows from the persistence of regularity (as shown in \cite{chen2019moments,lu2001on}) by use of the bilinear estimates \eqref{equ:bilinear estimate,Q-,L2} and \eqref{equ:Q+,bilinear estimate,L2} for the soft potential case.

\begin{proof}[\textbf{Proof of Well-posedness in Theorem $\ref{thm:main theorem}$}]

Let $\theta_{T}(t)=\theta(t/T)$. By estimate \eqref{equ:xsb,Lp estimate}, H\"{o}lder inequality, and the bilinear estimates for $Q^{\pm}$, we have
\begin{align}\label{equ:bilinear estimate,Q,xsb}
&\n{\theta_{T}(t)\wt{Q}(\wt{f},\wt{g})}_{X^{s,s+\ga,b-1}}\\
\lesssim& \n{\theta_{T}(t)\wt{Q}(\wt{f},\wt{g})}_{L_{t}^{\frac{2}{3-2b}}H_{x}^{s}H_{\xi}^{s+\ga}}\notag\\
\lesssim& T^{1-b}\n{\theta_{T}(t)\wt{Q}^{+}(\wt{f},\wt{g})}_{L_{t}^{2}H_{x}^{s}H_{\xi}^{s+\ga}}+
T^{1-b}\n{\theta_{T}(t)\wt{Q}^{-}(\wt{f},\wt{g})}_{L_{t}^{2+}H_{x}^{s}H_{\xi}^{s+\ga}}\notag\\
\lesssim& T^{1-b}\n{\wt{f}}_{X^{s,s+\ga,b}}\n{\wt{g}}_{X^{s,s+\ga,b}}.\notag
\end{align}

Let $B=\lr{\wt{f}:\n{\wt{f}}_{X^{s,s+\ga,b}}\leq R}$ with $R=2C\n{\wt{f}_{0}}_{H_{x}^{s}H_{\xi}^{s+\ga}}$ and define the nonlinear map
\begin{align*}
\Phi(\wt{f}):=\theta_{T}(t)U(t)\wt{f}_{0}+D(\wt{f},\wt{f}),
\end{align*}
where
\begin{align*}
D(\wt{f},\wt{f}):=\theta_{T}(t)\int_{0}^{t}U(t-\tau)\theta_{T}(\tau)\wt{Q}(\wt{f}(\tau),\wt{f}(\tau))d\tau.
\end{align*}
By estimates \eqref{equ:xsb,free solution}, \eqref{equ:xsb,duhamel estimate}, and \eqref{equ:bilinear estimate,Q,xsb}, we obtain
\begin{align*}
\n{\Phi(\wt{f})}_{X^{s,s+\ga,b}}\leq& \n{\theta_{T}(t)U(t)\wt{f}_{0}}_{X^{s,s+\ga,b}}+\n{D(\wt{f},\wt{f})}_{X^{s,s+\ga,b}}\\
\leq& C\n{\wt{f}_{0}}_{H_{x}^{s}H_{\xi}^{s+\ga}}+C\n{\theta_{T}\wt{Q}(\wt{f},\wt{f})}_{X^{s,s+\ga,b-1}}\\
\leq& \frac{R}{2}+CT^{1-b}\n{\wt{f}}_{X^{s,s+\ga,b}}^{2}\\
\leq& R
\end{align*}
where in the last inequality we have used that $CT^{1-b}R\leq \frac{1}{2}$. Thus, $\Phi$ maps the set $B$ into itself. In a similar way, for $\wt{f}$ and $\wt{g}\in B$ we have
\begin{align*}
\n{\Phi(\wt{f})-\Phi(\wt{g})}_{X^{s,s+\ga,b}}=&\n{D(\wt{f},\wt{f})-D(\wt{g},\wt{g})}_{X^{s,s+\ga,b}}\\
\leq& C\n{\theta_{T}\wt{Q}(\wt{f}-\wt{g},\wt{f})}_{X^{s,s+\ga,b-1}}+C\n{\theta_{T}\wt{Q}(\wt{g},\wt{f}-\wt{g})}_{X^{s,s+\ga,b-1}}\\
\leq& CT^{b-1}\lrs{\n{\wt{f}}_{X^{s,s+\ga,b}}+\n{\wt{g}}_{X^{s,s+\ga,b}}}\n{\wt{f}-\wt{g}}_{X^{s,s+\ga,b}}\\
\leq& \frac{1}{2}\n{\wt{f}-\wt{g}}_{X^{s,s+\ga,b}}.
\end{align*}
Therefore, $\Phi$ is a contraction mapping in $X^{s,s+\ga,b}$ and has a unique fixed point $\wt{f}$ on the time scale $|T|\sim  \lra{R}^{\frac{1}{b-1}}$.

Given two initial data $\wt{f}_{0}$ and $\wt{g}_{0}$, we set
$$R_{1}=2\max\lrs{\n{\wt{f}_{0}}_{H_{x}^{s}H_{\xi}^{s+\ga}},\n{\wt{g}_{0}}_{H_{x}^{s}H_{\xi}^{s+\ga}}}.$$
Let $\wt{f}$, $\wt{g}$ be the corresponding unique fixed points. Taking a difference gives that
\begin{align*}
\wt{f}-\wt{g}=\theta_{T}(t)U(t)(\wt{f}_{0}-\wt{g}_{0})+D(\wt{f}-\wt{g},\wt{f})+D(\wt{g},\wt{f}-\wt{g}).
\end{align*}
By estimates \eqref{equ:xsb,free solution}, \eqref{equ:xsb,duhamel estimate}, and \eqref{equ:bilinear estimate,Q,xsb}, we have
\begin{align*}
\n{\wt{f}-\wt{g}}_{X^{s,s+\ga,b}}\leq& \n{\theta_{T}(t)U(t)(\wt{f}_{0}-\wt{g}_{0})}_{X^{s,s+\ga,b}}+C\n{\theta_{T}\wt{Q}(\wt{f}-\wt{g},\wt{f})}_{X^{s,s+\ga,b-1}}\\
&+C\n{\theta_{T}\wt{Q}(\wt{g},\wt{f}-\wt{g})}_{X^{s,s+\ga,b-1}}\\
\leq& C\n{\wt{f}_{0}-\wt{g}_{0}}_{H_{x}^{s}H_{\xi}^{s+\ga}}+CT^{1-b}\lrs{\n{\wt{f}}_{X^{s,s+\ga,b}}
+\n{\wt{g}}_{X^{s,s+\ga,b}}}\n{\wt{f}-\wt{g}}_{X^{s,s+\ga,b}},
\end{align*}
which together with $CT^{1-b}R_{1}\leq \frac{1}{2}$ gives that
\begin{align*}
\n{\wt{f}-\wt{g}}_{X^{s,s+\ga,b}}\leq 2C\n{\wt{f}_{0}-\wt{g}_{0}}_{H_{x}^{s}H_{\xi}^{s+\ga}}.
\end{align*}
The Lipschitz continuity of the data-to-solution map on the time $[-T,T]$ follows from the embedding $X^{s,s+\ga,b}\hookrightarrow C([-T,T];H_{x}^{s}H_{\xi}^{s+\ga})$.
\end{proof}
\section{Ill-posedness}\label{section:Ill-posedness}

The idea is to first construct an approximation solution $f_{\mathrm{a}}(t)$ with the norm deflation property that
\begin{align}\label{equ:fa,expected}
\n{f_{\mathrm{a}}(0)}_{L_{v}^{2,r_{0}}H_{x}^{s_{0}}}\ll 1, \quad \n{f_{\mathrm{a}}(T_{*})}_{L_{v}^{2,r_{0}}H_{x}^{s_{0}}}\gtrsim 1,
\end{align}
with $T_{*}\nearrow 0$, and then use stability theory to perturb the approximation solution into an exact solution. Specifically, from the exact solution $f_{\mathrm{ex}}(t)$ to the Boltzmann equation \eqref{equ:Boltzmann}
\begin{equation}
\left\{
\begin{aligned}
&\pa_{t}f_{\mathrm{ex}}+v\cdot \nabla_{x}f_{\mathrm{ex}}=Q(f_{\mathrm{ex}},f_{\mathrm{ex}}),\\
&f_{\mathrm{ex}}(t)=f_{\mathrm{a}}(t)+f_{\mathrm{c}}(t).
\end{aligned}
\right.
\end{equation}
we have the equation for the correction term $f_{\mathrm{c}}$ that
\begin{equation}\label{equ:correction term,fc}
\left\{
\begin{aligned} \partial_t f_{\mathrm{c}} + v\cdot \nabla_x f_{\mathrm{c}} = & \pm Q^\pm(f_{\mathrm{c}},f_{ \mathrm{a}}) \pm
Q^\pm(f_{\mathrm{a}},f_{\mathrm{c}}) \pm Q^\pm(f_{\mathrm{c}},f_{\mathrm{c}
})-F_{\text{err}},\\
F_{\mathrm{err}}=&\pa_{t}f_{\mathrm{a}}+v\cdot \nabla_{x}f_{\mathrm{a}}+Q^{-}(f_{\mathrm{a}},f_{\mathrm{a}})-Q^{+}(f_{\mathrm{a}},f_{\mathrm{a}}).
 \end{aligned}
 \right.
\end{equation}
 To prove the existence of $f_{\mathrm{c}}$, we work with a $Z$-norm defined by \eqref{equ:z-norm} which is tailored to be stronger than the $ L_{v}^{2,r_{0}}H_{x}^{s_{0}}$ norms.
For the $Z$-norm, we are able to provide a closed bilinear estimate for the gain and loss terms in Lemma \ref{lemma:binlinear estimate}. Additionally, to work on the $Z$-norm space, we provide effective $Z$-norm bounds on the approximation solution $f_{\mathrm{a}}$, which we conclude in Proposition \ref{lemma:z-norm bounds on fa}, and then prove
the $Z$-norm error estimates on the error term $F_{\mathrm{err}}$ that
\begin{align}\label{equ:error bound,Ferr}
\bbn{\int_{\tau}^{t}e^{-(t-t_{0})v\cdot \nabla _{x}}F_{\text{err}
}(t_{0})\,dt_{0}}_{Z}\ll 1,
\end{align}
which we set up in Proposition \ref{lemma:bounds on ferr}.
Then by a perturbation argument in Proposition \ref{lemma:perturbation}, we prove that the correction term indeed satisfies the smallness property that
\begin{align}
\n{f_{\mathrm{c}}(t)}_{L^{\infty}([T_{*},0];Z)}\ll 1.
\end{align}
Finally in Section \ref{section:Proof of Illposedness}, we conclude the ill-posedness results.

\subsection{Norm Deflation of the Approximation Solution} \label{equ:Bounds on the approximation solution}
In the section, we get into the analysis of the construction of the approximation solution and its norm deflation property.
Following the analysis of a prototype approximation solution in \cite{chen2022well},\footnote{One could see \cite[Figure 1]{chen2022well} for a picture of the approximation solutions there. They look like bullets hitting a rock in \cite{chen2022well}. Our improvised and refined version is more like needles poking a rock through.} we decompose
\begin{align*}
f_{\mathrm{a}}(t)=f_{\mathrm{r}}(t)+f_{\mathrm{b}}(t).
\end{align*}

On the unit sphere, set $J\sim M^{d-1}N_{2}^{d-1}$ points $\lr{e_{j}}_{j=1}^{J}$, where the points $e_{j}$ are roughly equally
spaced. Let $P_{e_{j}}$ be the orthogonal projection onto the $1D$ subspace spanned by $e_{j}$ and $P_{e_{j}}^{\perp}$ denote the orthogonal projection onto
the orthogonal complement space $\lr{e_{j}}^{\perp}$.
Set
\begin{equation}
f_{\mathrm{b}}(t,x,v)=\frac{M^{\frac{d-1}{2}-s}}{N_{2}^{d+\ga}}\sum_{j=1}^{J}K_{j}(x-vt)I_{j}(v),
\end{equation}
where
\begin{align*}
K_{j}(x)=\chi (MP_{e_{j}}^{\perp }x)\chi (\frac{P_{e_{j}}x}{N_{2}}),\quad
I_{j}(v)=\chi(MP_{e_{j}}^{\perp }v)\chi (\frac{10P_{e_{j}}(v-N_{2}e_{j})}{N_{2}}).
\end{align*}
In fact, $f_{\mathrm{b}}(t,x,v)$ is a linear solution to the transport equation:
\begin{align}\label{equ:fb equation}
\partial _{t}f_{\mathrm{b}}+v\cdot \nabla _{x}f_{\mathrm{b}}=0.
\end{align}

Let $f_{\mathrm{r}}(t,x,v)$ be the solution to a drift-free linearized Boltzmann
equation:
\begin{equation}\label{equ:fr equation}
\partial _{t}f_{\mathrm{r}}(t,x,v)=-f_{\mathrm{r}}(t,x,v)\int \frac{f_{\mathrm{b}}(t,x,u)}{|u-v|^{-\ga}}\,du=-Q^{-}(f_{\mathrm{r}},f_{\mathrm{b}}),
\end{equation}
with initial data $f_{\mathrm{r}}(0)=M^{\frac{d}{2}-s}N_{1}^{\frac{d}{2}} \chi (Mx)\chi(N_{1}v)$. Therefore, we write out
\begin{equation}\label{equ:fr}
	f_{\mathrm{r}}(t,x,v)=M^{\frac{d}{2}-s}N_{1}^{\frac{d}{2}}\exp \left[
	-\int_{0}^{t}\int\frac{f_{\mathrm{b}}(\tau,x,u)}{|u-v|^{-\ga}}\,du\,d\tau\right] \chi (Mx)\chi(N_{1}v).
\end{equation}

Recall that $0 \leq  s_{0}<\frac{d-1}{2}$, $\frac{1-d}{2}\leq \ga\leq 0$, and $r_{0}=\max\lr{0,s_{0}+\ga}$.
In what follows, the parameters are set by
\begin{align}
	&M\gg 1, \quad  N_{1}\geq N_{2}^{10}\geq  M^{100},\label{equ:condition,parameters}\\
&s=s_{0}+\frac{\ln\ln \ln M}{\ln M}, \label{equ:condition,s,s0}\\
&T_{*}= -M^{s-\frac{d-1}{2}}(\ln \ln \ln M).\label{equ:T*}
\end{align}

Next, we give the Sobolev norm estimates on $f_{\mathrm{b}}$, $f_{\mathrm{r}}$ and $f_{\mathrm{a}}$.

\begin{lemma}[Sobolev norm bounds on $f_{\mathrm{b}}$]\label{lemma:Hs,bounds on fb}
We have for $t\leq 0$,
\begin{align}\label{equ:critical norm estimate for fb}
\n{f_{\mathrm{b}}(t,x,v)}_{L_{v}^{2,r_{0}}H_{x}^{s_{0}}}\lesssim M^{s_{0}-s}N_{2}^{\max\lr{s_{0},-\ga}-\frac{d-1}{2}}\leq \frac{1}{\ln \ln M}.
\end{align}
\end{lemma}
\begin{proof}
Recall that
\begin{equation}
f_{\mathrm{b}}(t,x,v)=\frac{M^{\frac{d-1}{2}-s}}{N_{2}^{d+\ga}}\sum_{j=1}^{J}K_{j}(x-vt)I_{j}(v).
\end{equation}
Due to the $v$-support of $f_{\mathrm{b}}$,
the weight on $v$-variable produces a factor of $N_{2}^{r_{0}}$. Then expanding $f_{\mathrm{b}}$ gives that
\begin{align*}
\n{\lra{\nabla_{x}}^{s_{0}}f_{\mathrm{b}}(t,x,v)}_{L_{v}^{2,r_{0}}L_{x}^{2}}^{2}
\lesssim N_{2}^{2r_{0}}\frac{M^{d-1-2s}}{N_{2}^{2d+2\ga}}\bbn{\sum_{j=1}^{J}\lra{\nabla_{x}}^{s_{0}}K_{j}(x-vt)I_{j}(v)}_{L_{v}^{2}L_{x}^{2}}^{2}.
\end{align*}
Due to the disjointness of the $v$-support, we have
\begin{align*}
\n{\lra{\nabla_{x}}^{s_{0}}f_{\mathrm{b}}(t,x,v)}_{L_{v}^{2,r_{0}}L_{x}^{2}}^{2}&\lesssim N_{2}^{2r_{0}}\frac{M^{d-1-2s}}{N_{2}^{2d+2\ga}}\sum_{j=1}^{J}\n{\lra{\nabla_{x}}^{s_{0}}K_{j}(x-vt)I_{j}(v)}_{L_{v}^{2}L_{x}^{2}}^{2}\\
&\lesssim N_{2}^{2r_{0}}\frac{M^{d-1-2s}}{N_{2}^{2d+2\ga}}\sum_{j=1}^{J}\n{\lra{\nabla_{x}}^{s_{0}}K_{j}(x)}_{L_{x}^{2}}^{2}\n{I_{j}(v)}_{L_{v}^{2}}^{2}\\
&\lesssim N_{2}^{2r_{0}}\frac{M^{d-1-2s}}{N_{2}^{2d+2\ga}}(MN_{2})^{d-1}(M^{2s_{0}+1-d}N_{2})(M^{1-d}N_{2})\\
&= N_{2}^{2r_{0}}\frac{M^{2s_{0}-2s}}{N_{2}^{d-1+2\ga}},
\end{align*}
where in the second-to-last inequality we used that
\begin{align*}
\n{\lra{\nabla_{x}}^{s_{0}}K_{j}(x)}_{L_{x}^{2}}^{2}\lesssim &M^{2s_{0}}M^{1-d}N_{2},\quad \n{I_{j}(v)}_{L_{v}^{2}}^{2}\sim M^{1-d}N_{2}.
\end{align*}
Notice that
\begin{align*}
r_{0}=\max\lr{0,s_{0}+\ga},\quad M^{s-s_{0}}=\ln\ln M,\quad \max\lr{s_{0},-\ga}\leq \frac{d-1}{2}.
\end{align*}
Hence, we complete the proof of \eqref{equ:critical norm estimate for fb}.

\end{proof}

Before proceeding to the analysis of $f_{\mathrm{r}}$, we give a useful pointwise bound on $f_{\mathrm{b}}$.
\begin{lemma}[Pointwise estimate on $f_{\mathrm{b}}$]\label{lemma:pointwise estimate,fb}
Let $-\frac{1}{4}\leq t\leq 0$ and
\begin{equation*}
\beta(t,x,v) = \int_{0}^{t} \int \frac{f_{\mathrm{b}
}(t_{0},x,u)}{|u-v|^{-\ga}}\,du  dt_0\leq 0.
\end{equation*}
For $k=0,1,2$, we have the pointwise upper bound
\begin{align}\label{equ:upper bound,fr,k}
\big|\chi(N_{1}v) \nabla_{x}^{k}\beta(t,x,v)\big |\lesssim |t|M^{k+\frac{d-1}{2}-s}.
\end{align}
For the pointwise lower bound, we have
\begin{equation}\label{equ:upper lower bound,fr}
|\chi(N_{1}v)\beta(t,x,v) |\gtrsim  |t| M^{\frac{d-1}{2}-s} \chi(N_{1}v),\quad \text{for $|x|\leq M^{-1}$}.
\end{equation}
\end{lemma}
\begin{proof}
For $-\frac{1}{4}\leq t\leq 0$, given the constraints on the $u$-variable, we have
\begin{align}\label{equ:pointwise estimate on fb}
&\frac{M^{\frac{d-1}{2}-s}}{N_{2}^{d+\ga}}\sum_{j=1}^{J}\chi
	(10MP_{e_{j}}^{\perp }x)\chi (\frac{10P_{e_{j}}x}{N_{2}})\chi
	(MP_{e_{j}}^{\perp }u)\chi (\frac{10P_{e_{j}}(u-N_{2}e_{j})}{N_{2}}) \\
 \leq& f_{\mathrm{b}}(t,x,u)\leq \frac{M^{\frac{d-1}{2}-s}}{N_{2}^{d+\ga}}\sum_{j=1}^{J}\chi (\frac{%
		MP_{e_{j}}^{\perp }x}{10})\chi (\frac{P_{e_{j}}x}{10N_{2}})\chi
	(MP_{e_{j}}^{\perp }u)\chi (\frac{10P_{e_{j}}(u-N_{2}e_{j})}{N_{2}}).\notag
\end{align}
From the $v$-support and $u$-support, we have $|v|\sim N_{1}^{-1}$, $|u|\sim N_{2}$, and hence $|u-v|\sim N_{2}$. Then by using \eqref{equ:pointwise estimate on fb}, we get
\begin{align}\label{equ:pointwise estimate on fb,integral,final}
\chi(N_{1}v)\int \frac{f_{\mathrm{b}
}(t,x,u)}{|u-v|^{-\ga}}\,du
\sim& N_{2}^{\ga} \chi(N_{1}v)\int f_{\mathrm{b}
}(t,x,u)\,du \\
\sim&  N_{2}^{\ga} \frac{M^{\frac{d-1}{2}-s}}{N_{2}^{d+\ga}}M^{1-d}N_{2}\chi(N_{1}v)\sum_{j}^{J}\chi (MP_{e_{j}}^{\perp }x)\chi (\frac{P_{e_{j}}x}{N_{2}})\notag\\
=&\frac{M^{\frac{1-d}{2}-s}}{N_{2}^{d-1}}\chi(N_{1}v)\sum_{j}^{J}\chi (MP_{e_{j}}^{\perp }x)\chi (\frac{P_{e_{j}}x}{N_{2}})\notag.
\end{align}
Thus, for the upper bound \eqref{equ:upper bound,fr,k} with $k=0$, we use that $|J|\sim (MN_{2})^{d-1}$ to obtain
\begin{align*}
|\chi(N_{1}v) \beta(t,x,v)|\lesssim |t|\frac{M^{\frac{1-d}{2}-s}}{N_{2}^{d-1}} (MN_{2})^{d-1}\chi(N_{1}v)=|t|M^{\frac{d-1}{2}-s}\chi(N_{1}v).
\end{align*}
Notice that the upper bounds of estimates \eqref{equ:pointwise estimate on fb} and \eqref{equ:pointwise estimate on fb,integral,final} remain true with the extra factor $M^{k}$ if $\nabla_x^k$
is applied, for $k\geq 0$. Therefore, we conclude the pointwise upper bound \eqref{equ:upper bound,fr,k} on $\beta(t,x,v)$ for $k=0,1,2$.

For the lower bound \eqref{equ:upper lower bound,fr}, by noting that $\chi (MP_{e_{j}}^{\perp }x)\chi (\frac{P_{e_{j}}x}{N_{2}})=1$ for $|x|\leq M^{-1}$, we use \eqref{equ:pointwise estimate on fb,integral,final} again to get
\begin{align*}
|\chi(N_{1}v) \beta(t,x,v)|=& \chi(N_{1}v)\int_{t}^{0} \int \frac{f_{\mathrm{b}
}(t_{0},x,u)}{|u-v|^{-\ga}}\,du \, dt_0\\
\gtrsim& |t|\frac{M^{\frac{1-d}{2}-s}}{N_{2}^{d-1}} (MN_{2})^{d-1}\chi(N_{1}v) =|t|M^{\frac{d-1}{2}-s}\chi(N_{1}v),
\end{align*}
which completes the proof of \eqref{equ:upper lower bound,fr}.

\end{proof}

Now, we are able to give the upper and lower bounds on $f_{\mathrm{r}}$.
\begin{lemma}[Sobolev norm bounds on $f_{\rm r}$]\label{lemma:bounds on fr}
For $-\frac{1}{4} \leq t\leq 0$, we have the upper bound estimate
\begin{align}
&\n{f_{\mathrm{r}}(t)}_{L_{v}^{2,r_{0}}H_{x}^{s_{0}}}\lesssim M^{s_{0}-s} \exp[|t| M^{\frac{d-1}{2}-s}]\lra{|t|M^{\frac{d-1}{2}-s}},\label{equ:critical norm estimate for fr,upper bound}
\end{align}
and the lower bound estimate
\begin{align}
&\n{f_{\mathrm{r}}(t)}_{L_{v}^{2,r_{0}}H_{x}^{s_{0}}}\gtrsim M^{s_{0}-s} \exp[|t| M^{\frac{d-1}{2}-s}].\label{equ:critical norm estimate for fr,lower bound}
\end{align}
In particular, we have
\begin{align}
\n{f_{\mathrm{r}}(0)}_{L_{v}^{2,r_{0}}H_{x}^{s_{0}}}\lesssim& \frac{1}{\ln\ln M},\label{equ:upper bound,fr,t=0}\\
\n{f_{\mathrm{r}}(T_{*})}_{L_{v}^{2,r_{0}}H_{x}^{s_{0}}}\gtrsim& 1,\label{equ:lower bound,fr,T}
\end{align}
with $T_{*}= -M^{s-\frac{d-1}{2}}(\ln \ln \ln M)$.
\end{lemma}
\begin{proof}
Recall that
\begin{equation}\label{equ:fr,expression}
f_{\mathrm{r}}(t,x,v)=M^{\frac{d}{2}-s}N_{1}^{\frac{d}{2}}\exp \left[
-\beta(t,x,v)\right] \chi (Mx)\chi (
N_{1}v).
\end{equation}
Due to the $v$-support of $f_{\mathrm{r}}$, we can discard the weight on the $v$-variable.
 For upper bound estimate \eqref{equ:critical norm estimate for fr,upper bound} on $f_{\mathrm{r}}$, we use the pointwise upper bound \eqref{equ:upper bound,fr,k} to get
\begin{align}\label{equ:upper bound,fr,sobolev}
	&\n{\nabla_{x}f_{\rm{r}}(t)}_{L_{v}^{2,r_{0}}L_{x}^{2}}\\
	\leq&M^{1+\frac{d}{2}-s}N_{1}^{\frac{d}{2}}\n{\exp[-\beta(t,x,v)](\nabla\chi)(Mx)\chi(N_{1}v)}_{L_{v}^{2}L_{x}^{2}}\notag\\ &+M^{\frac{d}{2}-s}N_{1}^{\frac{d}{2}}\n{\nabla_{x}\beta(t,x,v)\exp[-\beta(t,x,v)]\chi(Mx)\chi(N_{1}v)}_{L_{v}^{2}L_{x}^{2}}\notag\\
	\lesssim&
	M^{1+\frac{d}{2}-s}N_{1}^{\frac{d}{2}}\exp[|t|M^{\frac{d-1}{2}-s}]\n{(\nabla\chi)(Mx)}_{L_{x}^{2}}\n{\chi(N_{1}v)}_{L_{v}^{2}}\notag\\
	&+M^{\frac{d}{2}-s}N_{1}^{\frac{d}{2}}\lra{|t|M^{1+\frac{d-1}{2}-s}}\exp[|t|M^{\frac{d-1}{2}-s}]\n{\chi(Mx)}_{L_{x}^{2}}\n{\chi(N_{1}v)}_{L_{v}^{2}}\notag\\
	\lesssim& M^{1-s}\exp[|t| M^{\frac{d-1}{2}-s}]\lra{|t|M^{\frac{d-1}{2}-s}}.\notag
\end{align}
In the same way, we also have
\begin{align*}
\n{f_{\mathrm{r}}(t)}_{L_{v}^{2,r_{0}} L_{x}^{2}}\lesssim M^{-s}\exp[|t| M^{\frac{d-1}{2}-s}].
\end{align*}
By the interpolation inequality, we obtain
\begin{align*}
\n{f_{\mathrm{r}}(t)}_{L_{v}^{2,r_{0}} H_x^{s_{0}}}\leq& \n{f_{\mathrm{r}}(t)}_{L_{v}^{2,r_{0}} H_{x}^{1}}^{s_{0}}
\n{f_{\mathrm{r}}(t)}_{L_{v}^{2,r_{0}} L_{x}^{2}}^{1-s_{0}}
\lesssim
M^{s_{0}-s} \exp[|t| M^{\frac{d-1}{2}-s}]\lra{|t|M^{\frac{d-1}{2}-s}}.
\end{align*}

For the lower bound estimate \eqref{equ:critical norm estimate for fr,lower bound} on $f_{\mathrm{r}}$,
we use the Sobolev inequality and lower bound estimate \eqref{equ:upper lower bound,fr} to obtain
\begin{align*}
&\n{\lra{\nabla_{x}}^{s_{0}}f_{\mathrm{r}}(t,x,v)}_{L_{v}^{2,r_{0}}L_{x}^{2}}\\
\gtrsim&\n{f_{\mathrm{r}}(t,x,v)}_{L_{v}^{2,r_{0}}L_{x}^{\frac{2d}{d-2s_{0}}}}\\
\gtrsim&M^{\frac{d}{2}-s}N_{1}^{\frac{d}{2}}\n{\exp \left[
-\beta(t,x,v)\right] \chi (Mx)\chi (
N_{1}v)}_{L_{v}^{2}L_{x}^{\frac{2d}{d-2s_{0}}}}\\
\gtrsim& M^{\frac{d}{2}-s}N_{1}^{\frac{d}{2}}\exp[|t| M^{\frac{d-1}{2}-s}] \n{\chi(Mx)}_{L_{x}^{\frac{2d}{d-2s_{0}}}}\n{\chi(N_{1}v)}_{L_{v}^{2}}\\
\gtrsim& M^{s_{0}-s}\exp[|t| M^{\frac{d-1}{2}-s}].
\end{align*}
Hence, we have done the proof of estimate $(\ref{equ:critical norm estimate for fr,lower bound})$.

Inserting in $|T_{*}|= M^{s-\frac{d-1}{2}}(\ln \ln \ln M)$ and $M^{s_{0}-s}=\frac{1}{\ln\ln M}$, we have
\begin{align*}
\n{f_{\mathrm{r}}(T_{*})}_{L_{v}^{2,r_{0}}H_{x}^{s_{0}}}\gtrsim  M^{s_{0}-s}\exp[|T_{*}| M^{\frac{d-1}{2}-s}]\gtrsim 1,
\end{align*}
which completes the proof of \eqref{equ:lower bound,fr,T}.
\end{proof}
\begin{remark}\label{remark:lower bound,kernels}
The lower bound estimate \eqref{equ:critical norm estimate for fr,lower bound} on $f_{\mathrm{r}}(t)$ also holds for the kernel ($d=3$)
\begin{align}
B(u-v,\omega)=\lrs{1_{\lr{|u-v|\leq 1}}|u-v|+1_{\lr{|u-v|\geq 1}}|u-v|^{-1}} \textbf{b}(\frac{u-v}{|u-v|}\cdot \omega).
\end{align}
 Indeed, in the proof of the lower bound estimate \eqref{equ:upper lower bound,fr}, the term $1_{\lr{|u-v|\leq 1}}|u-v|$ would vanish due to that $|u-v|\sim N_{2}\gg 1$.
\end{remark}

In the end, we conclude the norm deflation property of the approximation solution $f_{\mathrm{a}}$.
\begin{proposition}[Norm deflation of $f_{\mathrm{a}}$]\label{lemma:norm deflation of fa}
Let $T_{*}= -M^{s-\frac{d-1}{2}}(\ln \ln \ln M)$. We have
\begin{align}
	\n{f_{\mathrm{a}}(0)}_{L_v^{2,r_{0}} H_x^{s_{0}}}\lesssim& \frac{1}{\ln \ln M}\ll 1,\label{equ:critical norm estimate for fa,t=0}\\
	\n{f_{\mathrm{a}}(T_{*})}_{L_v^{2,r_{0}} H_x^{s_{0}}}\gtrsim&  1.\label{equ:critical norm estimate for fa,t=T}
\end{align}

\end{proposition}
\begin{proof}

Since $f_{\mathrm{b}}$ and $f_{\mathrm{r}}$ have disjoint velocity supports, we get
\begin{align}
\n{f_{\mathrm{a}}(t)}_{L_v^{2,r_{0}} H_x^{s_{0}}}\sim \n{f_{\mathrm{r}}(t)}_{L_v^{2,r_{0}} H_x^{s_{0}}}+\n{f_{\mathrm{b}}(t)}_{L_v^{2,r_{0}} H_x^{s_{0}}}.
\end{align}
Then by estimate \eqref{equ:critical norm estimate for fb} on $f_{\mathrm{b}}$ in Lemma \ref{lemma:Hs,bounds on fb} and estimates \eqref{equ:upper bound,fr,t=0}--\eqref{equ:lower bound,fr,T} on $f_{\mathrm{r}}$ in Lemma \ref{lemma:bounds on fr}, we arrive at estimates \eqref{equ:critical norm estimate for fa,t=0} and \eqref{equ:critical norm estimate for fa,t=T}.

\end{proof}

\subsubsection{Discussion on the $L^{1}$-based space and hard potentials}\label{section:Discussion on the $L^{1}$-based space and hard potentials}

The equation \eqref{equ:Boltzmann} is invariant under the scaling
\begin{align}
f_{\la}(t,x,,v)=\la^{\al+(d-1+\ga)\be}f(\la^{\al-\be}t,\la^{\al}x,\la^{\be}v),
\end{align}
for any $\al$, $\be\in \R$ and $\la>0$. Then
\begin{align*}
\n{|\nabla_{x}|^{s}|v|^{r}f_{\la}}_{L_{xv}^{1}}=\la^{^{\al+(d-1+\ga)\be}}\la^{\al s-\be r}\la^{-d\al-d\be}
\n{|\nabla_{x}|^{s}|v|^{r}f}_{L_{xv}^{1}},
\end{align*}
which gives the $L^{1}$-based scaling-critical index
\begin{align}
s_{1}=d-1,\quad  r_{1}=1+\ga.
\end{align}
In the $L^{1}$ setting, we construct the approximation solution $f_{a,1}=f_{\mathrm{b},1}+f_{\mathrm{r},1}$, where
\begin{align*}
f_{\mathrm{b},1}(t,x,v)=&\frac{M^{d-1-s}}{N_{2}^{d+2+\ga}}\sum_{j=1}^{J}K_{j}(x-vt)I_{j}(v),\\
f_{\mathrm{r,1}}(t,x,v)=&M^{d-s}N_{1}^{d}\exp \left[
-\beta(t,x,v)\right] \chi (Mx)\chi (
N_{1}v).
\end{align*}
Repeating the proof of estimates \eqref{equ:critical norm estimate for fb} and \eqref{equ:critical norm estimate for fr,lower bound}, we also have
\begin{align}
\n{\lra{\nabla}^{s_{0}}f_{\mathrm{b},1}}_{L_{v}^{1,r_{1}}L_{x}^{1}}\lesssim& M^{s_{0}-s},\\
\n{\lra{\nabla}^{s_{0}}f_{\mathrm{r},1}}_{L_{v}^{1,r_{1}}L_{x}^{1}}\gtrsim& M^{s_{0}-s}\exp[|t|N_{2}^{-2}M^{d-1-s}].
\end{align}
If $s_{0}<s_{1}=d-1$, a similar mechanism of norm deflation could be possible in the $L^{1}$ setting.

For the hard potential case that $\ga> 0$, the norm deflation of the approximation solution $f_{\mathrm{a}}(t)$ also holds. But, to perturb it into the exact solution, it requires a much more different work space to prove the error bounds in Proposition \ref{lemma:bounds on ferr} and provide a closed estimate in Lemma \ref{lemma:binlinear estimate}. We leave the problem for future work.

\subsection{$Z$-norm Bounds on the Approximation Solution}\label{section:$Z$-norm Bounds on the Approximation Solution}
To perturb the approximation solution $f_{\mathrm{a}}(t)$ into an exact solution $f_{\mathrm{ex}}(t)$, we need to prove the existence of a small correction term $f_{\mathrm{c}}(t)$. As it satisfies a more complicated equation \eqref{equ:correction term,fc}, some terms of \eqref{equ:correction term,fc} cannot be effectively
treated using Strichartz estimates like \eqref{equ:lwp,bilinear estimate}. Hence, we tailor a $Z$-norm to provide a closed estimate for the gain and loss terms, that is,
\begin{equation}\label{equ:closed z-norm estimate}
\|Q^\pm (f_1,f_2) \|_{Z} \lesssim \|f_1\|_{Z} \|f_2\|_{Z},
\end{equation}
where the $Z$-norm is given by
\begin{align}\label{equ:z-norm}
\n{f(t)}_{Z}=&M^{\frac{d-3}{2}}\n{\nabla_{x} f(t)}_{L_{v}^{2,r_{0}}L_{x}^{2}}
+M^{\frac{d-1}{2}}\n{f(t)}_{L_{v}^{2,r_{0}}L_{x}^{2}}+N_{2}^{\ga}\Vert f(t)\Vert _{L_{v}^{1}L_{x}^{\infty}}\\
&+N_{2}^{\frac{2d}{5}+\ga}\Vert f(t)\Vert _{L_{v}^{\frac{5}{3}}L_{x}^{\infty}}
+M^{-1}N_{2}^{\ga}\n{\nabla_{x}f(t)}_{L_{v}^{1}L_{x}^{\infty}}+
M^{-1}N_{2}^{\frac{2d}{5}+\ga}\n{\nabla_{x}f(t)}_{L_{v}^{\frac{5}{3}}L_{x}^{\infty}}.\notag
\end{align}
The closed estimate \eqref{equ:closed z-norm estimate} which we will prove in Section \ref{section:Bounds on the Correction Term} indeed plays a key role in the perturbation argument.
In the section, we give $Z$-norm bounds on the approximation solution $f_{\mathrm{a}}=f_{\mathrm{r}}+f_{\mathrm{b}}$, which will be used to control the error term $F_{\mathrm{err}}$.

\begin{lemma}[$Z$-norm bounds on $f_{\mathrm{b}}$]\label{lemma:bounds on fb}
For the $Z$-norm, we have
\begin{align}\label{equ:z-norm estimate for fb}
\n{f_{\mathrm{b}}(t)}_{L^{\infty}([T_{*},0];Z)}\lesssim  M^{\frac{d-1}{2}-s}.
\end{align}

\end{lemma}
\begin{proof}
\textbf{The $M^{\frac{d-3}{2}}\Vert \nabla_{x}\bullet\Vert
_{L_{v}^{2,r_{0}}L_{x}^{2}}$ and $M^{\frac{d-1}{2}}\Vert \bullet\Vert
_{L_{v}^{2,r_{0}}L_{x}^{2}}$ estimates.}

 This can be done in the same way as estimate \eqref{equ:critical norm estimate for fb} with the regularity index $s_{0}$ replaced by $1$ and $0$. Therefore, we obtain
\begin{align}\label{equ:z-norm estimate for fb,L2}
M^{\frac{d-3}{2}}\n{\nabla_{x} f_{\mathrm{b}}}_{L_{v}^{2}L_{x}^{2}}\lesssim & M^{\frac{d-1}{2}-s}N_{2}^{\max\lr{s_{0},-\ga}-\frac{d-1}{2}}\leq M^{\frac{d-1}{2}-s},\\ M^{\frac{d-1}{2}}\n{f_{\mathrm{b}}}_{L_{v}^{2}L_{x}^{2}}\lesssim &M^{\frac{d-1}{2}-s}N_{2}^{\max\lr{s_{0},-\ga}-\frac{d-1}{2}}\leq M^{\frac{d-1}{2}-s}.
\end{align}

\textbf{The $N_{2}^{\ga}\Vert \bullet \Vert _{L_{v}^{1}L_{x}^{\infty}}$ and $N_{2}^{\frac{2d}{5}+\ga}\Vert \bullet \Vert _{L_{v}^{\frac{5}{3}}L_{x}^{\infty}}$ estimates.}
\begin{align}\label{equ:z-norm estimate for fb,Lv1}
N_{2}^{\ga}\n{f_{\mathrm{b}}}_{L_{v}^{1}L_{x}^{\infty}}
\lesssim &N_{2}^{\ga}\frac{M^{\frac{d-1}{2}-s}}{N_{2}^{d+\ga}}
\bbn{\sum_{j=1}^{J}\n{K_{j}(x-vt)}_{L_{x}^{\infty}}I_{j}(v)}_{L_{v}^{1}}\\
\lesssim &\frac{M^{\frac{d-1}{2}-s}}{N_{2}^{d}}
\bbn{\sum_{j=1}^{J}I_{j}(v)}_{L_{v}^{1}}\notag\\
\lesssim&\frac{M^{\frac{d-1}{2}-s}}{N_{2}^{d}}N_{2}^{d}=M^{\frac{d-1}{2}-s},\notag
\end{align}
where in the last inequality we have used that
\begin{align*}
\sum_{j=1}^{J}I_{j}(v)\sim  1_{\lr{\frac{9N_{2}}{10}\leq|v|\leq \frac{11N_{2}}{10}}}(v),\quad \bbn{1_{\lr{\frac{9N_{2}}{10}\leq|v|\leq \frac{11N_{2}}{10}}}(v)}_{L_{v}^{1}}\lesssim N_{2}^{d}.
\end{align*}
In the same way, we also have
\begin{align*}
N_{2}^{\frac{2d}{5}+\ga}\n{f_{\mathrm{b}}}_{L_{v}^{\frac{5}{3}}L_{x}^{\infty}}
\lesssim &N_{2}^{\frac{2d}{5}+\ga}\frac{M^{\frac{d-1}{2}-s}}{N_{2}^{d+\ga}}
\bbn{\sum_{j=1}^{J}\n{K_{j}(x-vt)}_{L_{x}^{\infty}}I_{j}(v)}_{L_{v}^{\frac{5}{3}}}\\
\lesssim &N_{2}^{\frac{2d}{5}}\frac{M^{\frac{d-1}{2}-s}}{N_{2}^{d}}
\bbn{\sum_{j=1}^{J}I_{j}(v)}_{L_{v}^{\frac{5}{3}}}\\
\lesssim&N_{2}^{\frac{2d}{5}}\frac{M^{\frac{d-1}{2}-s}}{N_{2}^{d}}N_{2}^{\frac{3d}{5}}=M^{\frac{d-1}{2}-s}.
\end{align*}

The same bound is obtained for $M^{-1}N_{2}^{\ga}\Vert \nabla_{x}f_{\mathrm{b}} \Vert _{L_{v}^{1}L_{x}^{\infty}}$ and $M^{-1}N_{2}^{\frac{2d}{5}+\ga}\Vert \nabla_{x}f_{\mathrm{b}} \Vert _{L_{v}^{\frac{5}{3}}L_{x}^{\infty}}$ with one $x$-derivative producing a factor of $M$. Therefore, we complete the proof of the $Z$-norm estimate \eqref{equ:z-norm estimate for fb}.

\end{proof}

\begin{lemma}[$Z$-norm bounds on $f_{\mathrm{r}}$]\label{lemma:z-norm bounds on fr}
For $T_{*}\leq t\leq 0$, we have
\begin{align}\label{equ:z-norm estimate for fr}
\n{f_{\mathrm{r}}(t)}_{Z}\lesssim M^{\frac{d-1}{2}-s} \exp[|t| M^{\frac{d-1}{2}-s}]\lra{|t|M^{\frac{d-1}{2}-s}}.
\end{align}
In particular,
\begin{align}\label{equ:z-norm bounds for fr}
\n{f_{\mathrm{r}}(t)}_{L^{\infty}([T_{*},0];Z)}\lesssim M^{\frac{d-1}{2}-s} (\ln \ln M)^{2}.
\end{align}
\end{lemma}
\begin{proof}
Recall
\begin{equation}
f_{\mathrm{r}}(t,x,v)=M^{\frac{d}{2}-s}N_{1}^{\frac{d}{2}}\exp \left[
-\beta(t,x,v)\right] \chi (Mx)\chi (
N_{1}v).
\end{equation}

\textbf{The $M^{\frac{d-3}{2}}\Vert \nabla_{x}\bullet\Vert
_{L_{v}^{2,r_{0}}L_{x}^{2}}$ and $M^{\frac{d-1}{2}}\Vert \bullet\Vert
_{L_{v}^{2,r_{0}}L_{x}^{2}}$ estimates.}

The weight on $v$-variable plays no role due to the $v$-support set, so we can discard it.
By the pointwise upper bound \eqref{equ:upper bound,fr,k}, we get
\begin{align*}
	&M^{\frac{d-3}{2}}\n{\nabla_{x}f_{\rm{r}}(t)}_{L_{v}^{2,r_{0}}L_{x}^{2}}\\
	\leq&M^{\frac{d-3}{2}}M^{1+\frac{d}{2}-s}N_{1}^{\frac{d}{2}}\n{\exp[-\beta(t,x,v)](\nabla\chi)(Mx)\chi(N_{1}v)}_{L_{v}^{2}L_{x}^{2}}\\ &+M^{\frac{d-3}{2}}M^{\frac{d}{2}-s}N_{1}^{\frac{d}{2}}\n{\nabla_{x}\beta(t,x,v)\exp[-\beta(t,x,v)]\chi(Mx)\chi(N_{1}v)}_{L_{v}^{2}L_{x}^{2}}\\
	\lesssim&
	M^{\frac{d-1}{2}}M^{\frac{d}{2}-s}N_{1}^{\frac{d}{2}}\exp[|t|M^{\frac{d-1}{2}-s}]\n{(\nabla\chi)(Mx)}_{L_{x}^{2}}\n{\chi(N_{1}v)}_{L_{v}^{2}}\\
	&+M^{\frac{d-3}{2}}M^{\frac{d}{2}-s}N_{1}^{\frac{d}{2}}\lra{|t|M^{1+\frac{d-1}{2}-s}}\exp[|t|M^{\frac{d-1}{2}-s}]\n{\chi(Mx)}_{L_{x}^{2}}\n{\chi(N_{1}v)}_{L_{v}^{2}}\\
	\lesssim& M^{\frac{d-1}{2}-s}\exp[|t| M^{\frac{d-1}{2}-s}]\lra{|t|M^{\frac{d-1}{2}-s}}.
\end{align*}
The $M^{\frac{d-1}{2}}\n{f_{\mathrm{r}}}_{L_{v}^{2,r_{0}}L_{x}^{2}}$ estimate can be handled in the same way.

\textbf{The $N_{2}^{\ga}\Vert \bullet \Vert _{L_{v}^{1}L_{x}^{\infty}}$ and $M^{-1}N_{2}^{\ga}\Vert \nabla_{x}\bullet\Vert _{L_{v}^{1}L_{x}^{\infty}}$ estimates.}

We only need to treat the $M^{-1}N_{2}^{\ga}\Vert \nabla_{x}\bullet\Vert _{L_{v}^{1}L_{x}^{\infty}}$ norm, as the $N_{2}^{\ga}\Vert \bullet\Vert _{L_{v}^{1}L_{x}^{\infty}}$ norm can be dealt with in a similar way.
We use the pointwise upper bound \eqref{equ:upper bound,fr,k} to obtain
\begin{align}\label{equ:fr,Z-norm,L1}
&M^{-1}N_{2}^{\ga}\n{\nabla_{x}f_{\rm{r}}(t)}_{L_{v}^{1}L_{x}^{\infty}}\\
	\leq& M^{-1}N_{2}^{\ga} M^{1+\frac{d}{2}-s}N_{1}^{\frac{d}{2}}\n{\exp[-\beta(t,x,v)](\nabla\chi)(Mx)\chi(N_{1}v)}_{L_{v}^{1}L_{x}^{\infty}}\notag\\ &+M^{-1}N_{2}^{\ga}M^{\frac{d}{2}-s}N_{1}^{\frac{d}{2}}\n{\nabla_{x}\beta(t,x,v)\exp[-\beta(t,x,v)]\chi(Mx)\chi(N_{1}v)}_{L_{v}^{1}L_{x}^{\infty}}\notag\\
	\lesssim&
	M^{-1}N_{2}^{\ga}M^{1+\frac{d}{2}-s}N_{1}^{\frac{d}{2}}\exp[|t|M^{\frac{d-1}{2}-s}]\n{(\nabla\chi)(Mx)}_{L_{x}^{\infty}}\n{\chi(N_{1}v)}_{L_{v}^{1}}\notag\\
	&+M^{-1}N_{2}^{\ga}M^{\frac{d}{2}-s}N_{1}^{\frac{d}{2}}\lra{|t|M^{1+\frac{d-1}{2}-s}}\exp[|t|M^{\frac{d-1}{2}-s}]\n{\chi(Mx)}_{L_{x}^{\infty}}\n{\chi(N_{1}v)}_{L_{v}^{1}}\notag\\
	\lesssim& M^{-1}N_{2}^{\ga} M^{1+\frac{d}{2}-s}N_{1}^{-\frac{d}{2}}\exp[|t| M^{\frac{d-1}{2}-s}]\lra{|t|M^{\frac{d-1}{2}-s}}\notag\\
\lesssim& N_{1}^{-\frac{d}{2}}N_{2}^{\ga}M^{\frac{d}{2}-s}\exp[|t| M^{\frac{d-1}{2}-s}]\lra{|t|M^{\frac{d-1}{2}-s}}.\notag
\end{align}
This bound is enough as it carries the smallness factor $N_{1}^{-\frac{d}{2}}$.

\textbf{The $N_{2}^{\frac{2d}{5}+\ga}\Vert \bullet \Vert _{L_{v}^{\frac{5}{3}}L_{x}^{\infty}}$ and $M^{-1}N_{2}^{\frac{2d}{5}+\ga}\Vert \nabla_{x}\bullet\Vert _{L_{v}^{\frac{5}{3}}L_{x}^{\infty}}$ estimates.}

These two norms can be controlled in the same manner as \eqref{equ:fr,Z-norm,L1} with the $L_{v}^{1}$ norm replaced by the
$L_{v}^{\frac{5}{3}}$ norm. As a result, we also have
\begin{align}
N_{2}^{\frac{2d}{5}+\ga} \n{f_{\rm{r}}(t)}_{L_{v}^{\frac{5}{3}}L_{x}^{\infty}}\lesssim& N_{1}^{-\frac{d}{10}}N_{2}^{\frac{2d}{5}+\ga} M^{\frac{d}{2}-s}\exp[|t| M^{\frac{d-1}{2}-s}]\lra{|t|M^{\frac{d-1}{2}-s}},\label{equ:vfr,L35}\\
M^{-1}N_{2}^{\frac{2d}{5}+\ga} \n{\nabla_{x}f_{\rm{r}}(t)}_{L_{v}^{\frac{5}{3}}L_{x}^{\infty}}\lesssim& N_{1}^{-\frac{d}{10}}N_{2}^{\frac{2d}{5}+\ga}M^{\frac{d}{2}-s}\exp[|t| M^{\frac{d-1}{2}-s}]\lra{|t|M^{\frac{d-1}{2}-s}}.\label{equ:vfr,L35,H1}
\end{align}
By the condition \eqref{equ:condition,parameters} that $N_{1}\geq N_{2}^{10}\geq M^{100}$, it is sufficient to obtain the desired bound. Thus, we complete the proof of \eqref{equ:z-norm estimate for fr}.

Inserting in $|T_{*}|=M^{s-\frac{d-1}{2}}(\ln \ln \ln M)$ and $M^{s-s_{0}}=\ln \ln M$, we obtain
\begin{align*}
\n{f_{\mathrm{r}}(t)}_{L^{\infty}([T_{*},0];Z)} \lesssim M^{\frac{d-1}{2}-s}(\ln \ln M)^{2},
\end{align*}
which completes the proof of \eqref{equ:z-norm bounds for fr}.
\end{proof}

To the end, we conclude the $Z$-norm bounds on $f_{\mathrm{a}}=f_{\mathrm{r}}+f_{\mathrm{b}}$.
\begin{proposition}[$Z$-norm bounds on $f_{\mathrm{a}}$]\label{lemma:z-norm bounds on fa}
For the $Z$-norm,
\begin{align}\label{equ:z-norm estimate for fa}
\n{f_{\mathrm{a}}(t)}_{L^{\infty}([T_{*},0];Z)}\lesssim M^{\frac{d-1}{2}-s} (\ln \ln M)^{2}.
\end{align}
\end{proposition}
\begin{proof}
By the triangle inequality, we have
\begin{align*}
	\n{f_{\mathrm{a}}(t)}_{Z}\lesssim \n{f_{\mathrm{r}}(t)}_{Z}+\n{f_{\mathrm{b}}(t)}_{Z}.
\end{align*}
Then combining estimate \eqref{equ:z-norm estimate for fb} on $f_{\mathrm{b}}$ and estimate \eqref{equ:z-norm estimate for fr} on $f_{\mathrm{r}}$, we complete the proof of estimate \eqref{equ:z-norm estimate for fa}.
\end{proof}

\subsection{$Z$-norm Bounds on the Error Terms} \label{section:Bounds on the Error Terms}
In the section, we give the $Z$-norm bounds on the error term $F_{\mathrm{err}}$.
Recall the error term
\begin{align*}
F_{\text{err}}& =\partial _{t}f_{\mathrm{a}}+v\cdot \nabla
_{x}f_{\mathrm{a}}+Q^{-}(f_{\mathrm{a}},f_{\mathrm{a}})-Q^{+}(f_{\mathrm{a}},f_{
\mathrm{a}}) \\ & =v\cdot \nabla
_{x}\,f_{\mathrm{r}}-Q^{+}(f_{\mathrm{r}},f_{\mathrm{b}})\mp Q^{\pm}(f_{\mathrm{b}},f_{\mathrm{r}})\mp Q^{\pm}(f_{\mathrm{r}},f_{
\mathrm{r}})\mp Q^{\pm}(f_{\mathrm{b}},f_{\mathrm{b}}),
\end{align*}
and thus the estimate on $F_{\mathrm{err}}$ highly relies on the $Z$-norm bounds of $f_{\mathrm{r}}$ and $f_{\mathrm{b}}$.
Recall the estimate \eqref{equ:z-norm estimate for fb,L2} in Lemma \ref{lemma:bounds on fb} that
\begin{align*}
M^{\frac{d-3}{2}}\n{\nabla_{x} f_{\mathrm{b}}}_{L_{v}^{2,r_{0}}L_{x}^{2}}\lesssim M^{\frac{d-1}{2}-s}N_{2}^{\max\lr{s_{0},-\ga}-\frac{d-1}{2}},\\
 M^{\frac{d-1}{2}}\n{f_{\mathrm{b}}}_{L_{v}^{2,r_{0}}L_{x}^{2}}\lesssim M^{\frac{d-1}{2}-s}N_{2}^{\max\lr{s_{0},-\ga}-\frac{d-1}{2}}.
\end{align*}
For the case $\ga \in(\frac{1-d}{2},0]$, the extra smallness comes from the factor $N_{2}^{\max\lr{s_{0},-\ga}-\frac{d-1}{2}}$ as we have required that $s_{0}<\frac{d-1}{2}$ and $N_{2}\gg M$. Thus, it is enough to deal with the hardest endpoint case that $\ga=\frac{1-d}{2}$, in which the $M^{\frac{d-3}{2}}\Vert \nabla_{x}\bullet\Vert
_{L_{v}^{2,r_{0}}L_{x}^{2}}$ and $M^{\frac{d-1}{2}}\Vert \bullet\Vert
_{L_{v}^{2,r_{0}}L_{x}^{2}}$ norms of $f_{\mathrm{b}}$ are the order of $M^{\frac{d-1}{2}-s}$ and hence would not give any smallness for $s<\frac{d-1}{2}$.
Additionally, we only need to prove the $d=3$ case as the $d=2$ case follows from a similar way.

In the section, we set $d=3$, $\ga=-1$ and hence $r_{0}=0$, for which the $Z$-norm is
\begin{align}
\n{f(t)}_{Z}=&\n{\nabla_{x} f(t)}_{L_{v}^{2}L_{x}^{2}}
+M\n{f(t)}_{L_{v}^{2}L_{x}^{2}}+N_{2}^{-1}\Vert f(t)\Vert _{L_{v}^{1}L_{x}^{\infty}}\\
&+N_{2}^{\frac{1}{5}}\Vert f(t)\Vert _{L_{v}^{\frac{5}{3}}L_{x}^{\infty}}
+M^{-1}N_{2}^{-1}\n{\nabla_{x}f(t)}_{L_{v}^{1}L_{x}^{\infty}}+
M^{-1}N_{2}^{\frac{1}{5}}\n{\nabla_{x}f(t)}_{L_{v}^{\frac{5}{3}}L_{x}^{\infty}}.\notag
\end{align}
The following is the main result about the $Z$-norm bounds on the error term $F_{\mathrm{err}}$.

\begin{proposition}[$Z$-norm bounds on $F_{\text{err}}$]\label{lemma:bounds on ferr}
 For $T_{*}\leq  \tau\leq t \leq 0$,
\begin{equation}\label{equ:Ferr_bound}
\bbn{\int_{\tau}^{t}e^{-(t-t_{0})v\cdot \nabla _{x}}F_{\text{err}
}(t_{0})\,dt_{0}}_{Z}\lesssim M^{-1}.
\end{equation}
\end{proposition}

We deal with all of the terms in the following separate sections. In section \ref{section:Analysis of term1}, we give estimates on the term $v \cdot \nabla_{x}f_{\mathrm{r}}$. In section \ref{section:Analysis of term2}, we handle the bilinear terms which contain $f_{\mathrm{r}}$. Finally, we deal with $Q^{\pm}(f_{\mathrm{b}},f_{\mathrm{b}})$ in sections \ref{section:Analysis of Q-}--\ref{section:Analysis of Q+}, which are the most intricate parts.

The estimates are mainly achieved by
moving the $t_{0}$ integration to the outside as follows:
\begin{equation*}
	\bbn{ \int_{\tau }^{t}e^{-(t-t_{0})v\cdot \nabla _{x}}F_{\text{err}%
	}(t_{0})\,dt_{0}}_{Z}\lesssim |T_{*}|\n{F_{\text{err}}}_{L_{t}^{\infty }Z}.
\end{equation*}
The only exception is the treatment of the bound on $L_{v}^{1}L_{x}^{\infty
} $ of $Q^{\pm }(f_{\mathrm{b}},f_{\mathrm{b}})$, where a substantial gain is captured by
carrying out the $t_{0}$ integration first.

\subsubsection{Analysis of $v\cdot \nabla_{x}f_{\mathrm{r}}$} \label{section:Analysis of term1}
\begin{lemma}
For $T_{*}\leq  \tau\leq t \leq 0$,
\begin{align}
\bbn{ \int_{\tau }^{t}e^{-(t-t_{0})v\cdot \nabla_{x}}(v\cdot \nabla_{x}f_{\mathrm{r}})(t_{0})\,dt_{0}} _{Z}\lesssim M^{-1}.
\end{align}
\end{lemma}
\begin{proof}
As we have required that $N_{1}\gg M$ in \eqref{equ:condition,parameters}, the desired decay bound is achieved provided the upper bound carries the smallness factor $N_{1}^{-\delta}$ for some $\delta>0$.
	
\textbf{The $\Vert \nabla_{x}\bullet \Vert_{L_{v}^{2}L_{x}^{2}}$ and $M\Vert \bullet \Vert_{L_{v}^{2}L_{x}^{2}}$ estimates.}

It suffices to deal with the $\Vert \nabla_{x}\bullet \Vert_{L_{v}^{2}L_{x}^{2}}$ norm, as the estimate for the $M\Vert \bullet \Vert_{L_{v}^{2}L_{x}^{2}}$ norm follows the same way.
Noting that $f_{\mathrm{r}}$ is supported on $\lr{|v|\lesssim N_{1}^{-1}}$, we have
\begin{align}\label{equ:v term,H1}
\n{\nabla_{x} (v\cdot \nabla_{x} f_{\mathrm{r}})}_{L_{v}^{2}L_{x}^{2}}\lesssim&
N_{1}^{-1}\n{\Delta_{x}f_{\mathrm{r}}}_{L_{v}^{2}L_{x}^{2}}
\lesssim N_{1}^{-1}M^{2-s}\exp[|t| M^{1-s}]\lra{|t|M^{1-s}}^{2},
\end{align}
where the last inequality follows from the proof of \eqref{equ:upper bound,fr,sobolev} with one $x$-derivative producing a factor of $M$.
We then insert in $|T_{*}|=M^{s-1}(\ln\ln\ln M)$ to get
\begin{align*}
	&\bbn{\nabla_{x} \int_{\tau}^{t}e^{-(t-t_{0})v\cdot \nabla _{x}}(v\cdot \nabla_{x}f)(t_{0})\,dt_{0}}_{L_{v}^{2}L_{x}^{2}}\\
	\lesssim& |T_{*}| \sup_{t_{0}\in[T_{*},0]}\n{\nabla_{x}(v\cdot \nabla_{x} f_{\mathrm{r}})}_{L_{v}^{2}L_{x}^{2}}\\
	\lesssim&M^{s-1}(\ln \ln \ln M) N_{1}^{-1}M^{2-s}(\ln \ln M)^{3}\\
\lesssim&  N_{1}^{-1}M (\ln \ln M)^{4}.
\end{align*}

\textbf{The $N_{2}^{-1}\Vert \bullet \Vert _{L_{v}^{1}L_{x}^{\infty}}$ and $M^{-1}N_{2}^{-1}\Vert \nabla_{x}\bullet\Vert _{L_{v}^{1}L_{x}^{\infty}}$ estimates.}

We only need to treat the $M^{-1}N_{2}^{-1}\Vert \nabla_{x}\bullet\Vert _{L_{v}^{1}L_{x}^{\infty}}$ norm, as the $N_{2}^{-1}\Vert \bullet\Vert _{L_{v}^{1}L_{x}^{\infty}}$ norm can be dealt with in a similar way.
Recalling that $(d=3)$
\begin{equation}
f_{\mathrm{r}}(t,x,v)=M^{\frac{3}{2}-s}N_{1}^{\frac{3}{2}}\exp \left[
-\beta(t,x,v)\right] \chi (Mx)\chi (
N_{1}v),
\end{equation}
we use the pointwise upper bound \eqref{equ:upper bound,fr,k} to get
\begin{align*}
&\n{\nabla_{x}(v\cdot \nabla_{x}f_{\mathrm{r}})}_{L_{v}^{1}L_{x}^{\infty}}\\
\lesssim &N_{1}^{-1}M^{2+\frac{3}{2}-s}N_{1}^{\frac{3}{2}}\n{\exp[-\beta(t,x,v)](\nabla^{2} \chi)(Mx)\chi(N_{1}v)}_{L_{v}^{1}L_{x}^{\infty}}\\
	&+N_{1}^{-1}M^{1+\frac{3}{2}-s}N_{1}^{\frac{3}{2}}\n{\nabla_{x}\beta(t,x,v)\exp[-\beta(t,x,v)](\nabla \chi)(Mx)\chi(N_{1}v)}_{L_{v}^{1}L_{x}^{\infty}}\\
&+N_{1}^{-1}M^{\frac{3}{2}-s}N_{1}^{\frac{3}{2}}\n{\nabla_{x}^{2}\beta(t,x,v)\exp[-\beta(t,x,v)] \chi (Mx)\chi(N_{1}v)}_{L_{v}^{1}L_{x}^{\infty}}\\
&+N_{1}^{-1}M^{\frac{3}{2}-s}N_{1}^{\frac{3}{2}}\n{|\nabla_{x}\beta(t,x,v)|^{2}\exp[-\beta(t,x,v)] \chi (Mx)\chi(N_{1}v)}_{L_{v}^{1}L_{x}^{\infty}}\\
\lesssim& N_{1}^{-1}N_{1}^{-\frac{3}{2}}M^{2-s}M^{\frac{3}{2}} \exp[M^{1-s}|t|]\lra{M^{1-s}|t|}^{2}.
\end{align*}
When multiplied by $|T_{*}|=M^{s-1}(\ln\ln\ln M)$, this gives
\begin{align}\label{equ:vfr,L1}
&M^{-1}N_{2}^{-1} \bbn{\nabla_{x} \int_{\tau}^{t}e^{-(t-t_{0})v\cdot \nabla _{x}}(v\cdot \nabla_{x}f)(t_{0})\,dt_{0}}_{L_{v}^{1}L_{x}^{\infty}}\\
\lesssim& |T_{*}|N_{1}^{-\frac{5}{2}}N_{2}^{-1}M^{1-s}M^{\frac{3}{2}}\exp[M^{1-s}|T_{*}|]\lra{M^{1-s}|T_{*}|}^{2}\notag\\
\lesssim& N_{1}^{-\frac{5}{2}}N_{2}^{-1}M^{\frac{3}{2}} (\ln \ln M)^{4}. \notag
\end{align}

\textbf{The $N_{2}^{\frac{1}{5}}\Vert \bullet \Vert _{L_{v}^{\frac{5}{3}}L_{x}^{\infty}}$ and $M^{-1}N_{2}^{\frac{1}{5}}\Vert \nabla_{x}\bullet\Vert _{L_{v}^{\frac{5}{3}}L_{x}^{\infty}}$ estimates.}

These two norms can be estimated in the same manner as \eqref{equ:vfr,L1} with the $L_{v}^{1}$ norm replaced by the
$L_{v}^{\frac{5}{3}}$ norm. Therefore, we also have
\begin{align}
&M^{-1}N_{2}^{\frac{1}{5}} \bbn{\nabla_{x} \int_{\tau}^{t}e^{-(t-t_{0})v\cdot \nabla _{x}}(v\cdot \nabla_{x}f)(t_{0})\,dt_{0}}_{L_{v}^{\frac{5}{3}}L_{x}^{\infty}}\\
\lesssim& |T_{*}|N_{1}^{-1}N_{1}^{-\frac{3}{10}}N_{2}^{\frac{1}{5}}M^{1-s}M^{\frac{3}{2}}\exp[M^{1-s}|T_{*}|]\lra{M^{1-s}|T_{*}|}^{2}\notag\\
\lesssim& N_{1}^{-1}N_{1}^{-\frac{3}{10}}N_{2}^{\frac{1}{5}}M^{\frac{3}{2}} (\ln \ln M)^{4}\notag\\
\lesssim& N_{1}^{-\frac{11}{10}}M^{\frac{3}{2}} (\ln \ln M)^{4},\notag
\end{align}
where in the last inequality we have used that $N_{1}\geq N_{2}$.
\end{proof}

\subsubsection{Analysis of $Q^{+}(f_{\mathrm{r}},f_{\mathrm{b}})$, $Q^{\pm}(f_{\mathrm{b}},f_{\mathrm{r}})$, and
$Q^{\pm}(f_{\mathrm{r}},f_{\mathrm{r}})$}\label{section:Analysis of term2}
Before getting into the analysis of the terms, we recall some estimates on $f_{\mathrm{b}}$ and $f_{\mathrm{r}}$, which are established in Lemma \ref{lemma:bounds on fb} and Lemma \ref{lemma:z-norm bounds on fr}. That is,
\begin{align}
\n{f_{\mathrm{b}}}_{L^{\infty}([T_{*},0];Z)}\lesssim& M^{1-s}, \label{equ:z-norm,fb,used}\\
\n{f_{\mathrm{r}}}_{L^{\infty}([T_{*},0];Z)}\lesssim& M^{1-s}(\ln \ln M)^{2},\label{equ:z-norm,fr,used}\\
\n{f_{\mathrm{r}}}_{L_{v}^{1}L_{x}^{\infty}}^{\frac{1}{6}}\n{f_{\mathrm{r}}}_{L_{v}^{\frac{5}{3}}L_{x}^{\infty}}^{\frac{5}{6}}\lesssim & N_{1}^{-\frac{1}{2}}M^{\frac{3}{2}-s}(\ln \ln M)^{2},\label{equ:z-norm,lv1,used}\\
M^{-1}\n{\nabla_{x}f_{\mathrm{r}}}_{L_{v}^{1}L_{x}^{\infty}}^{\frac{1}{6}}\n{\nabla_{x}f_{\mathrm{r}}}_{L_{v}^{\frac{5}{3}}
L_{x}^{\infty}}^{\frac{5}{6}}\lesssim& N_{1}^{-\frac{1}{2}}M^{\frac{3}{2}-s}
(\ln \ln M)^{2},\label{equ:z-norm,lv1,H1,used}
\end{align}
where the last two inequalities \eqref{equ:z-norm,lv1,used}--\eqref{equ:z-norm,lv1,H1,used} follow from estimates \eqref{equ:fr,Z-norm,L1}--\eqref{equ:vfr,L35,H1}. In addition, during the proof of the bilinear estimate on $Q^{\pm}$ in Lemma \ref{lemma:binlinear estimate} we postpone to Section \ref{section:Bounds on the Correction Term}, we actually have that
\begin{align*}
\n{Q^{-}(f_{\mathrm{b}},f_{\mathrm{r}})}_{Z}\lesssim \n{f_{\mathrm{b}}}_{Z}\lrs{\n{f_{\mathrm{r}}}_{L_{v}^{1}L_{x}^{\infty}}^{\frac{1}{6}}\n{f_{\mathrm{r}}}_{L_{v}^{\frac{5}{3}}L_{x}^{\infty}}^{\frac{5}{6}}
+M^{-1}\n{\nabla_{x}f_{\mathrm{r}}}_{L_{v}^{1}L_{x}^{\infty}}^{\frac{1}{6}}\n{\nabla_{x}f_{\mathrm{r}}}_{L_{v}^{\frac{5}{3}}L_{x}^{\infty}}^{\frac{5}{6}}},\\
\n{Q^{+}(f_{\mathrm{b}},f_{\mathrm{r}})}_{Z}\lesssim \n{f_{\mathrm{b}}}_{Z}\lrs{\n{f_{\mathrm{r}}}_{L_{v}^{1}L_{x}^{\infty}}^{\frac{1}{6}}\n{f_{\mathrm{r}}}_{L_{v}^{\frac{5}{3}}L_{x}^{\infty}}^{\frac{5}{6}}
+M^{-1}\n{\nabla_{x}f_{\mathrm{r}}}_{L_{v}^{1}L_{x}^{\infty}}^{\frac{1}{6}}\n{\nabla_{x}f_{\mathrm{r}}}_{L_{v}^{\frac{5}{3}}L_{x}^{\infty}}^{\frac{5}{6}}},\\
\n{Q^{+}(f_{\mathrm{r}},f_{\mathrm{b}})}_{Z}\lesssim \n{f_{\mathrm{b}}}_{Z}\lrs{\n{f_{\mathrm{r}}}_{L_{v}^{1}L_{x}^{\infty}}^{\frac{1}{6}}\n{f_{\mathrm{r}}}_{L_{v}^{\frac{5}{3}}L_{x}^{\infty}}^{\frac{5}{6}}
+M^{-1}\n{\nabla_{x}f_{\mathrm{r}}}_{L_{v}^{1}L_{x}^{\infty}}^{\frac{1}{6}}\n{\nabla_{x}f_{\mathrm{r}}}_{L_{v}^{\frac{5}{3}}L_{x}^{\infty}}^{\frac{5}{6}}},\\
\n{Q^{\pm}(f_{\mathrm{r}},f_{\mathrm{r}})}_{Z}\lesssim \n{f_{\mathrm{r}}}_{Z}\lrs{\n{f_{\mathrm{r}}}_{L_{v}^{1}L_{x}^{\infty}}^{\frac{1}{6}}\n{f_{\mathrm{r}}}_{L_{v}^{\frac{5}{3}}L_{x}^{\infty}}^{\frac{5}{6}}
+M^{-1}\n{\nabla_{x}f_{\mathrm{r}}}_{L_{v}^{1}L_{x}^{\infty}}^{\frac{1}{6}}\n{\nabla_{x}f_{\mathrm{r}}}_{L_{v}^{\frac{5}{3}}L_{x}^{\infty}}^{\frac{5}{6}}}.
\end{align*}
Note that such an estimate is not possible for $Q^{-}(f_{\mathrm{r}},f_{\mathrm{b}})$, which is not contained in the error terms.
Therefore, for
$$(sgn,1,2)\in\lr{(+,r,b),(\pm,b,r),(\pm,r,r)}, $$
by estimates \eqref{equ:z-norm,fb,used}--\eqref{equ:z-norm,lv1,used}, we have
\begin{align*}
&\bbn{ \int_{\tau }^{t}e^{-(t-t_{0})v\cdot \nabla_{x}}Q^{sgn}(f_{1},f_{2})(t_{0})\,dt_{0}}_{Z}\\
\lesssim& |T_{*}|\lrs{\n{f_{\mathrm{r}}}_{Z}+\n{f_{\mathrm{b}}}_{Z}}\lrs{\n{f_{\mathrm{r}}}_{L_{v}^{1}L_{x}^{\infty}}^{\frac{1}{6}}\n{f_{\mathrm{r}}}_{L_{v}^{\frac{5}{3}}L_{x}^{\infty}}^{\frac{5}{6}}
+M^{-1}\n{\nabla_{x}f_{\mathrm{r}}}_{L_{v}^{1}L_{x}^{\infty}}^{\frac{1}{6}}\n{\nabla_{x}f_{\mathrm{r}}}_{L_{v}^{\frac{5}{3}}L_{x}^{\infty}}^{\frac{5}{6}}}\\
\lesssim& |T_{*}|M^{1-s}(\ln \ln M)^{2}N_{1}^{-\frac{1}{2}}M^{\frac{3}{2}-s}
(\ln \ln M)^{2}\\
\lesssim& N_{1}^{-\frac{1}{2}}M^{\frac{3}{2}-s}(\ln \ln  M)^{5},
\end{align*}
where in the last inequality we have inserted in $|T_{*}|=M^{s-1}(\ln\ln\ln M)$. This bound suffices for our goal as it carries the smallness factor $N_{1}^{-\frac{1}{2}}$.

\subsubsection{Analysis of $Q^{-}(f_{\mathrm{b}},f_{\mathrm{b}})$} \label{section:Analysis of Q-}

\begin{lemma}
For $T_{*}\leq  \tau\leq t \leq 0$,
	\begin{align}
		\bbn{ \int_{\tau}^{t}e^{-(t-t_{0})v\cdot \nabla _{x}}Q^{-}(f_{\mathrm{b}},f_{\mathrm{b}})(t_{0})\,dt_{0}} _{Z}\lesssim M^{-1}.
	\end{align}
\end{lemma}
\begin{proof}
As we have required that $N_{2}\gg M$ in \eqref{equ:condition,parameters}, the desired smallness comes from the factor $N_{2}^{-\delta}$ for some $\delta>0$.
As the $x$-derivative, which is put on $f_{\mathrm{b}}$, produces a factor of $M$, it is sufficient to estimate the $L_{v}^{2}L_{x}^{2}$, $L_{v}^{1}L_{x}^{\infty}$ and $L_{v}^{\frac{5}{3}}L_{x}^{\infty}$ norms.
	
\textbf{The $M\Vert \bullet \Vert
	_{L_{v}^{2}L_{x}^{2}}$ estimate.}

Note that
\begin{equation}\label{equ:equ:Q-,fb,fb,pointwise,first}
	M\bbn{ \int_{\tau}^{t}e^{-(t-t_{0})v\cdot \nabla _{x}}Q^{-}(f_{\mathrm{b}},f_{\mathrm{b}})(t_{0})\,dt_{0}}_{L_{v}^{2}L_{x}^{2}}\lesssim |T_{*}|M\n{Q^{-}(f_{\mathrm{b}},f_{\mathrm{b}})}_{L_{t}^{\infty }(T_{*},0;L_{v}^{2}L_{x}^{2})}.
\end{equation}
We only need to control $M\n{Q^{-}(f_{\mathrm{b}},f_{\mathrm{b}})}_{L_{t}^{\infty }L_{v}^{2}L_{x}^{2}}$.
Recall the upper bound \eqref{equ:pointwise estimate on fb} that
 \begin{align}\label{equ:pointwise estimate on fb,recall}
 f_{\mathrm{b}}(t,x,u)\lesssim& \frac{M^{1-s}}{N_{2}^{2}}\sum_{j=1}^{J}\wt{K}_{j}(x)I_{j}(u),
 \end{align}
 where
 \begin{align*}
  \wt{K}_{j}(x)=\chi (\frac{
		MP_{e_{j}}^{\perp }x}{10})\chi (\frac{P_{e_{j}}x}{10N_{2}}),\quad
I_{j}(u)=\chi(MP_{e_{j}}^{\perp }u)\chi (\frac{10P_{e_{j}}(u-N_{2}e_{j})}{N_{2}}).
 \end{align*}
Then we have
 \begin{align}\label{equ:Q-,fb,fb,pointwise,two cases}
&Q^{-}(f_{\mathrm{b}},f_{\mathrm{b}})(t,x,v)\\
\lesssim &\lrs{\frac{M^{1-s}}{N_{2}^{2}}}^{2} \lrs{\sum_{|j-k|\lesssim 1}\wt{K}_{j}(x)\wt{K}_{k}(x)Q^{-}(I_{j},I_{k})(v)+\sum_{|j-k|\gtrsim 1}\wt{K}_{j}(x)\wt{K}_{k}(x)Q^{-}(I_{j},I_{k})(v)}.\notag
\end{align}

\textbf{Case $I$: $|j-k|\lesssim 1$.}

For the case that $|j-k|\lesssim 1$, the summands in the double sum $\sum_{k}^{J}\sum_{j}^{J}$ are reduced to $(MN_{2})^{2}$. By H\"{o}lder and Hardy-Sobolev-Littlewood inequality \eqref{equ:endpoint estimate,hls}, we obtain
\begin{align}\label{equ:Q-,fb,fb,pointwise,case 1}
&\lrs{\frac{M^{1-s}}{N_{2}^{2}}}^{2}\bbn{\sum_{|j-k|\lesssim 1}\wt{K}_{j}(x)\wt{K}_{k}(x)Q^{-}(I_{j},I_{k})(v)}_{L_{x}^{2}L_{v}^{2}}\\
\lesssim &\lrs{\frac{M^{1-s}}{N_{2}^{2}}}^{2}\bbn{\sum_{j}\wt{K}_{j}(x)I_{j}(v)}_{L_{x}^{2}L_{v}^{2}}\n{\wt{K}_{k}(x)}_{L_{x}^{\infty}}\bbn{\int \frac{I_{k}(u)}{|u-v|}du}_{L_{v}^{\infty}}\notag\\
\lesssim &\lrs{\frac{M^{1-s}}{N_{2}^{2}}}^{2}\bbn{\sum_{j}\wt{K}_{j}(x)I_{j}(v)}_{L_{x}^{2}L_{v}^{2}}
\n{I_{k}}_{L_{v}^{1}}^{\frac{1}{3}}\n{I_{k}}_{L_{v}^{2}}^{\frac{2}{3}}\notag\\
\lesssim &\lrs{\frac{M^{1-s}}{N_{2}^{2}}}^{2} (M^{-1}N_{2}^{2})(M^{-2}N_{2})^{\frac{2}{3}}\notag\\
=&M^{-\frac{1}{3}-2s}N_{2}^{-\frac{4}{3}}\notag
\end{align}
where in the second-to-last inequality we have used the disjointness of the $v$-support to get
\begin{align}
\bbn{\sum_{j}\wt{K}_{j}(x)I_{j}(v)}_{L_{x}^{2}L_{v}^{2}}^{2}\lesssim \sum_{j}\n{\wt{K}_{j}(x)}_{L_{x}^{2}}^{2}\n{I_{j}(v)}_{L_{v}^{2}}^{2}\lesssim (MN_{2})^{2}(M^{-2}N_{2})^{2}=M^{-2}N_{2}^{4}.
\end{align}

\textbf{Case $II$: $|j-k|\gtrsim 1$.}

For the case that $|j-k|\gtrsim 1$, this implies that $\sin \alpha_{j,k}\gtrsim (MN_{2})^{-1}$, where $\alpha_{j,k}$ denotes the angle between $e_{j}$ and $e_{k}$.
Then we have
\begin{align*}
Q^{-}(I_{j},I_{k})(v)=&I_{j}(v)\int \frac{I_{k}(u)}{|u-v|}du\\
=&I_{j}(v)\int \frac{1}{|u-v|}\chi
	(MP_{e_{k}}^{\perp }u)\chi (\frac{10P_{e_{k}}(u-N_{2}e_{k})}{N_{2}})du\\
\lesssim &I_{j}(v)\int \frac{1}{|P_{e_{k}}(u-v)|+|P_{e_{k}}^{\perp}(u-v)|}\chi
	(MP_{e_{k}}^{\perp }u)\chi (\frac{10P_{e_{k}}(u-N_{2}e_{k})}{N_{2}})du.
\end{align*}
Due to the $v$-support and $u$-support, we write
$$v=ae_{j}+ce_{j}^{\perp},\quad u=be_{k}+de_{k}^{\perp}$$
where $a\sim b\sim N_{2}$ and $c\sim d\sim M^{-1}$. Therefore, this gives
\begin{align}\label{equ:lower bound,ej,ek}
|P_{e_{k}}^{\perp}(u-v)|=&|P_{e_{k}}^{\perp}(be_{k}+de_{k}^{\perp}-ae_{j}-ce_{j}^{\perp})|\\
\gtrsim &a|P_{e_{k}}^{\perp}e_{j}|-d-c\notag\\
\gtrsim &N_{2}\sin \alpha_{j,k}-M^{-1}\gtrsim M^{-1}\notag
\end{align}
where in the last inequality we have used that $\sin \alpha_{j,k}\gtrsim (MN_{2})^{-1}$. By the estimate \eqref{equ:lower bound,ej,ek}, we then set $\xi=\lra{u,e_{k}}$ to get
\begin{align}\label{equ:estimate,IjIk}
I_{j}(v)\int \frac{I_{k}(u)}{|u-v|}du\lesssim &I_{j}(v)\int \frac{1}{|P_{e_{k}}(u-v)|+M^{-1}}\chi
	(MP_{e_{k}}^{\perp }u)\chi (\frac{10P_{e_{k}}(u-N_{2}e_{k})}{N_{2}})du\\
\lesssim &I_{j}(v)\int \frac{1}{|\xi-\lra{e_{k},v}|+M^{-1}}\chi
	(M\xi^{\perp})\chi (\frac{10(\xi-N_{2})}{N_{2}})d\xi d\xi^{\perp}\notag\\
\lesssim &I_{j}(v)M^{-2}\int\frac{1}{|\xi|+M^{-1}}\chi(\frac{10(\xi+\lra{e_{k},v}-N_{2})}{N_{2}})d\xi\notag\\
=&I_{j}(v)M^{-2}N_{2}\int_{-1}^{1}\frac{M}{MN_{2}|\xi|+1}\chi(\frac{10(N_{2}\xi+\lra{e_{k},v}-N_{2})}{N_{2}})d\xi\notag\\
\lesssim &I_{j}(v)M^{-2}N_{2}\int_{-1}^{1}\frac{M}{MN_{2}|\xi|+1}d\xi\notag\\
\lesssim &I_{j}(v)\frac{\ln(MN_{2})}{M^{2}}.\notag
\end{align}
Consequently, we arrive at
\begin{align}
&\lrs{\frac{M^{1-s}}{N_{2}^{2}}}^{2}\sum_{|j-k|\gtrsim 1}\wt{K}_{j}(x)\wt{K}_{k}(x)Q^{-}(I_{j},I_{k})\\
=&\lrs{\frac{M^{1-s}}{N_{2}^{2}}}^{2}\sum_{|j-k|\gtrsim 1}\wt{K}_{j}(x)\wt{K}_{k}(x)I_{j}(v)\int \frac{I_{k}(u)}{|u-v|}du\notag\\
\lesssim & \lrs{\frac{M^{1-s}}{N_{2}^{2}}}^{2}\sum_{j,k}\wt{K}_{j}(x)\wt{K}_{k}(x)I_{j}(v)\frac{\ln(MN_{2})}{M^{2}}\notag\\
\leq& \lrs{\frac{M^{1-s}}{N_{2}^{2}}}^{2} \frac{\ln(MN_{2})}{M^{2}} \lrs{\frac{N_{2}}{|x|+M^{-1}}}^{2}\chi(\frac{x}{N_{2}}) \sum_{j}\wt{K}_{j}(x)I_{j}(v)\notag
\end{align}
where in the last inequality we have used that
\begin{align}\label{equ:fb,sum}
\sum_{k}\wt{K}_{k}(x)=\sum_{k}^{J}\chi (\frac{MP_{e_{k}}^{\perp }x}{10})\chi (\frac{P_{e_{k}}x}{N_{2}})\lesssim\lrs{\frac{N_{2}}{
|x|+M^{-1}}}^{2}\chi (\frac{x}{N_{2}}).
\end{align}
To see (\ref{equ:fb,sum}), we might as well take $x=(0,0,|x|)$ with $M^{-1}\leq |x|\leq N_{2}$. Let $\theta_{j}$ be the angle between $e_{j}$ and $(0,0,1)$. Then, we have
\begin{align*}
\sum_{j}^{J}\chi (\frac{MP_{e_{j}}^{\perp }x}{10})\chi (\frac{P_{e_{j}}x}{N_{2}})=\sum_{j}^{J}\chi (\frac{M|x|\sin \theta_{j}}{10})=\sum_{j:\sin\theta_{j}\lesssim \frac{1}{|x|M}}1\sim \frac{(MN_{2})^{2}}{(|x|M)^{2}}= \frac{N_{2}^{2}}{|x|^{2}}.
\end{align*}
Applying the $L_{v}^{2}L_{x}^{2}$ norm, we have
\begin{align}\label{equ:Q-,fb,fb,pointwise,case 2}
&\lrs{\frac{M^{1-s}}{N_{2}^{2}}}^{2}\bbn{\sum_{|j-k|\gtrsim 1}\wt{K}_{j}(x)\wt{K}_{k}(x)Q^{-}(I_{j},I_{k})}_{L_{v}^{2}L_{x}^{2}}\\
\lesssim& \lrs{\frac{M^{1-s}}{N_{2}^{2}}}^{2} \frac{\ln(MN_{2})}{M^{2}} \bbn{\sum_{j}\wt{K}_{j}(x)I_{j}(v)}_{L_{v}^{2}L_{x}^\infty}
\bbn{\lrs{\frac{N_{2}}{|x|+M^{-1}}}^{2}\chi(\frac{x}{N_{2}})}_{L_{x}^{2}}\notag\\
\lesssim& \lrs{\frac{M^{1-s}}{N_{2}^{2}}}^{2} \frac{\ln(MN_{2})}{M^{2}} N_{2}^{\frac{3}{2}} (M^{\frac{1}{2}}N_{2}^{2})\notag\\
=& M^{\frac{1}{2}-2s}N_{2}^{-\frac{1}{2}}\ln(MN_{2})\notag
\end{align}
where we have used that
\begin{align*}
&\bbn{\sum_{j}\wt{K}_{j}(x)I_{j}(v)}_{L_{v}^{2}L_{x}^\infty}\lesssim \bbn{\sum_{j}I_{j}(v)}_{L_{v}^{2}}\sim N_{2}^{\frac{3}{2}},
\end{align*}
and
\begin{align}
& \bbn{\lrs{\frac{N_{2}}{|x|+M^{-1}}}^{2}\chi(\frac{x}{N_{2}})}_{L_{x}^{2}}=M^{\frac{1}{2}}N_{2}^{2}
\bbn{\lrs{\frac{1}{|x|+1}}^{2}\chi(\frac{x}{MN_{2}})}_{L_{x}^{2}} \lesssim M^{\frac{1}{2}}N_{2}^{2}.\label{equ:pointwise estimate,fb}
\end{align}

Combining estimates \eqref{equ:Q-,fb,fb,pointwise,case 1} and \eqref{equ:Q-,fb,fb,pointwise,case 2} in the two cases, we finally reach
\begin{align*}
M\n{Q^{-}(f_{\mathrm{b}},f_{\mathrm{b}})}_{L_{v}^{2}L_{x}^{2}}\lesssim M^{\frac{3}{2}-2s}N_{2}^{-\frac{1}{2}}\ln(MN_{2}).
\end{align*}
Together with \eqref{equ:equ:Q-,fb,fb,pointwise,first}, we insert in $|T_{*}|=M^{s-1}(\ln \ln \ln M) $ to obtain
\begin{align*}
M\bbn{ \int_{\tau}^{t}e^{-(t-t_{0})v\cdot \nabla _{x}}Q^{-}(f_{\mathrm{b}},f_{\mathrm{b}})(t_{0})\,dt_{0}}_{L_{v}^{2}L_{x}^{2}}\lesssim N_{2}^{-\frac{1}{2}} M^{\frac{1}{2}-s}\ln(MN_{2})(\ln \ln \ln M),
\end{align*}
which suffices for our goal.

\textbf{The $N_{2}^{-1}\Vert \bullet\Vert _{L_{v}^{1}L_{x}^{\infty}}$ estimate.}

For convenience, we use the notation
\begin{equation*}
	D^{-}=\int_{\tau}^{t}e^{-(t-t_{0})v\cdot \nabla
		_{x}}Q^{-}(f_{\mathrm{b}},f_{\mathrm{b}})(t_{0})dt_{0}.
\end{equation*}

From the analysis on $Q^{-}(f_{\mathrm{b}},f_{\mathrm{b}})$ in estimates \eqref{equ:Q-,fb,fb,pointwise,case 1} and \eqref{equ:Q-,fb,fb,pointwise,case 2}, we actually get a pointwise estimate on $Q^{-}(f_{\mathrm{b}},f_{\mathrm{b}})$ that
\begin{align*}
Q^{-}(f_{\mathrm{b}},f_{\mathrm{b}})\lesssim \lrs{\frac{M^{1-s}}{N_{2}^{2}}}^{2}\sum_{j}\wt{K}_{j}(x)I_{j}(v)\lrc{\frac{\ln(MN_{2})}{M^{2}}
\lrs{\frac{N_{2}}{|x|+M^{-1}}}^{2}\chi (\frac{x}{N_{2}})+(M^{-2}N_{2})^{\frac{2}{3}}}.
\end{align*}
Expanding $D^{-}$ gives that
\begin{align*}
D^{-}=&\int_{\tau}^{t}Q^{-}(f_{\mathrm{b}},f_{\mathrm{b}})(t_{0},x-v(t-t_{0}),v)dt_{0}\\
\lesssim &\int_{\tau}^{t} \lrs{\frac{M^{1-s}}{N_{2}^{2}}}^{2} \sum_{j}\wt{K}_{j}(x-v(t-t_{0}))I_{j}(v) \\
&\times \lrc{\frac{\ln(MN_{2})}{M^{2}} \lrs{\frac{N_{2}}{|x-v(t-t_{0})|+M^{-1}}}^{2}\chi(\frac{x-v(t-t_{0})}{N_{2}})+(M^{-2}N_{2})^{\frac{2}{3}}}dt_{0}\\
\lesssim &\frac{M^{-2s}}{N_{2}^{4}}I(v)
 \lrc{\ln (MN_{2}) \int_{T_{*}}^{0}  \lrs{\frac{N_{2}}{
		|x-v(t-t_{0})|+M^{-1}}}^{2} \chi (\frac{x-v(t-t_{0})}{N_{2}})dt_{0}+(MN_{2})^{\frac{2}{3}}},
\end{align*}
where in the last inequality we have used that
\begin{align*}
\sum_{j}\wt{K}_{j}(x-v(t-t_{0}))I_{j}(v)\leq \sum_{j}I_{j}(v)=:I(v)\sim 1_{\lr{\frac{9N_{2}}{10}\leq|v|\leq \frac{11N_{2}}{10}}}(v).
\end{align*}

We then deal with the time integral. By change of variable, we have
\begin{align*}
	& I(v)\int_{T_{*}}^{0}  \lrs{\frac{N_{2}}{
		|x-v(t-t_{0})|+M^{-1}}}^{2} \chi (\frac{x-v(t-t_{0})}{N_{2}})dt_{0}\\
	\leq&I(v)\int_{T_{*}}^{|T_{*}|} \lrs{\frac{MN_{2}}{|Mx-Mv\sigma|+1}}^{2}
\chi (\frac{x-v\sigma}{N_{2}})d\sigma\\
	\leq& \frac{(MN_{2})^{2}I(v)}{M|v|}\int_{-M|T_{*}||v|}^{M|T_{*}||v|}\lrs{\frac{1}{\babs{\sigma-M|x|}+1}}^{2}d\sigma\\
	\lesssim& MN_{2}I(v),
\end{align*}
where in the last inequality we have used that $|v|\sim N_{2}$ and
$\int \frac{d\tau}{\lra{\tau}^{2}}\lesssim 1$.
Hence, after carrying out the 1D $dt_{0}$ integral, we arrive at
\begin{align}\label{equ:Q-,fb,fb,L1,final}
	N_{2}^{-1}\n{D^{-}}_{L_{v}^{1}L_{x}^{\infty}}\lesssim&N_{2}^{-1}\frac{M^{-2s}}{N_{2}^{4}} \n{I(v)}_{L_{v}^{1}} \lrc{\ln (MN_{2})MN_{2}+(MN_{2})^{\frac{2}{3}}}
	\\
	\lesssim& N_{2}^{-1}\frac{M^{-2s}}{N_{2}^{4}}  N_{2}^{3} \ln (MN_{2})MN_{2}\notag\\
	=&N_{2}^{-1} M^{1-2s}\ln (MN_{2}).\notag
\end{align}

\textbf{The $N_{2}^{\frac{1}{5}}\Vert \bullet\Vert _{L_{v}^{\frac{5}{3}}L_{x}^{\infty}}$ estimate.}

By the interpolation inequality, we have
\begin{align}\label{equ:Q-,fb,fb,L53}
N_{2}^{\frac{1}{5}}\n{D^{-}}_{L_{v}^{\frac{5}{3}}L_{x}^{\infty}}\leq \lrs{N_{2}^{-1}\n{D^{-}}_{L_{v}^{1}L_{x}^{\infty}}}^{\frac{1}{5}}
\lrs{N_{2}^{\frac{1}{2}}\n{D^{-}}_{L_{v}^{2}L_{x}^{\infty}}}^{\frac{4}{5}}.
\end{align}
For the $L_{v}^{2}L_{x}^{\infty}$ norm on $D^{-}$, by H\"{o}lder inequality, we have
\begin{align*}
	N_{2}^{\frac{1}{2}}\n{D^{-}}_{L_{v}^{2}L_{x}^{\infty}}\lesssim& N_{2}^{\frac{1}{2}}|T_{*}|\n{Q^{-}(f_{\mathrm{b}},f_{\mathrm{b}})}_{L_{v}^{2}L_{x}^{\infty}}\\
	\lesssim &N_{2}^{\frac{1}{2}}|T_{*}| \bbn{\int \frac{f_{\mathrm{b}}(x,u)}{|u-v|}du}_{L_{v}^{\infty}L_{x}^{\infty}}
	\n{f_{\mathrm{b}}}_{L_{v}^{2}L_{x}^{\infty}}.
\end{align*}
We then use the $L^{\infty}$ estimate \eqref{equ:endpoint estimate,hls} in Lemma \ref{lemma:endpoint estimate,hls} and interpolation inequality to get
\begin{align*}\bbn{\int \frac{f_{\mathrm{b}}(x,u)}{|u-v|}du}_{L_{v}^{\infty}L_{x}^{\infty}}\lesssim \n{f_{\mathrm{b}}}_{L_{v}^{1}L_{x}^{\infty}}^{\frac{1}{6}}\n{f_{\mathrm{b}}}_{L_{v}^{\frac{5}{3}}L_{x}^{\infty}}^{\frac{5}{6}}\lesssim
N_{2}^{-1}\n{f_{\mathrm{b}}}_{L_{v}^{1}L_{x}^{\infty}}+N_{2}^{\frac{1}{5}}\n{f_{\mathrm{b}}}_{L_{v}^{\frac{5}{3}}L_{x}^{\infty}}\leq \n{f_{\mathrm{b}}}_{Z}.
\end{align*}
By the $Z$-norm bound on $f_{\mathrm{b}}$ in Lemma \ref{lemma:bounds on fb}, we have that
$$\n{f_{\mathrm{b}}}_{Z}\lesssim M^{1-s},\quad N_{2}^{\frac{1}{2}}\n{f_{\mathrm{b}}}_{L_{v}^{2}L_{x}^{\infty}}\lesssim M^{1-s}.$$
Thus, inserting in $|T_{*}|=M^{s-1}(\ln\ln\ln M)$, we obtain
\begin{align}\label{equ:Q-,fb,fb,L2}
N_{2}^{\frac{1}{2}}\n{D^{-}}_{L_{v}^{2}L_{x}^{\infty}}
    \lesssim&  M^{1-s}(\ln\ln\ln M).
\end{align}
Combining estimates \eqref{equ:Q-,fb,fb,L1,final}, \eqref{equ:Q-,fb,fb,L53} and \eqref{equ:Q-,fb,fb,L2}, we reach
\begin{align}
N_{2}^{\frac{1}{5}}\n{D^{-}}_{L_{v}^{\frac{5}{3}}L_{x}^{\infty}}\lesssim N_{2}^{-\frac{1}{5}}M^{1-\frac{6}{5}s}\ln (MN_{2})(\ln\ln\ln M).
\end{align}
This bound is enough as it carries the smallness parameter $N_{2}^{-\frac{1}{5}}$.
\end{proof}
\subsubsection{Analysis of $Q^{+}(f_{\mathrm{b}},f_{\mathrm{b}})$} \label{section:Analysis of Q+}

\begin{lemma}
For $T_{*}\leq  \tau\leq t \leq 0$,
	\begin{align}
		\bbn{\int_{\tau}^{t}e^{-(t-t_{0})v\cdot \nabla _{x}}Q^{+}(f_{\mathrm{b}},f_{\mathrm{b}})(t_{0})\,dt_{0}}_{Z}\lesssim M^{-1}.
	\end{align}
\end{lemma}
\begin{proof}
In a similar way to estimate $Q^{-}(f_{\mathrm{b}},f_{\mathrm{b}})$, we obtain the desired estimate provided the upper bound carries the smallness factor $N_{2}^{-\delta}$ for some $\delta>0$.
	For convenience, we use the notation
\begin{equation*}
	D^{+}=\int_{\tau}^{t}e^{-(t-t_{0})v\cdot \nabla
		_{x}}Q^{+}(f_{\mathrm{b}},f_{\mathrm{b}})(t_{0})dt_{0}.
\end{equation*}
The $x$-derivative produces the factor of $M$, so we only need to estimate the $L_{v}^{2}L_{x}^{2}$, $L_{v}^{1}L_{x}^{\infty}$ and $L_{v}^{\frac{5}{3}}L_{x}^{\infty}$ norms.

\textbf{The $M\Vert \bullet \Vert
	_{L_{v}^{2}L_{x}^{2}}$ estimate.}

We use again the upper bound \eqref{equ:pointwise estimate on fb} that
 \begin{align}
 f_{\mathrm{b}}(t,x,u)\lesssim& \frac{M^{1-s}}{N_{2}^{2}}\sum_{j=1}^{J}\wt{K}_{j}(x)I_{j}(u),
 \end{align}
  where
  $$\wt{K}_{j}(x)=\chi (\frac{
		MP_{e_{j}}^{\perp }x}{10})\chi (\frac{P_{e_{j}}x}{10N_{2}}),\quad I_{j}(u)=\chi(MP_{e_{j}}^{\perp }u)\chi (\frac{10P_{e_{j}}(u-N_{2}e_{j})}{N_{2}}).$$
Then we expand $Q^{+}(f_{\mathrm{b}},f_{\mathrm{b}})$ to get
\begin{align*}
&\n{Q^{+}(f_{\mathrm{b}},f_{\mathrm{b}})}_{L_{v}^{2}L_{x}^{2}}^{2}\\
\lesssim &
\frac{M^{4-4s}}{N_{2}^{8}}\sum_{j_{1},j_{2},j_{3},j_{4}} \int \wt{K}_{j_{1}}(x)
\wt{K}_{j_{2}}(x)\wt{K}_{j_{3}}(x)\wt{K}_{j_{4}}(x)
 Q^{+}(I_{j_{1}},I_{j_{2}})(v)Q^{+}(I_{j_{3}},I_{j_{4}})(v) dx dv.
\end{align*}
By using that $\wt{K}_{j_{2}}(x)\lesssim 1$ and $\wt{K}_{j_{4}}(x)\lesssim 1$, we obtain
\begin{align*}
	&\sum_{j_{1},j_{2},j_{3},j_{4}} \int \wt{K}_{j_{1}}(x)
	\wt{K}_{j_{2}}(x)\wt{K}_{j_{3}}(x)\wt{K}_{j_{4}}(x)
	Q^{+}(I_{j_{1}},I_{j_{2}})(v)Q^{+}(I_{j_{3}},I_{j_{4}})(v) dx dv\\
	\lesssim&
	\sum_{j_{1},j_{2},j_{3},j_{4}} \int \wt{K}_{j_{1}}(x)
	\wt{K}_{j_{3}}(x)dx
	\int Q^{+}(I_{j_{1}},I_{j_{2}})(v)Q^{+}(I_{j_{3}},I_{j_{4}})(v)  dv\\
	=&\sum_{j_{1},j_{3}} \int \wt{K}_{j_{1}}(x)
	\wt{K}_{j_{3}}(x)dx
	\int Q^{+}(I_{j_{1}},I)(v)Q^{+}(I_{j_{3}},I)(v)  dv
\end{align*}
where
\begin{align*}
I(v)=\sum_{j}^{J}I_{j}(v)\sim 1_{\lr{\frac{9N_{2}}{10}\leq|v|\leq \frac{11N_{2}}{10}}}(v).
\end{align*}
By H\"{o}lder inequality and bilinear estimate \eqref{equ:bilinear estimate,Q+} for $Q^{+}$ in Lemma \ref{lemma:bilinear estimate,Q+},
\begin{align*}
\n{Q^{+}(f_{\mathrm{b}},f_{\mathrm{b}})}_{L_{v}^{2}L_{x}^{2}}^{2}\lesssim&  \frac{M^{4-4s}}{N_{2}^{8}}\sum_{j_{1},j_{3}} \int \wt{K}_{j_{1}}(x)
\wt{K}_{j_{3}}(x)dx
\n{Q^{+}(I_{j_{1}},I)}_{L_{v}^{2}} \n{Q^{+}(I_{j_{3}},I)}_{L_{v}^{2}}\\
\lesssim&\frac{M^{4-4s}}{N_{2}^{8}} \n{I}_{L_{v}^{3}}^{2} \sum_{j_{1},j_{3}} \int \wt{K}_{j_{1}}(x)
\wt{K}_{j_{3}}(x)dx
\n{I_{j_{1}}}_{L_{v}^{\frac{6}{5}}}\n{I_{j_{3}}}_{L^{\frac{6}{5}}}.
\end{align*}
Using $\n{I_{j}}_{L_{v}^{\frac{6}{5}}}\lesssim (M^{-2}N_{2})^{\frac{5}{6}}$, $\n{I}_{L_{v}^{3}}\lesssim N_{2}$, estimates \eqref{equ:fb,sum} and \eqref{equ:pointwise estimate,fb} for the sum, we obtain
\begin{align*}
	\n{Q^{+}(f_{\mathrm{b}},f_{\mathrm{b}})}_{L_{v}^{2}L_{x}^{2}}^{2}\lesssim & \frac{M^{4-4s}}{N_{2}^{8}}(M^{-2}N_{2})^{\frac{5}{3}} N_{2}^{2}\bbn{\sum_{j}\wt{K}_{j}}_{L_{x}^{2}}^{2}\\
	\lesssim &\frac{M^{4-4s}}{N_{2}^{8}}(M^{-2}N_{2})^{\frac{5}{3}} N_{2}^{2}	\bbn{\lrs{\frac{N_{2}}{|x|+M^{-1}}}^{2}\chi(\frac{x}{N_{2}})}_{L_{x}^{2}}^{2}\\
	\lesssim &\frac{M^{4-4s}}{N_{2}^{8}}(M^{-2}N_{2})^{\frac{5}{3}} N_{2}^{2}(MN_{2}^{4})\\
=&  M^{\frac{5}{3}-4s}N_{2}^{-\frac{1}{3}}.
\end{align*}
Thus, we arrive at
\begin{align}\label{equ:Q+,fb,fb,est1,L2}
M\n{Q^{+}(f_{\mathrm{b}},f_{\mathrm{b}})}_{L_{v}^{2}L_{x}^{2}}\lesssim N_{2}^{-\frac{1}{6}} M^{\frac{11}{6}-2s}.
\end{align}
Upon multiplying by the time factor $|T_*|=M^{s-1}(\ln\ln\ln M)$, this yields a desired bound
\begin{align}
	M\n{D^{+}}_{L_{v}^{2}L_{x}^{2}}\lesssim&
	N_{2}^{-\frac{1}{6}} M^{\frac{5}{6}-s}(\ln\ln\ln M).
\end{align}

\textbf{The $N_{2}^{-1}\Vert \bullet\Vert _{L_{v}^{1}L_{x}^{\infty}}$ estimate.}

Recall that
\begin{align*}
f_{\mathrm{b}}(t)=&\frac{M^{1-s}}{N_{2}^{2}}\sum_{j=1}^{J}K_{j}(x-vt)I_{j}(v),
\end{align*}
where
$K_{j}(x)=\chi (MP_{e_{j}}^{\perp }x)\chi (\frac{P_{e_{j}}x}{N_{2}})$, $I_{j}(v)=\chi(MP_{e_{j}}^{\perp }v)\chi (\frac{10P_{e_{j}}(v-N_{2}e_{j})}{N_{2}})$. The gain term is
\begin{equation*}
Q^{+}(f,g)=\int_{\mathbb{S}^{2}}\int_{\mathbb{R}^{3}}B(u-v,\om)f(v^{\ast })g(u^{\ast
})\,du\,d\omega,
\end{equation*}
with the relationship that
\begin{align*}
v^{\ast } =&P_{\omega }^{\Vert }u+P_{\omega }^{\bot }v,\quad u^{\ast
}=P_{\omega }^{\Vert }v+P_{\omega }^{\bot }u, \\
v =&P_{\omega }^{\bot }v^{\ast }+P_{\omega }^{\Vert }u^{\ast },\quad%
u=P_{\omega }^{\Vert }v^{\ast }+P_{\omega }^{\bot }u^{\ast }.
\end{align*}
Then, expanding $D^{+}$ gives
\begin{align*}
D^{+} =&\frac{M^{2-2s}}{N_{2}^{4}}\sum_{k}^{J}\sum_{j}^{J}\int_{\tau}^{t}e^{-(t-t_{0})v\cdot \nabla_{x}}Q^{+}(K_{k}(x-tv)I_{k}(v),K_{j}(x-tv)I_{j}(v))(t_{0})dt_{0}\\
=&\frac{M^{2-2s}}{N_{2}^{4}}\sum_{k}^{J}\sum_{j}^{J}\int_{
\mathbb{S}^{2}} \int_{\R^{3}} B(u-v,\omega) S_{j,k}(t,x,\omega,u^{\ast },v^{\ast})dud\omega,
\end{align*}
where
\begin{align}
&S_{j,k}(t,x,\omega,u^{\ast },v^{\ast })\\
=&\int_{\tau}^{t}  K_{k}(x-v(t-t_{0})-v^{*}t_{0})I_{k}(v^{*}) K_{j}(x-v(t-t_{0})-u^{*}t_{0})I_{j}(u^{*})dt_{0}.\notag
\end{align}

We estimate by%
\begin{align}\label{equ:estimate,D+}
\left\Vert D^{+}\right\Vert _{L_{v}^{1}L_{x}^{\infty }} \leqslant& \frac{M^{2-2s}}{N_{2}^{4}}\left\Vert \sum_{k}^{J}\sum_{j}^{J}\int_{
\mathbb{S}^{2}} \int_{\R^{3}} B(u-v,\omega)  S_{j,k}(t,x,\omega,u^{\ast },v^{\ast })dud\omega \right\Vert
_{L_{v}^{1}L_{x}^{\infty }} \\
\leqslant &\frac{M^{2-2s}}{N_{2}^{4}}\sum_{k}^{J}\sum_{j}^{J}\int_{
\mathbb{S}^{2}} \int_{\R^{3}\times \R^{3}} B(u-v,\omega)\left\Vert S_{j,k}(t,x,\omega,u^{\ast
},v^{\ast })\right\Vert _{L_{x}^{\infty }} du dv d\omega\notag\\
=&\frac{M^{2-2s}}{N_{2}^{4}}\sum_{k}^{J}\sum_{j}^{J}\int_{
\mathbb{S}^{2}} \int_{\R^{3}\times \R^{3}} B(u-v,\omega) \left\Vert S_{j,k}(t,x,\omega,u^{\ast
},v^{\ast })\right\Vert _{L_{x}^{\infty }}du^{*}dv^{*} d\omega \notag\\
\lesssim &\frac{M^{2-2s}}{N_{2}^{4}}\sum_{k}^{J}\sum_{j}^{J}\int du^{\ast }\int
dv^{\ast }\int_{\mathbb{S}^{2}}d\omega  \frac{1}{|u^{*}-v^{*}|}\left\Vert S_{j,k}(t,x,\omega,u^{\ast
},v^{\ast })\right\Vert _{L_{x}^{\infty }}\notag
\end{align}
where in the second-to-last equality we used the change of variable, and in the last inequality we used that $B(u-v,\omega)=|u-v|^{-1}\textbf{b}(\cos \theta)\lesssim |u-v|^{-1}$ and $|u-v|=|u^{*}-v^{*}|$.
We note that
\begin{equation*}
v-u^{\ast }=P_{\omega }^{\bot }(v^{\ast }-u^{\ast }),\quad v-v^{\ast }=-P_{\omega
}^{\Vert }(v^{\ast }-u^{\ast }),
\end{equation*}
and hence get
\begin{align*}
&x-v(t-t_{0})-v^{\ast }t_{0} =x-vt-P_{\omega }^{\Vert }(v^{\ast }-u^{\ast
})t_{0},\\
&x-v(t-t_{0})-u^{\ast }t_{0} =x-vt+P_{\omega }^{\bot }(v^{\ast }-u^{\ast
})t_{0}.
\end{align*}
For fixed $u^{*}$ and $v^{*}$, we get
\begin{align}\label{equ:Sjk,estimate}
&S_{j,k}(t,x,\omega,u^{\ast },v^{\ast }) \\
\lesssim &\int_{T_{*}}^{0}\chi \left( MP_{e_{k}}^{\perp }(x-vt-P_{\omega
}^{\Vert }(v^{\ast }-u^{\ast })t_{0})\right)
\chi (\frac{P_{e_{k}}(x-vt-P_{\omega }^{\Vert }(v^{\ast }-u^{\ast })t_{0})}{N_{2}})
I_{k}(v^{*})\notag\\
&\chi \left( MP_{e_{j}}^{\perp }(x-vt+P_{\omega
}^{\bot }(v^{\ast }-u^{\ast })t_{0})\right) \chi (\frac{P_{e_{j}}(x-vt+P_{
\omega }^{\bot }(v^{\ast }-u^{\ast })t_{0})}{N_{2}}) I_{j}(u^{*})dt_{0}\notag\\
\leq &I_{k}(v^{*}) I_{j}(u^{*})E_{k}(t,x,v,\omega,u^{*},v^{*}),\notag
\end{align}
where
\begin{align*}
E_{k}(t,x,v,\omega,u^{*},v^{*})
=:&\int_{T_{*}}^{0}\chi \left( MP_{e_{k}}^{\perp }(x-vt-P_{\omega
}^{\Vert }(v^{\ast }-u^{\ast })t_{0})\right) dt_{0}.
\end{align*}
We split into two cases in terms of the angle $\alpha_{j,k}$ between $e_{j}$ and $e_{k}$.

\textbf{Case $I$: $\alpha_{j,k}\neq 0$.}
(In this case, we have that $\sin \alpha_{j,k}\gtrsim  \frac{1}{MN_{2}}$.)

Now, we get into the analysis of $E_{k}(t,x,v,\omega,u^{*},v^{*})$.
First of all, it gives a trial upper bound that
\begin{align}
E_{k}(t,x,v,\omega,u^{*},v^{*})\leq |T_{*}|\leq  1.
\end{align}

By the radial symmetry and monotonicity of the cutoff function $\chi$, we obtain
\begin{align}\label{equ:chi,inequality}
\int_{\R} \chi(\overrightarrow{n}t_{0}+\overrightarrow{m})dt_{0}\leq \frac{1}{|\overrightarrow{n}|}\int_{\R}\chi(t_{0})dt_{0}.
\end{align}
To see \eqref{equ:chi,inequality}, without loss of generality, we take $\overrightarrow{n}=(0,0,1)$ and $\overrightarrow{m}=(m_{1},m_{2},m_{3})$ to get
\begin{align*}
\int_{\R} \chi(\overrightarrow{n}t_{0}+\overrightarrow{m})dt_{0}=&\int_{\R} \chi\lrs{\sqrt{m_{1}^{2}+m_{2}^{2}+(t_{0}+m)^{2}}}dt_{0}\\
\leq & \int_{\R} \chi\lrs{|t_{0}+m|}dt_{0}=\int_{\R} \chi\lrs{t_{0}}dt_{0}.
\end{align*}
Thus, by \eqref{equ:chi,inequality} we arrive at
\begin{align}
E_{k}(t,x,v,\omega,u^{*},v^{*})
=&\int_{T_{*}}^{0}\chi \left( MP_{e_{k}}^{\perp }(x-vt-P_{\omega
}^{\Vert }(v^{\ast }-u^{\ast })t_{0})\right)dt_{0} \notag\\
\lesssim& \frac{\int \chi(t_{0})dt_{0}}{M|P_{e_{k}}^{\perp}P_{\omega
}^{\Vert }(v^{\ast }-u^{\ast })|}\\
\lesssim & \frac{1}{M\sin \phi_{k}} \frac{1}{|P_{\omega
}^{\Vert }(v^{\ast }-u^{\ast })|},\notag
\end{align}
where $\phi_{k}$ is the angle between $\omega$ and $e_{k}$.
Due to the $v^{*}$-support and $u^{*}$-support, we have
\begin{align}
|v^{*}-u^{*}|^{2}\sim |ae_{k}-be_{j}|^{2}=&(a-b)^{2}+2ab(1-\cos \alpha_{j,k})\\
\geq& N_{2}^{2}(1-\cos \alpha_{j,k})\gtrsim N_{2}^{2}(\sin\alpha_{j,k})^{2}.\notag
\end{align}
Let $\theta$ be the angle between $w$ and $v^{*}-u^{*}$. Then we obtain
\begin{align}
|P_{\omega
}^{\Vert }(v^{\ast }-u^{\ast })|=|v^{\ast }-u^{\ast }|\cos \theta\gtrsim N_{2}\sin \alpha_{j,k} \cos \theta.
\end{align}
Therefore, we get a useful upper bound that
\begin{align}\label{equ:upper bound,Ek}
I_{k}(v^{*}) I_{j}(u^{*})E_{k}(t,x,v,\omega,u^{*},v^{*}) \lesssim &\frac{I_{k}(v^{*}) I_{j}(u^{*})}{M N_{2} \sin \phi_{k} \cos \theta \sin \alpha_{j,k}}.
\end{align}

Now, we are able to establish the effective bound on $E_{k}(t,x,v,\omega,u^{*},v^{*})$.
Set
\begin{align*}
A=\lr{\omega\in \mathbb{S}^{2}:\phi_{k}\leq \frac{1}{MN_{2}}}\bigcup \lr{\frac{\pi}{2}-\theta\leq \frac{1}{MN_{2}}},
\end{align*}
and denote by $A^{c}$ the complementary set of $A$.
With the trivial bound that $E_{k}\leq 1$ on the set $A$, we have
\begin{align}
\int_{\mathbb{S}^{2}} \n{E_{k}(t,x,v,\omega,u^{*},v^{*})}_{L_{x}^{\infty}}d\omega
\leq &\int_{A}1d\omega +\int_{A^{c}}\n{E_{k}(t,x,v,\omega,u^{*},v^{*})}_{L_{x}^{\infty}}d\omega\notag\\
\lesssim & \frac{1}{(MN_{2})^{2}}+\int_{A^{c}}\n{E_{k}(t,x,v,\omega,u^{*},v^{*})}_{L_{x}^{\infty}}d\omega.\label{equ:estimate,Ek,Ac}
\end{align}
For the second term on the right hand side of \eqref{equ:estimate,Ek,Ac}, by the upper bound \eqref{equ:upper bound,Ek}, we get
\begin{align}
& I_{k}(v^{*})I_{j}(u^{*})\int_{A^{c}}\n{E_{k}(t,x,v,\omega,u^{*},v^{*})}_{L_{x}^{\infty}}d\omega\notag\\
\leq&\frac{I_{k}(v^{*})I_{j}(u^{*})}{MN_{2}\sin \alpha_{j,k}}\int_{A^{c}} \frac{1}{\sin \phi_{k}\cos \theta}d\omega\notag\\
\lesssim & \frac{I_{k}(v^{*})I_{j}(u^{*})}{M N_{2}\sin \alpha_{j,k}}\lrc{\int_{A^{c}}\frac{1}{(\sin \phi_{k})^{2}}d\omega+\int_{A^{c}}\frac{1}{(\cos \theta)^{2}}d\omega}.\label{equ:estimate,Ek,Ac,two terms}
\end{align}
For the last two terms on the right hand side of \eqref{equ:estimate,Ek,Ac,two terms}, by the rotational symmetry, we might as well to consider
\begin{align*}
\int_{\mathbb{S}^{2}\bigcap \lr{|\phi|\geq \frac{1}{MN_{2}}}} \frac{1}{(\sin \phi)^{2}}d\omega,
\end{align*}
where $\phi$ is the angle between the $z$-vector $(0,0,1)$ and $\om$. Using the surface integral formula,
\begin{align*}
\int_{\mathbb{S}^{2}\bigcap \lr{|\phi|\geq \frac{1}{MN_{2}}}} \frac{1}{(\sin \phi)^{2}}d\omega\lesssim \int_{|\phi|\geq \frac{1}{MN_{2}}}\frac{|\sin \phi|}{(\sin \phi)^{2}} d\phi\lesssim \int_{|\phi|\geq\frac{1}{MN_{2}}}^{\frac{\pi}{2}} \frac{1}{|\phi|}d\phi \lesssim \ln (MN_{2}).
\end{align*}
Together with \eqref{equ:estimate,Ek,Ac,two terms}, this bound yields
\begin{align}\label{equ:estimate,Ek,Ac,final}
I_{k}(v^{*})I_{j}(u^{*})\int_{A^{c}}\n{E_{k}(t,x,v,\omega,u^{*},v^{*})}_{L_{x}^{\infty}}d\omega
\leq \frac{I_{k}(v^{*})I_{j}(u^{*})\ln (MN_{2})}{MN_{2}\sin \alpha_{j,k}}.
\end{align}
Therefore, combining estimates \eqref{equ:Sjk,estimate}, \eqref{equ:estimate,Ek,Ac} and \eqref{equ:estimate,Ek,Ac,final}, we arrive at
\begin{align}
\int_{\mathbb{S}^{2}}\n{S_{j,k}(t,x,\omega,u^{\ast },v^{\ast })}_{L_{x}^{\infty}}d\omega \lesssim \frac{\ln (MN_{2})}{MN_{2}\sin \alpha_{j,k}}I_{k}(v^{*})I_{j}(u^{*}).
\end{align}
Then, going back to the estimate \eqref{equ:estimate,D+} on $D^{+}$, we have
\begin{align}\label{equ:D+,Lv1}
&\n{D^{+}}_{L_{v}^{1}L_{x}^{\infty}}\\
\lesssim & \frac{M^{2-2s}}{N_{2}^{4}}\sum_{k\neq j}\int du^{*}\int dv^{*}
\int_{\mathbb{S}^{2}} d\omega \frac{1}{|u^{*}-v^{*}|}\n{S_{j,k}(t,x,\omega,u^{\ast },v^{\ast })}_{L_{x}^{\infty}}\notag\\
\lesssim & \frac{M^{2-2s}}{N_{2}^{4}}\sum_{k\neq j} \frac{\ln (MN_{2})}{MN_{2}\sin \alpha_{j,k}}
\int du^{*}\int dv^{*}
\frac{1}{|u^{*}-v^{*}|}I_{k}(v^{*})I_{j}(u^{*}).\notag
\end{align}
In this case, since we have that $\sin \alpha_{j,k}\gtrsim  \frac{1}{MN_{2}}$, we can use estimate \eqref{equ:estimate,IjIk}, which is established in the analysis of $Q^{-}(f_{\mathrm{b}},f_{\mathrm{b}})$, to get
\begin{align}\label{equ:D+,Lv1,IjIk}
\int du^{*}\int dv^{*}
\frac{1}{|u^{*}-v^{*}|}I_{k}(v^{*})I_{j}(u^{*})
\lesssim&\frac{\ln(MN_{2})}{M^{2}} \int I_{j}(v^{*})dv^{*}=\frac{\ln(MN_{2})}{M^{2}}M^{-2}N_{2}.
\end{align}
Consequently, combining estimates \eqref{equ:D+,Lv1} and \eqref{equ:D+,Lv1,IjIk}, we arrive at
\begin{align}\label{equ:Q+,fb,fb,est2,final,case1}
N_{2}^{-1}\n{D^{+}}_{L_{v}^{1}L_{x}^{\infty}}\lesssim& N_{2}^{-1}\frac{M^{2-2s}}{N_{2}^{4}}\sum_{k\neq j}^{J} \frac{\ln (MN_{2})}{MN_{2}\sin \alpha_{j,k}} \frac{\ln (MN_{2})}{M^{4}}N_{2}\\
\lesssim&N_{2}^{-1} \frac{M^{2-2s}}{N_{2}^{4}} \frac{\ln (MN_{2})}{MN_{2}} (MN_{2})^{4} \frac{\ln (MN_{2})}{M^{4}}N_{2}\notag\\
=&N_{2}^{-1}M^{1-2s}\lrc{\ln (MN_{2})}^{2},\notag
\end{align}
where in the second-to-last inequality we have used that
\begin{align*}
\sum_{k\neq j}^{J} \frac{1}{\sin \alpha_{j,k}}=&\sum_{ j}^{J}\sum_{i=1}^{MN_{2}}\sum_{\sin \alpha_{j,k}\sim \frac{i}{MN_{2}}}
\frac{1}{\sin \alpha_{j,k}} \\
\lesssim&\sum_{ j}^{J}\sum_{i=1}^{MN_{2}} MN_{2}\sin \alpha_{j,k}
\frac{1}{\sin \alpha_{j,k}}
\lesssim (MN_{2})^{4}.
\end{align*}
This completes the estimate of the $L_{v}^{1}L_{x}^{\infty}$ norm for $D^{+}$.

\textbf{Case $II$: $\alpha_{j,k}\sim 0$.} (That is, $|j-k|\lesssim 1$.)

In this case, the summands in the double sum $\sum_{k}^{J}\sum_{j}^{J}$ are reduced to $(MN_{2})^{2}$, so we only need to use the trivial bound that
\begin{align*}
\int_{\mathbb{S}^{2}}\n{S_{j,k}(t,x,\omega,u^{\ast },v^{\ast })}_{L_{x}^{\infty}}d\omega\lesssim I_{k}(v^{*})I_{j}(u^{*}).
\end{align*}
Then, with the estimate \eqref{equ:estimate,D+} on $D^{+}$, we use Hardy-Sobolev-Littlewood inequality \eqref{equ:hls,integral} to get
\begin{align}\label{equ:Q+,fb,fb,est2,final,case2}
&\n{D^{+}}_{L_{v}^{1}L_{x}^{\infty}}\\
\lesssim&  \frac{M^{2-2s}}{N_{2}^{4}}\sum_{k}^{J}\sum_{j}^{J}\int du^{\ast }\int
dv^{\ast }\int_{\mathbb{S}^{2}}d\omega  \frac{1}{|u^{*}-v^{*}|}\left\Vert S_{j,k}(t,x,\omega,u^{\ast
},v^{\ast })\right\Vert _{L_{x}^{\infty }}\notag\\
\lesssim&  \frac{M^{2-2s}}{N_{2}^{4}}\sum_{|j-k|\lesssim 1}^{J}
\int du^{*}\int dv^{*}
\frac{1}{|u^{*}-v^{*}|}I_{k}(v^{*})I_{j}(u^{*})\notag\\
\lesssim&  \frac{M^{2-2s}}{N_{2}^{4}}\sum_{|j-k|\lesssim 1}^{J}
\n{I_{j}}_{L^{\frac{6}{5}}}\n{I_{k}}_{L^{\frac{6}{5}}}\notag\\
\lesssim&  \frac{M^{2-2s}}{N_{2}^{4}} (MN_{2})^{2} (M^{-2}N_{2})^{\frac{5}{3}}\notag\\
\lesssim& M^{\frac{2}{3}-2s} N_{2}^{-\frac{1}{3}}.\notag
\end{align}

Combining estimates \eqref{equ:Q+,fb,fb,est2,final,case1} and \eqref{equ:Q+,fb,fb,est2,final,case2} in the two cases, we finally reach
\begin{align}\label{equ:Q+,fb,fb,L1,final}
N_{2}^{-1}\n{D^{+}}_{L_{v}^{1}L_{x}^{\infty}}\lesssim M^{1-2s}N_{2}^{-1}\lrc{\ln(MN_{2})}^{2}.
\end{align}

\textbf{The $N_{2}^{\frac{1}{5}}\Vert \bullet\Vert _{L_{v}^{\frac{5}{3}}L_{x}^{\infty}}$ estimate.}

By the interpolation inequality, we have
\begin{align}\label{equ:Q+,fb,fb,L53}
N_{2}^{\frac{1}{5}}\n{D^{+}}_{L_{v}^{\frac{5}{3}}L_{x}^{\infty}}\leq \lrs{N_{2}^{-1}\n{D^{+}}_{L_{v}^{1}L_{x}^{\infty}}}^{\frac{1}{5}}
\lrs{N_{2}^{\frac{1}{2}}\n{D^{+}}_{L_{v}^{2}L_{x}^{\infty}}}^{\frac{4}{5}}.
\end{align}
For the $L_{v}^{2}L_{x}^{\infty}$ norm, by the bilinear estimate \eqref{equ:bilinear estimate,Q+} for $Q^{+}$ in Lemma \ref{lemma:bilinear estimate,Q+}, we have
\begin{align}\label{equ:Q+,fb,fb,L2}
	N_{2}^{\frac{1}{2}}\n{D^{+}}_{L_{v}^{2}L_{x}^{\infty}}\lesssim& N_{2}^{\frac{1}{2}}|T_{*}|\n{Q^{+}(f_{\mathrm{b}},f_{\mathrm{b}})}_{L_{v}^{2}L_{x}^{\infty}}\\
	\lesssim &N_{2}^{\frac{1}{2}}|T_{*}| \n{f_{\mathrm{b}}}_{L_{v}^{\frac{3}{2}}L_{x}^{\infty}}
	\n{f_{\mathrm{b}}}_{L_{v}^{2}L_{x}^{\infty}}\notag\\
	\lesssim& |T^{*}|M^{1-s}M^{1-s}\notag\\
    \lesssim&  M^{1-s}(\ln \ln \ln M),\notag
\end{align}
where we have used the bounds on $f_{b}$ established in Lemma \ref{lemma:bounds on fb} that
\begin{align*}
\n{f_{\mathrm{b}}}_{L_{v}^{\frac{3}{2}}L_{x}^{\infty}}\leq & N_{2}^{-1}\n{f_{\mathrm{b}}}_{L_{v}^{1}L_{x}^{\infty}}+N_{2}^{\frac{1}{5}}\n{f_{\mathrm{b}}}_{L_{v}^{\frac{5}{3}}L_{x}^{\infty}} \leq \n{f_{\mathrm{b}}}_{Z}\lesssim M^{1-s},\\
N_{2}^{\frac{1}{2}}\n{f_{\mathrm{b}}}_{L_{v}^{2}L_{x}^{\infty}}\lesssim &M^{1-s}.
\end{align*}
Thus, combining estimates \eqref{equ:Q+,fb,fb,L1,final}, \eqref{equ:Q+,fb,fb,L53} and \eqref{equ:Q+,fb,fb,L2}, we reach
\begin{align}
N_{2}^{\frac{1}{5}}\n{D^{+}}_{L_{v}^{\frac{5}{3}}L_{x}^{\infty}}\lesssim N_{2}^{-\frac{1}{5}}M^{1-\frac{6}{5}s}(\ln\ln\ln M)\ln (MN_{2}),
\end{align}
which is sufficient for our goal.

\end{proof}

\subsection{$Z$-norm Bounds on the Correction Term} \label{section:Bounds on the Correction Term}
Recall the equation \eqref{equ:correction term,fc} for the correction term $f_{\mathrm{c}}$ that
\begin{equation}
\left\{
\begin{aligned} &\partial_t f_{\mathrm{c}} + v\cdot \nabla_x f_{\mathrm{c}} = G,
\\ &G= \pm Q^\pm(f_{\mathrm{c}},f_{ \mathrm{a}}) \pm
Q^\pm(f_{\mathrm{a}},f_{\mathrm{c}}) \pm Q^\pm(f_{\mathrm{c}},f_{\mathrm{c}
})-F_{\text{err}}.
 \end{aligned}  \label{equ:fc}
\right.
\end{equation}
For $T_{*}=-M^{s-\frac{d-1}{2}}(\ln \ln \ln M) \leq t\leq 0$, we are looking for the correction term $f_{\mathrm{c}}(t)$ with
\begin{align}
\|f_{\mathrm{c}}(t)\|_{L_v^{2,r_{0}}H_x^{s_{0}}} \lesssim M^{-1/2}.
\end{align}
To achieve it, we apply a perturbation argument and work on the stronger $Z$-norm \eqref{equ:z-norm}.
By interpolation inequality, for $d=2,3$, we indeed have
\begin{align*}
\n{f}_{L_{v}^{2,r_{0}}H_{x}^{s_{0}}}\leq \n{f}_{L_{v}^{2,r_{0}}H_{x}^{\frac{d-1}{2}}}\leq & \n{f}_{L_{v}^{2,r_{0}}L_{x}^{2}}^{\frac{3-d}{2}}\n{\lra{\nabla_{x}}f}_{L_{v}^{2,r_{0}}L_{x}^{2}}^{\frac{d-1}{2}}\\
\leq &
M^{\frac{d-1}{2}}\n{f}_{L_{v}^{2,r_{0}}L_{x}^{2}}+ M^{\frac{d-3}{2}}\n{\lra{\nabla_{x}}f}_{L_{v}^{2,r_{0}}L_{x}^{2}}\leq \n{f}_{Z}.
\end{align*}
Certainly, there are multiple choices of $Z$-norms. As we are fully in the perturbation regime, we expect the correction term $f_{\mathrm{c}}$ to be much smoother and hence we choose the $L_{v}^{2,r_{0}}H_{x}^{1}$ norm. On the other hand, to beat the difficulties caused by singularities of soft potentials, the $L_{v}^{1}L_{x}^{\infty}$ and $L_{v}^{\frac{5}{3}}L_{x}^{\infty}$ norms\footnote{The index $\frac{5}{3}$ is just one of the multiple choices. We choose it, as it would not yield much more difficulties in the estimates on the approximation solution and error terms.} are needed as shown in the following estimate \eqref{equ:z-norm,L1,L53}.

In the section, we first prove a closed estimate for the loss and gain terms in Lemma \ref{lemma:binlinear estimate} and then use it to conclude the existence of small correction term $f_{\mathrm{c}}(t)$ in Proposition \ref{lemma:perturbation}.

\begin{lemma}[Bilinear $Z$-norm estimates for loss/gain operator $Q^\pm$]\label{lemma:binlinear estimate}
\label{L:Z_bilinear} For $f_1$, $f_2$, we have
\begin{equation*}
\|Q^\pm (f_1,f_2) \|_{Z} \lesssim \|f_1\|_{Z} \|f_2\|_{Z}.
\end{equation*}
\end{lemma}
\begin{proof}
We only need to prove that
\begin{align}\label{equ:z-norm,L1,L53}
\n{Q^{\pm}(f_{1},f_{2})}_{Z}\lesssim \n{f_{1}}_{Z}\lrs{\n{f_{2}}_{L_{v}^{1}L_{x}^{\infty}}^{1+\frac{5\ga}{2d}}\n{f_{2}}_{L_{v}^{\frac{5}{3}}L_{x}^{\infty}}^{\frac{-5\ga}{2d}}
+M^{-1}\n{\nabla_{x}f_{2}}_{L_{v}^{1}L_{x}^{\infty}}^{1+\frac{5\ga}{2d}}\n{\nabla_{x}f_{2}}_{L_{v}^{\frac{5}{3}}L_{x}^{\infty}}^{\frac{-5\ga}{2d}}}.
\end{align}
since we have that
\begin{align*}
\n{f}_{L_{v}^{1}L_{x}^{\infty}}^{1+\frac{5\ga}{2d}}\n{f}_{L_{v}^{\frac{5}{3}}L_{x}^{\infty}}^{\frac{-5\ga}{2d}}\leq& N_{2}^{-1}\n{f}_{L_{v}^{1}L_{x}^{\infty}}
+N_{2}^{\frac{2d}{5}+\ga}\n{f}_{L_{v}^{\frac{5}{3}}L_{x}^{\infty}}\leq \n{f}_{Z},\\
M^{-1}\n{\nabla_{x} f}_{L_{v}^{1}L_{x}^{\infty}}^{1+\frac{5\ga}{2d}}
\n{\nabla_{x} f}_{L_{v}^{\frac{5}{3}}L_{x}^{\infty}}^{\frac{-5\ga}{2d}}\leq& M^{-1}N_{2}^{-1}\n{\nabla_{x}f}_{L_{v}^{1}L_{x}^{\infty}}+
M^{-1}N_{2}^{\frac{2d}{5}+\ga}\n{\nabla_{x}f}_{L_{v}^{\frac{5}{3}}L_{x}^{\infty}}\leq \n{f}_{Z}.
\end{align*}

\textbf{The $M^{\frac{d-3}{2}}\Vert \nabla_{x}\bullet \Vert_{L_{v}^{2,r_{0}}L_{x}^{2}}$ and $M^{\frac{d-1}{2}}\Vert \bullet \Vert_{L_{v}^{2,r_{0}}L_{x}^{2}}$ estimates for $Q^{\pm}(f_{1},f_{2})$.}

It suffices to deal with $M^{\frac{d-3}{2}}\Vert \nabla_{x}\bullet \Vert_{L_{v}^{2,r_{0}}L_{x}^{2}}$ norm as the $M^{\frac{d-1}{2}}\Vert \bullet \Vert_{L_{v}^{2,r_{0}}L_{x}^{2}}$ norm can be estimated in a similar way.
For the estimate on $Q^{-}$, we use Leibniz rule and H\"{o}lder inequality to get
\begin{align*}
&M^{\frac{d-3}{2}}\n{\nabla_{x}Q^{-}(f_1,f_2)}_{L_{v}^{2,r_{0}}L_{x}^{2}}\\
\leq& M^{\frac{d-3}{2}}\n{Q^{-}(\nabla_{x}f_1,f_2)}_{L_{v}^{2,r_{0}}L_{x}^{2}}+M^{\frac{d-3}{2}}\n{Q^{-}(f_1,\nabla_{x}f_2)}_{L_{v}^{2,r_{0}}L_{x}^{2}}\\
\lesssim &
M^{\frac{d-3}{2}}\bbn{(\nabla_{x}f_{1})(x,v)\int \frac{f_{2}(x,u)}{|u-v|^{-\ga}} du}_{L_{v}^{2,r_{0}}L_{x}^{2}}+
M^{\frac{d-3}{2}}\bbn{f_{1}(x,v)\int \frac{\nabla_{x}f_{2}(x,u)}{|u-v|^{-\ga}} du}_{L_{v}^{2,r_{0}}L_{x}^{2}}\\
\lesssim& M^{\frac{d-3}{2}}\n{\nabla_{x}f_{1}}_{L_{v}^{2,r_{0}}L_{x}^{2}}\bbn{\int \frac{f_{2}(x,u)}{|u-v|^{-\ga}}du}_{L_{v,x}^{\infty}}+ M^{\frac{d-3}{2}}\n{f_{1}}_{L_{v}^{2,r_{0}}L_{x}^{2}}
\bbn{\int \frac{\nabla_{x}f_{2}(x,u)}{|u-v|^{-\ga}}du}_{L_{v,x}^{\infty}}.
\end{align*}
Then by $L^{\infty}$ estimate \eqref{equ:endpoint estimate,hls} in Lemma \ref{lemma:endpoint estimate,hls}, we obtain
\begin{align*}
&M^{\frac{d-3}{2}}\n{\nabla_{x}Q^{-}(f_1,f_2)}_{L_{v}^{2,r_{0}}L_{x}^{2}}\\
\lesssim&M^{\frac{d-3}{2}} \n{\nabla_{x}f_{1}}_{L_{v}^{2,r_{0}}L_{x}^{2}}\n{f_{2}}_{L_{v}^{1}L_{x}^{\infty}}^{1+\frac{5\ga}{2d}}
\n{f_{2}}_{L_{v}^{\frac{5}{3}}L_{x}^{\infty}}^{\frac{-5\ga}{2d}}+
M^{\frac{d-1}{2}}\n{f_{1}}_{L_{v}^{2,r_{0}}L_{x}^{2}}M^{-1}\n{\nabla_{x} f_{2}}_{L_{v}^{1}L_{x}^{\infty}}^{1+\frac{5\ga}{2d}}
\n{\nabla_{x} f_{2}}_{L_{v}^{\frac{5}{3}}L_{x}^{\infty}}^{\frac{-5\ga}{2d}}\\
\lesssim& \n{f_{1}}_{Z}\n{f_{2}}_{Z}.
\end{align*}

For the estimate on $Q^{+}$, from the conservation of energy that $|v|^{2}+|u|^{2}=|v^{*}|^{2}+|u^{*}|^{2}$, we use Leibniz rule to get
\begin{align*}
|\lra{v}^{r_{0}}\nabla_{x}Q^{+}(f_{1},f_{2})|\lesssim & Q^{+}(\lra{v}^{r_{0}}|\nabla_{x}f_{1}|,|f_{2}|)+
Q^{+}(\lra{v}^{r_{0}}|f_{1}|,|\nabla_{x}f_{2}|)\\
&+Q^{+}(|\nabla_{x}f_{1}|,\lra{v}^{r_{0}}|f_{2}|)+Q^{+}(|f_{1}|,\lra{v}^{r_{0}}|\nabla_{x}f_{2}|).
\end{align*}
Then by bilinear estimate \eqref{equ:bilinear estimate,Q+} on $Q^{+}$, we have
\begin{align*}
&M^{\frac{d-3}{2}}\n{\nabla_{x}Q^{+}(f_{1},f_{2})}_{L_{v}^{2,r_{0}}L_{x}^{2}}\\
\lesssim& M^{\frac{d-3}{2}}\n{\lra{v}^{r_{0}}\nabla_{x}f_{1}}_{L_{v}^{2}L_{x}^{2}}
\n{f_{2}}_{L_{v}^{\frac{d}{d+\ga}}L_{x}^{\infty}}+
M^{\frac{d-1}{2}}\n{\lra{v}^{r_{0}}f_{1}}_{L_{v}^{2}L_{x}^{2}}M^{-1}\n{\nabla_{x} f_{2}}_{L_{v}^{\frac{d}{d+\ga}}L_{x}^{\infty}}\\
&+M^{-1}\n{\nabla_{x}f_{1}}_{L_{v}^{\frac{d}{d+\ga}}L_{x}^{\infty}}
M^{\frac{d-1}{2}}\n{\lra{v}^{r_{0}}f_{2}}_{L_{v}^{2}L_{x}^{2}}
+\n{f_{1}}_{L_{v}^{\frac{d}{d+\ga}}L_{x}^{\infty}}
M^{\frac{d-3}{2}}\n{\lra{v}^{r_{0}}\nabla_{x} f_{2}}_{L_{v}^{2}L_{x}^{2}}.
\end{align*}
By the interpolation inequality that
\begin{align*}
\n{f}_{L_{v}^{\frac{d}{d+\ga}}L_{x}^{\infty}}\leq \n{f}_{L_{v}^{1}L_{x}^{\infty}}^{1+\frac{5\ga}{2d}}\n{f}_{L_{v}^{\frac{5}{3}}L_{x}^{\infty}}^{\frac{-5\ga}{2d}},
\end{align*}
we arrive at
\begin{align*}
M^{\frac{d-3}{2}}\n{\nabla_{x}Q^{+}(f_{1},f_{2})}_{L_{v}^{2,r_{0}}L_{x}^{2}}\lesssim \n{f_{1}}_{Z}\n{f_{2}}_{Z}.
\end{align*}

\textbf{The $N_{2}^{\ga}\Vert \bullet\Vert _{L_{v}^{1}L_{x}^{\infty}}$ and $N_{2}^{\frac{2d}{5}+\ga}\Vert \bullet\Vert _{L_{v}^{\frac{5}{3}}L_{x}^{\infty}}$ estimates for $Q^{\pm}(f_{1},f_{2})$.}

For the estimate on $Q^{-}$ ,
we use H\"{o}lder inequality and the $L^{\infty}$ estimate \eqref{equ:endpoint estimate,hls} to get
\begin{align}\label{equ:Q-,Lv1}
N_{2}^{\ga}\n{Q^{-}(f_{1},f_{2})}_{L_{v}^{1}L_{x}^{\infty}}\lesssim& N_{2}^{\ga}\n{f_{1}}_{L_{v}^{1}L_{x}^{\infty}}\bbn{\int \frac{f_{2}(x,u)}{|u-v|}du}_{L_{v}^{\infty}L_{x}^{\infty}}\\
\lesssim&N_{2}^{\ga}\n{f_{1}}_{L_{v}^{1}L_{x}^{\infty}}\n{f_{2}}_{L_{v}^{1}L_{x}^{\infty}}^{1+\frac{5\ga}{2d}}
\n{f_{2}}_{L_{v}^{\frac{5}{3}}L_{x}^{\infty}}^{\frac{-5\ga}{2d}}\notag\\
\lesssim& \n{f_{1}}_{Z}\n{f_{2}}_{Z}.\notag
\end{align}
In the same way, we also have
\begin{align*}
N_{2}^{\frac{2d}{5}+\ga}\n{Q^{-}(f_{1},f_{2})}_{L_{v}^{\frac{5}{3}}L_{x}^{\infty}}
\lesssim& N_{2}^{\frac{2d}{5}+\ga}\n{f_{1}}_{L_{v}^{\frac{5}{3}}L_{x}^{\infty}}
\n{f_{2}}_{L_{v}^{1}L_{x}^{\infty}}^{1+\frac{5\ga}{2d}}\n{f_{2}}_{L_{v}^{\frac{5}{3}}L_{x}^{\infty}}^{\frac{-5\ga}{2d}}\\
\lesssim& \n{f_{1}}_{Z}\n{f_{2}}_{Z}.
\end{align*}

For the estimate on $Q^{+}$, by the bilinear estimate \eqref{equ:endpoint estimate,Q+,f} for $Q^{+}$ in Lemma \ref{lemma:endpoint estimate,Q+}, we have
\begin{align}\label{equ:Q+,Lv1}
N_{2}^{\ga}\n{Q^{+}(f_{1},f_{2})}_{L_{v}^{1}L_{x}^{\infty}}\lesssim N_{2}^{\ga}\n{f_{1}}_{L_{v}^{1}L_{x}^{\infty}}\n{f_{2}}_{L_{v}^{1}L_{x}^{\infty}}^{1+\frac{5\ga}{2d}}
\n{f_{2}}_{L_{v}^{\frac{5}{3}}L_{x}^{\infty}}^{\frac{-5\ga}{2d}}
\lesssim \n{f_{1}}_{Z}\n{f_{2}}_{Z}.
\end{align}
Similarly, by the bilinear estimate \eqref{equ:bilinear estimate,Q+} in Lemma \ref{lemma:bilinear estimate,Q+}, we obtain
\begin{align*}
N_{2}^{\frac{2d}{5}+\ga}\n{Q^{+}(f_{1},f_{2})}_{L_{v}^{\frac{5}{3}}L_{x}^{\infty}}\lesssim N_{2}^{\frac{2d}{5}+\ga}\n{f_{1}}_{L_{v}^{\frac{5}{3}}L_{x}^{\infty}}\n{f_{2}}_{L_{v}^{1}L_{x}^{\infty}}^{1+\frac{5\ga}{2d}}
\n{f_{2}}_{L_{v}^{\frac{5}{3}}L_{x}^{\infty}}^{\frac{-5\ga}{2d}}
\lesssim \n{f_{1}}_{Z}\n{f_{2}}_{Z}.
\end{align*}

\textbf{The $M^{-1}N_{2}^{\ga}\Vert \nabla_{x}\bullet\Vert _{L_{v}^{1}L_{x}^{\infty}}$ and $M^{-1}N_{2}^{\frac{2d}{5}+\ga}\Vert \nabla_{x}\bullet\Vert _{L_{v}^{\frac{5}{3}}L_{x}^{\infty}}$ estimates for $Q^{\pm}(f_{1},f_{2})$.}

For the estimate on $Q^{-}$,
in a similar way to \eqref{equ:Q-,Lv1}, we use the Leibniz rule to get
\begin{align*}
&M^{-1}N_{2}^{\ga}\Vert \nabla_{x}Q^{-}(f_{1},f_{2})\Vert _{L_{v}^{1}L_{x}^{\infty}}\\
\leq& M^{-1}N_{2}^{\ga}\lrs{\n{Q^{-}(\nabla_{x}f_{1},f_{2})}_{L_{v}^{1}L_{x}^{\infty}}
+\n{Q^{-}(f_{1},\nabla_{x} f_{2})}_{L_{v}^{1}L_{x}^{\infty}}}\\
\lesssim &M^{-1}N_{2}^{\ga}\n{\nabla_{x}f_{1}}_{L_{v}^{1}L_{x}^{\infty}}
\n{f_{2}}_{L_{v}^{1}L_{x}^{\infty}}^{1+\frac{5\ga}{2d}}\n{f_{2}}_{L_{v}^{\frac{5}{3}}L_{x}^{\infty}}^{\frac{-5\ga}{2d}}\\
&+ N_{2}^{\ga}\n{f_{1}}_{L_{v}^{1}L_{x}^{\infty}} M^{-1}\n{\nabla_{x}f_{2}}_{L_{v}^{1}L_{x}^{\infty}}^{1+\frac{5\ga}{2d}}\n{\nabla_{x}f_{2}}_{L_{v}^{\frac{5}{3}}L_{x}^{\infty}}^{\frac{-5\ga}{2d}}\\
\lesssim& \n{f_{1}}_{Z}\n{f_{2}}_{Z}.
\end{align*}
The same also holds for the $M^{-1}N_{2}^{\frac{2d}{5}+\ga}\Vert \nabla_{x} Q^{-}(f_{1},f_{2})\Vert _{L_{v}^{\frac{5}{3}}L_{x}^{\infty}}$ norm.

For the estimate on $Q^{+}$, in a similar way to \eqref{equ:Q+,Lv1}, we also have
\begin{align*}
&M^{-1}N_{2}^{\ga}\n{\nabla_{x}Q^{+}(f_{1},f_{2})}_{L_{v}^{1}L_{x}^{\infty}}\\
\leq&
M^{-1}N_{2}^{\ga}\lrs{\n{Q^{+}(\nabla_{x}f_{1},f_{2})}_{L_{v}^{1}L_{x}^{\infty}}
+\n{Q^{+}(f_{1},\nabla_{x}f_{2})}_{L_{v}^{1}L_{x}^{\infty}}}\\
\lesssim& M^{-1}N_{2}^{\ga}\n{\nabla_{x}f_{1}}_{L_{v}^{1}L_{x}^{\infty}}
\n{f_{2}}_{L_{v}^{1}L_{x}^{\infty}}^{1+\frac{5\ga}{2d}}
\n{f_{2}}_{L_{v}^{\frac{5}{3}}L_{x}^{\infty}}^{\frac{-5\ga}{2d}}\\
&+ N_{2}^{\ga}\n{f_{1}}_{L_{v}^{1}L_{x}^{\infty}}
M^{-1}\n{\nabla_{x}f_{2}}_{L_{v}^{1}L_{x}^{\infty}}^{1+\frac{5\ga}{2d}}
\n{\nabla_{x}f_{2}}_{L_{v}^{\frac{5}{3}}L_{x}^{\infty}}^{\frac{-5\ga}{2d}}\\
\lesssim& \n{f_{1}}_{Z}\n{f_{2}}_{Z}.
\end{align*}
The estimate for the $M^{-1}N_{2}^{\frac{2d}{5}+\ga}\Vert \nabla_{x} Q^{+}(f_{1},f_{2})\Vert _{L_{v}^{\frac{5}{3}}L_{x}^{\infty}}$
norm follows the same way by using bilinear estimate \eqref{equ:bilinear estimate,Q+} in Lemma \ref{lemma:bilinear estimate,Q+}.

\end{proof}

Now, we take a perturbation argument to generate the correction term $f_{\mathrm{c}}(t)$ using the $Z$-norm bounds on $f_{\mathrm{a}}$ in Proposition \ref{lemma:z-norm bounds on fa} and the $Z$-norm bounds on $F_{\mathrm{err}}$ in Proposition \ref{lemma:bounds on ferr}.

\begin{proposition}\label{lemma:perturbation} Suppose that $f_{\mathrm{c}}$
	solves \eqref{equ:fc} with $f_{\mathrm{c}}(0)=0$. Then for all $t$ such
	that
	\begin{equation*}
		T_{\ast }=- M^{s-\frac{d-1}{2}}(\ln \ln \ln M)\leq t\leq 0,
	\end{equation*}%
	we have the bound
	\begin{equation}
		\Vert f_{\mathrm{c}}(t)\Vert_{Z}\lesssim M^{-1/2}.  \label{E:fc_bound2}
	\end{equation}
\end{proposition}

\begin{proof}
	Let the time interval $T_{\ast }\leq t\leq 0$ be partitioned as
	\begin{equation*}
		T_{\ast }=T_{n}<T_{n-1}<T_{n-2}<\cdots <T_{2}<T_{1}<T_{0}=0
	\end{equation*}%
	where
$$T_{j}= \frac{- jM^{s-\frac{d-1}{2}}}{\sqrt{\ln M}},\quad n= \sqrt{\ln M} (\ln\ln \ln M).$$ Thus, the length
	of each time interval $I_{j}=[T_{j+1},T_{j}]$ is
	\begin{equation*}
		|I_{j}|=  \frac{M^{s-\frac{d-1}{2}}}{\sqrt{\ln M}}.
	\end{equation*}
	For $t\in I_{j}=[T_{j+1},T_{j}]$, we rewrite the equation \eqref{equ:fc} in Duhamel form
\begin{align*}
f_{\mathrm{c}}(T_{j}+t)=e^{-(t-T_{j})v\cdot \nabla_{x}}f_{\mathrm{c}}(T_{j})+\int_{T_{j}}^{t}e^{-(t-t_{0})v\cdot \nabla_{x}}G(t_{0})dt_{0}
\end{align*}
with $f_{\mathrm{c}}(T_{0})=0$.
	Applying the $Z$-norm,
	\begin{align*}
		\Vert f_{\mathrm{c}}\Vert _{L_{I_{j}}^{\infty }Z} \leq& \n{f_{\mathrm{c}}(T_{j})}_{Z}
+\bbn{\int_{T_{j}}^{t}e^{-(t-t_{0})v\cdot \nabla_{x}}G(t_{0})dt_{0}}_{L_{I_{j}}^{\infty}Z} \\
		 \leq& \n{f_{\mathrm{c}}(T_{j})}_{Z}+ |I_{j}|\Vert Q^{\pm }(f_{\mathrm{c}},f_{\mathrm{a}})\Vert
		_{L_{I_{j}}^{\infty }Z}+|I_{j}|\Vert Q^{\pm }(f_{\mathrm{a}},f_{\mathrm{c}%
		})\Vert _{L_{I_{j}}^{\infty }Z}+|I_{j}|\Vert Q^{\pm }(f_{\mathrm{c}},f_{\text{c%
		}})\Vert _{L_{I_{j}}^{\infty }Z} \\
		& +\bbn{\int_{T_{j}}^{t}e^{-(t-t_{0})v\cdot \nabla_{x}}F_{\mathrm{err}}(t_{0})dt_{0}}_{L_{I_{j}}^{\infty}Z}.
	\end{align*}%
	For these terms on the second line, we apply the bilinear estimate in Lemma \ref{lemma:binlinear estimate}, and then the estimate \eqref{equ:z-norm estimate for fa} on $\n{f_{\mathrm{a}}}_{L_{I_{j}}^{\infty }Z}$ from Lemma \ref{lemma:z-norm bounds on fa}. For the $F_{\mathrm{err}}$ term on the
	last line, we use the estimate \eqref{equ:Ferr_bound} in Proposition \ref{lemma:bounds on ferr}. Then we have
	\begin{equation*}
		\Vert f_{\mathrm{c}}\Vert _{L_{I_{j}}^{\infty }Z}\leq \Vert f_{\mathrm{c}%
		}(T_{j})\Vert _{Z}+ \frac{C(\ln \ln M)^{2}}{\sqrt{\ln M}}\Vert f_{\mathrm{c}}\Vert _{L_{I_{j}}^{\infty
			}Z}+ \frac{CM^{s-\frac{d-1}{2}}}{\sqrt{\ln M}}\Vert f_{\mathrm{c}}\Vert _{L_{I_{j}}^{\infty
			}Z}^{2}+CM^{-1},
	\end{equation*}
where $C$ is some absolute constant. Absorbing the $\Vert f_{\mathrm{c}}\Vert
_{L_{I_{j}}^{\infty }Z}$ term on the right gives
\begin{equation*}
\Vert f_{\mathrm{c}}\Vert _{L_{I_{j}}^{\infty }Z}\leq 2\Vert f_{\mathrm{c}
}(T_{j})\Vert _{Z}+2CM^{-1}.
\end{equation*}
Applying this successively for $j=0,1,\ldots $, we obtain
\begin{equation*}
\Vert f_{\mathrm{c}}\Vert _{L_{I_{j}}^{\infty }Z}\leq (2^{j+1}-1)2CM^{-1}.
\end{equation*}%
With $j=n=\sqrt{\ln M}(\ln\ln \ln M)$, we arrive at
\begin{equation*}
\Vert f_{\mathrm{c}}(T_{\ast })\Vert _{Z}\leq \frac{Ce^{\sqrt{\ln M}\ln \ln M}}{M} = \frac{Ce^{\sqrt{\ln M}\ln \ln M}}{e^{ \ln M}}\leq
M^{-1 /2}\ll 1.
\end{equation*}

\end{proof}

\subsection{Proof of Illposedness} \label{section:Proof of Illposedness}
We get into the proof the ill-posedness.

\begin{proof}[\textbf{Proof of Ill-posedness in Theorem $\ref{thm:main theorem}$}]
Let
$$f_{\mathrm{ex}}(t)=f_{\mathrm{r}}(t)+f_{\mathrm{b}}(t)+f_{\mathrm{c}}(t),$$
 with $f_{\mathrm{c}}(t)$ given in Proposition \ref{lemma:perturbation}. By the upper and lower bounds in Lemma \ref{lemma:bounds on fr} that
$$\n{f_{\mathrm{r}}(0)}_{L_{v}^{2,r_{0}}H_{x}^{s_{0}}}\lesssim \frac{1}{\ln\ln M},\quad \n{f_{\mathrm{r}}(T_{*})}_{L_{v}^{2,r_{0}}H_{x}^{s_{0}}}\gtrsim 1,\quad $$ we can take $t_{0}\in [T_{*},0]$ such that $\n{f_{\mathrm{r}}(t_{0})}_{L_{v}^{2,r_{0}}H_{x}^{s_{0}}}=1$.
Note that
\begin{align*}
\n{f_{\mathrm{r}}}_{Z}\lesssim M^{\frac{d-1}{2}-s} (\ln \ln M)^{2},\quad \n{v\cdot \nabla_{x}f_{\mathrm{r}}}_{Z}\ll M^{-1},\quad \n{Q^{\pm}(f_{\mathrm{r}},f_{\mathrm{r}})}_{Z}\ll M^{-1},
\end{align*}
which are established in Lemma \ref{lemma:z-norm bounds on fr} and Section \ref{section:Bounds on the Error Terms}.
Therefore, by the same perturbation argument in Lemma \ref{lemma:perturbation}, we generate an exact solution $g_{\mathrm{ex}}(t)$ to Boltzmann equation
\begin{equation*}
g_{\mathrm{ex}}(t)=f_{\mathrm{r}}(t_{0})+g_{\mathrm{c}}(t),
\end{equation*}
with $g_{\mathrm{c}}(0)=0$. This gives that
\begin{equation*}
\left\{
\begin{aligned}
&\n{g_{\mathrm{ex}}(0)}_{L_{v}^{2,r_{0}}H_{x}^{s_{0}}}=\n{f_{\mathrm{r}}(t_{0})}_{L_{v}^{2,r_{0}}H_{x}^{s_{0}}}=1,\\
&\n{ g_{\mathrm{c}}(t)}_{L^{\infty}([T_{*},0];Z)}\lesssim M^{-\frac{1}{2}}.
\end{aligned}
\right.
\end{equation*}

Now, we have two solutions with the decompositions
\begin{equation*}
\left\{
\begin{aligned}
f_{\mathrm{ex}}(t)=&f_{\mathrm{r}}(t)+f_{\mathrm{b}}(t)+f_{\mathrm{c}}(t),\\
g_{\mathrm{ex}}(t)=&f_{\mathrm{r}}(t_{0})+g_{c}(t),
\end{aligned}
\right.
\end{equation*}
which gives
\begin{equation*}
f_{\mathrm{ex}}(t)-g_{\mathrm{ex}}(t)=(f_{\mathrm{r}}(t)-f_{\mathrm{r}}(t_{0}))+f_{
\mathrm{b}}(t)+f_{\mathrm{c}}(t)-g_{\mathrm{c}}(t).
\end{equation*}%
For $t\in [T_{*},0]$, by Lemma \ref{lemma:Hs,bounds on fb} and Proposition \ref{lemma:perturbation},
we have
\begin{align*}
&\Vert f_{\mathrm{b}}(t)\Vert _{L_{v}^{2,r_{0}}H_{x}^{s_{0}}}\lesssim
M^{s_{0}-s}= \frac{1}{\ln \ln M},\\
&\Vert f_{\mathrm{c}}(t)\Vert _{L_{v}^{2,r_{0}}H_{x}^{s_{0}}}\leq \Vert f_{
\mathrm{c}}(t)\Vert _{Z}\lesssim M^{-\frac{1}{2}},\\
&\n{g_{\mathrm{c}}(t)}_{L_{v}^{2,r_{0}}H_{x}^{s_{0}}}\leq \n{g_{\mathrm{c}}(t)}_{Z}\lesssim M^{-\frac{1}{2}}.
\end{align*}
Thus, we obtain
\begin{equation*}
\Vert f_{\mathrm{ex}}(t_{0})-g_{\mathrm{ex}}(t_{0})\Vert
_{L_{v}^{2,r_{0}}H_{x}^{s_{0}}}\lesssim \frac{1}{\ln\ln M},
\end{equation*}
and
\begin{equation*}
\Vert f_{\mathrm{ex}}(0)-g_{\mathrm{ex}}(0)\Vert
_{L_{v}^{2,r_{0}}H_{x}^{s_{0}}}\sim \Vert f_{\mathrm{a}}(0)-f_{\mathrm{r}
}(t_{0})\Vert _{L_{v}^{2,r_{0}}H_{x}^{s_{0}}}\sim  \n{f_{\mathrm{r}
}(t_{0})}_{L_{v}^{2,r_{0}}H_{x}^{s_{0}}}=1,
\end{equation*}
where we have used that $\n{f_{\mathrm{a}}(0)}_{L_{v}^{2,r_{0}}H_{x}^{s_{0}}}\lesssim \frac{1}{\ln \ln M}$. Hence, we complete the proof.
\end{proof}
\begin{remark}
We actually have found an exact solution $f_{\mathrm{ex}}(t)$ which satisfies the norm deflation property. This is the key to conclude the failure of uniform continuity of the data-to-solution map.
\end{remark}

In the end, we prove Corollary \ref{lemma:ill-posedness,kernels}.
\begin{proof}[\textbf{Proof of Corollary $\ref{lemma:ill-posedness,kernels}$}]
Recall the kernel
\begin{align}\label{equ:kernels}
B(u-v,\omega)=\lrs{1_{\lr{|u-v|\leq 1}}|u-v|+1_{\lr{|u-v|\geq 1}}|u-v|^{-1}} \textbf{b}(\frac{u-v}{|u-v|}\cdot \omega),
\end{align}
and notice the pointwise upper bound estimate
\begin{align}\label{equ:pointwise upper bound,kernels}
\lrs{1_{\lr{|u-v|\leq 1}}|u-v|+1_{\lr{|u-v|\geq 1}}|u-v|^{-1}} \textbf{b}(\frac{u-v}{|u-v|}\cdot \omega)\leq \frac{1}{|u-v|} \textbf{b}(\frac{u-v}{|u-v|}\cdot \omega).
\end{align}
Therefore, for the kernel $B(u-v,\omega)$ in \eqref{equ:kernels}, all the same upper bound estimates on $f_{\mathrm{b}}$, $f_{\mathrm{r}}$, $f_{\mathrm{a}}$, $F_{\mathrm{err}}$, and $f_{\mathrm{c}}$ follow from the pointwise upper bound estimate \eqref{equ:pointwise upper bound,kernels}.
The only one lower bound on $f_{\mathrm{r}}$ we need is given in Remark \ref{remark:lower bound,kernels}. Then by repeating the proof of ill-posedness for the endpoint case $(d,\ga,r_{0})=(3,-1,0)$ in Theorem \ref{thm:main theorem}, we complete the proof of Corollary \ref{lemma:ill-posedness,kernels}.
\end{proof}

\appendix
\section{Sobolev-type and Time-independent Bilinear Estimates}\label{section:Sobolev-type and Time-independent Bilinear Estimates}

\begin{lemma}[Fractional Leibniz rule, {\cite{gulisashvili1996exact}}]\label{lemma:generalized leibniz rule}
Suppose $1<r<\infty$, $s\geq 0$ and $\frac{1}{r}=\frac{1}{p_{i}}+\frac{1}{q_{i}}$ with $i=1,2$, $1<q_{1}\leq \infty$, $1<p_{2}\leq \infty$.
Then
\begin{align}
\n{\lra{\nabla_{x}}^{s}(fg)}_{L^{r}}\leq C\n{\lra{\nabla_{x}}^{s}f}_{L^{p_{1}}}\n{g}_{L^{q_{1}}}+\n{f}_{L^{p_{2}}}
\n{\lra{\nabla_{x}}^{s}g}_{L^{q_{2}}}
\end{align}
where the constant $C$ depends on all of the parameters.
\end{lemma}

Next, we present the standard Hardy-Littlewood-Sobolev inequality, which is widely used in our various estimates for the soft potential case.
\begin{lemma}\label{lemma:hardy-littlewood-sobolev inequality,dual form}
Let $p>1$, $r>1$ and $-d<\gamma \leq 0$ with
$$\frac{1}{p}+\frac{1}{r}=2+\frac{\gamma}{d}.$$
Let $f\in L^{p}(\R^{d})$ and $h\in L^{r}(\R^{d})$,
then there exists a constant $C(d, \gamma, p)$, independent of $f$ and $h$, such that
\begin{align} \label{equ:hls,integral}
\int_{\mathbb{R}^d} \int_{\mathbb{R}^d} f(x)|x-y|^{\ga} h(y) \mathrm{d} x \mathrm{d} y\leq C(d, \gamma, p,r)\|f\|_p\|h\|_r.
\end{align}
In particular, for $p>1$, $q>1$ with
$$1+\frac{1}{q}+\frac{\gamma}{d}=\frac{1}{p},$$
we also have
\begin{align}\label{equ:hardy-littlewood-sobolev inequality}
\n{f*|\cdot|^{\gamma}}_{L^{q}}\leq C(d, \gamma, p,q)\n{f}_{L^{p}}.
\end{align}
\end{lemma}

\begin{lemma}[Endpoint case]\label{lemma:endpoint estimate,hls}
Let $d\geq 2$, $-d<\ga\leq 0$, and $1 \leq p< \frac{d}{d+\ga}<q\leq \infty$. Then for $f \in L^p\left(\mathbb{R}^d\right) \cap L^q\left(\mathbb{R}^d\right)$,
it holds that
\begin{align}\label{equ:endpoint estimate,hls,L1}
\int  |x|^{\ga}|f(x)|dx \lesssim\n{f}_{L^{p}}^{\frac{\frac{q-1}{q}+\frac{\ga}{d}}{\frac{1}{p}-\frac{1}{q}}}
\n{f}_{L^{q}}^{\frac{-\frac{\ga}{d}-\frac{p-1}{p}}{\frac{1}{p}-\frac{1}{q}}}.
\end{align}
In particular, when $\ga=-1$, $p=1$, and $q>\frac{d}{d-1}$, we have
\begin{align}\label{equ:endpoint estimate,hls}
\bbn{\int  \frac{f(y)}{|x-y|}dy}_{L_{x}^{\infty}} \lesssim\n{f}_{L^{1}}^{1-\frac{1}{d\left(1-\frac{1}{q}\right)}}\n{f}_{L^{q}}^{\frac{1}{d\left(1-\frac{1}{q}\right)}}.
\end{align}

\end{lemma}
\begin{proof}
The endpoint case is also known. For completeness, we include a proof.
We split the integral into two parts and use H\"{o}lder inequality to get
\begin{align*}
\int_{\R^{d}}|x|^{\ga}|f(x)|dx\leq& \int_{|x|\leq \eta}|x|^{\ga}|f(x)|dx+\int_{|x|> \eta}|x|^{\ga}|f(x)|dx\\
\lesssim &\n{f}_{L^{q}}\eta^{\frac{d}{q'}+\ga}+\n{f}_{L^{p}}\eta^{\frac{d}{p'}+\ga},
\end{align*}
where $p'=\frac{p}{p-1}$ and $q'=\frac{q}{q-1}$. Optimizing the choice of $\eta$ gives the desired estimate that
\begin{align*}
\int_{\R^{d}}|x|^{\ga}|f(x)|dx \lesssim &\n{f}_{L^{p}}^{\frac{\frac{q-1}{q}+\frac{\ga}{d}}{\frac{1}{p}-\frac{1}{q}}}
\n{f}_{L^{q}}^{\frac{-\frac{\ga}{d}-\frac{p-1}{p}}{\frac{1}{p}-\frac{1}{q}}}.
\end{align*}

\end{proof}

The following parts focus on time-independent bilinear estimates for gain/loss terms.

\begin{lemma}[{\cite[Theorem 2, Corollary 9]{alonso2010convolution}}] \label{lemma:bilinear estimate,Q+}
Let $1<p, q, r<\infty$ and $-d<\gamma \leq 0$ with
$$\frac{1}{p}+\frac{1}{q}=1+\frac{\gamma}{d}+\frac{1}{r}.$$
Assume the collision kernel
$$
B\left(u-v,\omega\right)=\left|u-v\right|^\gamma \textbf{b}(\frac{u-v}{|u-v|}\cdot \omega),
$$
with $\textbf{b}(\frac{u-v}{|u-v|}\cdot \omega)$ satisfying Grad's angular cutoff assumption.
Then, it holds that
\begin{align}
&\left\|Q^{+}(f, g)\right\|_{L^{r}\left(\mathbb{R}^{d}\right)} \leq C\|f\|_{L^{p}\left(\mathbb{R}^{d}\right)}\|g\|_{L^{q}\left(\mathbb{R}^{d}\right)},\label{equ:bilinear estimate,Q+}\\
&\left\|Q^{-}(f, g)\right\|_{L^{r}\left(\mathbb{R}^{d}\right)} \leq C\|f\|_{L^{p}\left(\mathbb{R}^{d}\right)}\|g\|_{L^{q}\left(\mathbb{R}^{d}\right)},\quad p>r.\label{equ:bilinear estimate,Q-}
\end{align}
\end{lemma}

\begin{lemma}[$L^{1}$ endpoint estimate for $Q^{+}$]\label{lemma:endpoint estimate,Q+}
For $\ga=-1$, we have
\begin{align}\label{equ:endpoint estimate,Q+,f}
\n{Q^{+}(f,g)}_{L^{1}}
\leq& \n{f}_{L^{1}}\n{g}_{L^1}^{1-\frac{1}{d\left(1-\frac{1}{p}\right)}}\n{g}_{L^p}^{\frac{1}{d\left(1-\frac{1}{p}\right)}},
\end{align}
\begin{align}\label{equ:endpoint estimate,Q+,g}
\n{Q^{+}(f,g)}_{L^{1}}
\leq& \|f\|_{L^1}^{1-\frac{1}{d\left(1-\frac{1}{p}\right)}}\|f\|_{L^p}^{\frac{1}{d\left(1-\frac{1}{p}\right)}}\n{g}_{L^{1}}.
\end{align}
\end{lemma}
\begin{proof}
By the change of variable, we have
\begin{align*}
\n{Q^{+}(f,g)}_{L^{1}} \lesssim& \int_{\mathbb{S}^{d-1}} \int_{\R^{2d}}\frac{|f(u^{*})g(v^{*})|}{|u^{*}-v^{*}|} du dv d\omega\\
=&\int_{\mathbb{S}^{d-1}} \int_{\R^{2d}}\frac{|f(u^{*})g(v^{*})|}{|u^{*}-v^{*}|} du^{*} dv^{*} d\omega\\
\lesssim& \n{f}_{L^{1}}\bbn{\int \frac{|g(v^*)|}{|u^{*}-v^{*}|}dv^{*}}_{L^{\infty}}.
\end{align*}
Using the $L^{\infty}$ estimate \eqref{equ:endpoint estimate,hls}, we get
\begin{align*}
\n{Q^{+}(f,g)}_{L^{1}}\lesssim \n{f}_{L^{1}}\n{g}_{L^1}^{1-\frac{1}{d\left(1-\frac{1}{p}\right)}}\n{g}_{L^p}^{\frac{1}{d\left(1-\frac{1}{p}\right)}}.
\end{align*}
In the same way, we also obtain estimate \eqref{equ:endpoint estimate,Q+,g}.

\end{proof}

\section{Strichartz Estimates}\label{section:Strichartz Estimates}
Recall the abstract Strichartz estimates.
\begin{theorem}[{\cite[Theorem 1.2]{keel1998endpoint}}]\label{lemma:strichartz estimate,keel-tao}
Suppose that for each time
$t$ we have an operator $U(t)$ such that
\begin{align*}
\n{U(t)f}_{L_{x}^{2}}\lesssim& \n{f}_{L_{x}^{2}},\\
\n{U(t)(U(s)^{*})f}_{L_{x}^{\infty}}\lesssim& |t-s|^{-\sigma}\n{f}_{L_{x}^{1}}.
\end{align*}
Then it holds that
\begin{align}\label{equ:strichartz,usual one}
\n{U(t)f}_{L_{t}^{q}L_{x}^{p}}\lesssim \n{f}_{L_{x}^{2}},
\end{align}
for all sharp $\sigma$-admissible exponent pair that
\begin{align}
\frac{2}{q}+\frac{2\sigma}{p}=\sigma,\quad q\geq 2,\ \sigma>1.
\end{align}

\end{theorem}

The symmetric hyperbolic Schr\"{o}dinger equation is
\begin{equation}
\left\{
\begin{aligned}
i\pa_{t}\phi+\nabla_{\xi}\cdot \nabla_{x}\phi=&0,\\
\phi(0)=&\phi_{0}.
\end{aligned}
\right.
\end{equation}
Note that the linear propagator $U(t)=e^{it\nabla_{\xi}\cdot \nabla_{x}}$ satisfies the energy and dispersive estimates
\begin{equation}
\begin{aligned}
&\n{e^{it\nabla_{\xi}\cdot\nabla_{x}}\phi_{0}}_{L_{x\xi}^{2}}\lesssim \n{\phi_{0}}_{L_{x\xi}^{2}},\\
&\n{e^{it\nabla_{\xi}\cdot\nabla_{x}}\phi_{0}}_{L_{x\xi}^{\infty}}\lesssim t^{-d}\n{\phi_{0}}_{L_{x\xi}^{1}}.
\end{aligned}
\end{equation}
Then by Theorem \ref{lemma:strichartz estimate,keel-tao}, this gives a Strichartz estimate that
\begin{align}\label{equ:strichartz estimate,linear}
\n{e^{it\nabla_{\xi}\cdot \nabla_{x}}\phi_{0}}_{L_{t}^{q}L_{x\xi}^{p}}\lesssim \n{\phi_{0}}_{L_{x\xi}^{2}},\quad \frac{2}{q}+\frac{2d}{p}=d,\quad q\geq 2,\ d\geq2.
\end{align}

\bibliographystyle{abbrv}
\bibliography{references}

\begin{thebibliography}{10}

\bibitem{alexandre2000entropy}
R.~Alexandre, L.~Desvillettes, C.~Villani, and B.~Wennberg.
\newblock Entropy dissipation and long-range interactions.
\newblock {\em Arch. Ration. Mech. Anal.}, 152(4):327--355, 2000.

\bibitem{alexandre2011global}
R.~Alexandre, Y.~Morimoto, S.~Ukai, C.-J. Xu, and T.~Yang.
\newblock Global existence and full regularity of the {B}oltzmann equation
  without angular cutoff.
\newblock {\em Comm. Math. Phys.}, 304(2):513--581, 2011.

\bibitem{alexandre2013local}
R.~Alexandre, Y.~Morimoto, S.~Ukai, C.-J. Xu, and T.~Yang.
\newblock Local existence with mild regularity for the {B}oltzmann equation.
\newblock {\em Kinet. Relat. Models}, 6(4):1011--1041, 2013.

\bibitem{alonso2010estimates}
R.~J. Alonso and E.~Carneiro.
\newblock Estimates for the {B}oltzmann collision operator via radial symmetry
  and {F}ourier transform.
\newblock {\em Adv. Math.}, 223(2):511--528, 2010.

\bibitem{alonso2010convolution}
R.~J. Alonso, E.~Carneiro, and I.~M. Gamba.
\newblock Convolution inequalities for the {B}oltzmann collision operator.
\newblock {\em Comm. Math. Phys.}, 298(2):293--322, 2010.

\bibitem{ampatzoglou2022global}
I.~Ampatzoglou, I.~M. Gamba, N.~Pavlovi\'{c}, and M.~Taskovi\'{c}.
\newblock Global well-posedness of a binary-ternary {B}oltzmann equation.
\newblock {\em Ann. Inst. H. Poincar\'{e} C Anal. Non Lin\'{e}aire},
  39(2):327--369, 2022.

\bibitem{arsenio2011global}
D.~Ars\'{e}nio.
\newblock On the global existence of mild solutions to the {B}oltzmann equation
  for small data in {$L^D$}.
\newblock {\em Comm. Math. Phys.}, 302(2):453--476, 2011.

\bibitem{beals1983self}
M.~Beals.
\newblock Self-spreading and strength of singularities for solutions to
  semilinear wave equations.
\newblock {\em Ann. of Math. (2)}, 118(1):187--214, 1983.

\bibitem{bourgain1993fourier1}
J.~Bourgain.
\newblock Fourier transform restriction phenomena for certain lattice subsets
  and applications to nonlinear evolution equations. {I}. {S}chr\"{o}dinger
  equations.
\newblock {\em Geom. Funct. Anal.}, 3(2):107--156, 1993.

\bibitem{bourgain1993fourier2}
J.~Bourgain.
\newblock Fourier transform restriction phenomena for certain lattice subsets
  and applications to nonlinear evolution equations. {II}. {T}he
  {K}d{V}-equation.
\newblock {\em Geom. Funct. Anal.}, 3(3):209--262, 1993.

\bibitem{cercignani1988boltzmann}
C.~Cercignani.
\newblock {\em The {B}oltzmann equation and its applications}, volume~67 of
  {\em Applied Mathematical Sciences}.
\newblock Springer-Verlag, New York, 1988.

\bibitem{cercignani1994mathematical}
C.~Cercignani, R.~Illner, and M.~Pulvirenti.
\newblock {\em The mathematical theory of dilute gases}, volume 106 of {\em
  Applied Mathematical Sciences}.
\newblock Springer-Verlag, New York, 1994.

\bibitem{chen2019local}
T.~Chen, R.~Denlinger, and N.~Pavlovi\'{c}.
\newblock Local well-posedness for {B}oltzmann's equation and the {B}oltzmann
  hierarchy via {W}igner transform.
\newblock {\em Comm. Math. Phys.}, 368(1):427--465, 2019.

\bibitem{chen2019moments}
T.~Chen, R.~Denlinger, and N.~Pavlovi\'{c}.
\newblock Moments and regularity for a {B}oltzmann equation via {W}igner
  transform.
\newblock {\em Discrete Contin. Dyn. Syst.}, 39(9):4979--5015, 2019.

\bibitem{chen2021small}
T.~Chen, R.~Denlinger, and N.~Pavlovi\'{c}.
\newblock Small data global well-posedness for a {B}oltzmann equation via
  bilinear spacetime estimates.
\newblock {\em Arch. Ration. Mech. Anal.}, 240(1):327--381, 2021.

\bibitem{chen2015unconditional}
T.~Chen, C.~Hainzl, N.~Pavlovi\'{c}, and R.~Seiringer.
\newblock Unconditional uniqueness for the cubic {G}ross-{P}itaevskii hierarchy
  via quantum de {F}inetti.
\newblock {\em Comm. Pure Appl. Math.}, 68(10):1845--1884, 2015.

\bibitem{chen2014derivation}
T.~Chen and N.~Pavlovi\'{c}.
\newblock Derivation of the cubic {NLS} and {G}ross-{P}itaevskii hierarchy from
  manybody dynamics in {$d=3$} based on spacetime norms.
\newblock {\em Ann. Henri Poincar\'{e}}, 15(3):543--588, 2014.

\bibitem{chen2016collapsing}
X.~Chen.
\newblock Collapsing estimates and the rigorous derivation of the 2d cubic
  nonlinear {S}chr\"{o}dinger equation with anisotropic switchable quadratic
  traps.
\newblock {\em J. Math. Pures Appl. (9)}, 98(4):450--478, 2012.

\bibitem{chen2013rigorous}
X.~Chen.
\newblock On the rigorous derivation of the 3{D} cubic nonlinear
  {S}chr\"{o}dinger equation with a quadratic trap.
\newblock {\em Arch. Ration. Mech. Anal.}, 210(2):365--408, 2013.

\bibitem{chen2016correlation}
X.~Chen and J.~Holmer.
\newblock Correlation structures, many-body scattering processes, and the
  derivation of the {G}ross-{P}itaevskii hierarchy.
\newblock {\em Int. Math. Res. Not. IMRN}, 2016(10):3051--3110, 2016.

\bibitem{chen2016focusing}
X.~Chen and J.~Holmer.
\newblock Focusing quantum many-body dynamics: the rigorous derivation of the
  1{D} focusing cubic nonlinear {S}chr\"{o}dinger equation.
\newblock {\em Arch. Ration. Mech. Anal.}, 221(2):631--676, 2016.

\bibitem{chen2016klainerman}
X.~Chen and J.~Holmer.
\newblock On the {K}lainerman-{M}achedon conjecture for the quantum {BBGKY}
  hierarchy with self-interaction.
\newblock {\em J. Eur. Math. Soc. (JEMS)}, 18(6):1161--1200, 2016.

\bibitem{chen2019derivation}
X.~Chen and J.~Holmer.
\newblock The derivation of the {$\Bbb T^3$} energy-critical {NLS} from quantum
  many-body dynamics.
\newblock {\em Invent. Math.}, 217(2):433--547, 2019.

\bibitem{chen2022quantitative}
X.~Chen and J.~Holmer.
\newblock Quantitative derivation and scattering of the 3{D} cubic {NLS} in the
  energy space.
\newblock {\em Ann. PDE}, 8(2):Paper No. 11, 39, 2022.

\bibitem{chen2022unconditional}
X.~Chen and J.~Holmer.
\newblock Unconditional uniqueness for the energy-critical nonlinear
  {S}chr\"{o}dinger equation on {$\Bbb T^4$}.
\newblock {\em Forum Math. Pi}, 10:Paper No. e3, 49, 2022.

\bibitem{chen2022well}
X.~Chen and J.~Holmer.
\newblock Well/ill-posedness bifurcation for the {B}oltzmann equation with
  constant collision kernel.
\newblock {\em arXiv preprint arXiv:2206.11931}, 2022.

\bibitem{chen2023derivationboltzmann}
X.~Chen and J.~Holmer.
\newblock The derivation of the {B}oltzmann equation from quantum many-body
  dynamics.
\newblock {\em In preparation}, 2023.

\bibitem{chen2023derivation}
X.~Chen, S.~Shen, J.~Wu, and Z.~Zhang.
\newblock The derivation of the compressible {E}uler equation from quantum
  many-body dynamics.
\newblock {\em Peking Math. J., https://doi.org/10.1007/s42543-023-00066-4},
  2023.

\bibitem{CSZ22}
X.~Chen, S.~Shen, and Z.~Zhang.
\newblock The unconditional uniqueness for the energy-supercritical {NLS}.
\newblock {\em Ann. PDE}, 8(2):Paper No. 14, 82, 2022.

\bibitem{chen2023sharp}
X.~Chen, S.~Shen, and Z.~Zhang.
\newblock Sharp global well-posedness and scattering of the {B}oltzmann
  equation.
\newblock {\em In prepartion}, 2023.

\bibitem{christ2003asymptotics}
M.~Christ, J.~Colliander, and T.~Tao.
\newblock Asymptotics, frequency modulation, and low regularity ill-posedness
  for canonical defocusing equations.
\newblock {\em Amer. J. Math.}, 125(6):1235--1293, 2003.

\bibitem{christ2003ill-posedness}
M.~Christ, J.~Colliander, and T.~Tao.
\newblock Ill-posedness for nonlinear {S}chr\"{o}dinger and wave equations.
\newblock {\em arXiv preprint arXiv:0311048}, 2003.

\bibitem{desvillettes2003use}
L.~Desvillettes.
\newblock About the use of the {F}ourier transform for the {B}oltzmann
  equation.
\newblock {\em Riv. Mat. Univ. Parma (7)}, 2*:1--99, 2003.

\bibitem{diperna1989cauchy}
R.~J. DiPerna and P.-L. Lions.
\newblock On the {C}auchy problem for {B}oltzmann equations: global existence
  and weak stability.
\newblock {\em Ann. of Math. (2)}, 130(2):321--366, 1989.

\bibitem{duan2017global}
R.~Duan, F.~Huang, Y.~Wang, and T.~Yang.
\newblock Global well-posedness of the {B}oltzmann equation with large
  amplitude initial data.
\newblock {\em Arch. Ration. Mech. Anal.}, 225(1):375--424, 2017.

\bibitem{duan2021global}
R.~Duan, S.~Liu, S.~Sakamoto, and R.~M. Strain.
\newblock Global mild solutions of the {L}andau and non-cutoff {B}oltzmann
  equations.
\newblock {\em Comm. Pure Appl. Math.}, 74(5):932--1020, 2021.

\bibitem{duan2016global}
R.~Duan, S.~Liu, and J.~Xu.
\newblock Global well-posedness in spatially critical {B}esov space for the
  {B}oltzmann equation.
\newblock {\em Arch. Ration. Mech. Anal.}, 220(2):711--745, 2016.

\bibitem{duan2018solution}
R.~Duan and S.~Sakamoto.
\newblock Solution to the {B}oltzmann equation in velocity-weighted
  {C}hemin-{L}erner type spaces.
\newblock {\em Kinet. Relat. Models}, 11(6):1301--1331, 2018.

\bibitem{gressman2011global}
P.~T. Gressman and R.~M. Strain.
\newblock Global classical solutions of the {B}oltzmann equation without
  angular cut-off.
\newblock {\em J. Amer. Math. Soc.}, 24(3):771--847, 2011.

\bibitem{gulisashvili1996exact}
A.~Gulisashvili and M.~A. Kon.
\newblock Exact smoothing properties of {S}chr\"{o}dinger semigroups.
\newblock {\em Amer. J. Math.}, 118(6):1215--1248, 1996.

\bibitem{guo2003classical}
Y.~Guo.
\newblock Classical solutions to the {B}oltzmann equation for molecules with an
  angular cutoff.
\newblock {\em Arch. Ration. Mech. Anal.}, 169(4):305--353, 2003.

\bibitem{guo2003the}
Y.~Guo.
\newblock The {V}lasov-{M}axwell-{B}oltzmann system near {M}axwellians.
\newblock {\em Invent. Math.}, 153(3):593--630, 2003.

\bibitem{he2017well}
L.~He and J.~Jiang.
\newblock Well-posedness and scattering for the {B}oltzmann equations: soft
  potential with cut-off.
\newblock {\em J. Stat. Phys.}, 168(2):470--481, 2017.

\bibitem{he2023cauchy}
L.~He and J.~Jiang.
\newblock On the {C}auchy problem for the cutoff {B}oltzmann equation with
  small initial data.
\newblock {\em J. Stat. Phys.}, 190(3):Paper No. 52, 25, 2023.

\bibitem{herr2016gross}
S.~Herr and V.~Sohinger.
\newblock The {G}ross-{P}itaevskii hierarchy on general rectangular tori.
\newblock {\em Arch. Ration. Mech. Anal.}, 220(3):1119--1158, 2016.

\bibitem{herr2019unconditional}
S.~Herr and V.~Sohinger.
\newblock Unconditional uniqueness results for the nonlinear {S}chr\"{o}dinger
  equation.
\newblock {\em Commun. Contemp. Math.}, 21(7):1850058, 33, 2019.

\bibitem{mihaela2023local}
M.~Ifrim and D.~Tataru.
\newblock Local well-posedness for quasi-linear problems: a primer.
\newblock {\em Bull. Amer. Math. Soc. (N.S.)}, 60(2):167--194, 2023.

\bibitem{keel1998endpoint}
M.~Keel and T.~Tao.
\newblock Endpoint {S}trichartz estimates.
\newblock {\em Amer. J. Math.}, 120(5):955--980, 1998.

\bibitem{kenig1996bilinear}
C.~E. Kenig, G.~Ponce, and L.~Vega.
\newblock A bilinear estimate with applications to the {K}d{V} equation.
\newblock {\em J. Amer. Math. Soc.}, 9(2):573--603, 1996.

\bibitem{kenig1996quadratic}
C.~E. Kenig, G.~Ponce, and L.~Vega.
\newblock Quadratic forms for the {$1$}-{D} semilinear {S}chr\"{o}dinger
  equation.
\newblock {\em Trans. Amer. Math. Soc.}, 348(8):3323--3353, 1996.

\bibitem{kenig2001ill-posedness}
C.~E. Kenig, G.~Ponce, and L.~Vega.
\newblock On the ill-posedness of some canonical dispersive equations.
\newblock {\em Duke Math. J.}, 106(3):617--633, 2001.

\bibitem{kirkpatrick2011derivation}
K.~Kirkpatrick, B.~Schlein, and G.~Staffilani.
\newblock Derivation of the two-dimensional nonlinear {S}chr\"{o}dinger
  equation from many body quantum dynamics.
\newblock {\em Amer. J. Math.}, 133(1):91--130, 2011.

\bibitem{klainerman1993space}
S.~Klainerman and M.~Machedon.
\newblock Space-time estimates for null forms and the local existence theorem.
\newblock {\em Comm. Pure Appl. Math.}, 46(9):1221--1268, 1993.

\bibitem{klainerman2008uniqueness}
S.~Klainerman and M.~Machedon.
\newblock On the uniqueness of solutions to the {G}ross-{P}itaevskii hierarchy.
\newblock {\em Comm. Math. Phys.}, 279(1):169--185, 2008.

\bibitem{lu2001on}
X.~Lu and Y.~Zhang.
\newblock On nonnegativity of solutions of the {B}oltzmann equation.
\newblock {\em Transport Theory Statist. Phys.}, 30(7):641--657, 2001.

\bibitem{molinet2001ill-posedness}
L.~Molinet, J.~C. Saut, and N.~Tzvetkov.
\newblock Ill-posedness issues for the {B}enjamin-{O}no and related equations.
\newblock {\em SIAM J. Math. Anal.}, 33(4):982--988, 2001.

\bibitem{molinet2002well-posedness}
L.~Molinet, J.-C. Saut, and N.~Tzvetkov.
\newblock Well-posedness and ill-posedness results for the
  {K}adomtsev-{P}etviashvili-{I} equation.
\newblock {\em Duke Math. J.}, 115(2):353--384, 2002.

\bibitem{rauch1982nonlinear}
J.~Rauch and M.~Reed.
\newblock Nonlinear microlocal analysis of semilinear hyperbolic systems in one
  space dimension.
\newblock {\em Duke Math. J.}, 49(2):397--475, 1982.

\bibitem{saint2009hydrodynamcis}
L.~Saint-Raymond.
\newblock {\em Hydrodynamic limits of the {B}oltzmann equation}, volume 1971 of
  {\em Lecture Notes in Mathematics}.
\newblock Springer-Verlag, Berlin, 2009.

\bibitem{sohinger2015rigorous}
V.~Sohinger.
\newblock A rigorous derivation of the defocusing cubic nonlinear
  {S}chr\"{o}dinger equation on {$\Bbb{T}^3$} from the dynamics of many-body
  quantum systems.
\newblock {\em Ann. Inst. H. Poincar\'{e} Anal. Non Lin\'{e}aire},
  32(6):1337--1365, 2015.

\bibitem{sohinger2014boltzmann}
V.~Sohinger and R.~M. Strain.
\newblock The {B}oltzmann equation, {B}esov spaces, and optimal time decay
  rates in {$\Bbb{R}_x^n$}.
\newblock {\em Adv. Math.}, 261:274--332, 2014.

\bibitem{strain2008exponential}
R.~M. Strain and Y.~Guo.
\newblock Exponential decay for soft potentials near {M}axwellian.
\newblock {\em Arch. Ration. Mech. Anal.}, 187(2):287--339, 2008.

\bibitem{tao2006nonlinear}
T.~Tao.
\newblock {\em Nonlinear dispersive equations}, volume 106 of {\em CBMS
  Regional Conference Series in Mathematics}.
\newblock Conference Board of the Mathematical Sciences, Washington, DC; by the
  American Mathematical Society, Providence, RI, 2006.
\newblock Local and global analysis.

\bibitem{toscani1988global}
G.~Toscani.
\newblock Global solution of the initial value problem for the {B}oltzmann
  equation near a local {M}axwellian.
\newblock {\em Arch. Rational Mech. Anal.}, 102(3):231--241, 1988.

\bibitem{tzvetkov2006ill-posedness}
N.~Tzvetkov.
\newblock Ill-posedness issues for nonlinear dispersive equations.
\newblock In {\em Lectures on nonlinear dispersive equations}, volume~27 of
  {\em GAKUTO Internat. Ser. Math. Sci. Appl.}, pages 63--103. Gakk\={o}tosho,
  Tokyo, 2006.

\bibitem{villani2002review}
C.~Villani.
\newblock A review of mathematical topics in collisional kinetic theory.
\newblock In {\em Handbook of mathematical fluid dynamics, {V}ol. {I}}, pages
  71--305. North-Holland, Amsterdam, 2002.

\end{thebibliography}

\end{document}